\documentclass[11pt, a4paper]{article}

\usepackage{etex}

\usepackage{amsmath,dsfont}
\usepackage{amsfonts}
\usepackage{amssymb}
\usepackage{amsthm}
\usepackage{authblk}
\usepackage{fullpage}
\usepackage[
	ocgcolorlinks,
	linkcolor=linkblue,
	citecolor=linkred,
	urlcolor=linkred]
{hyperref}
\usepackage{verbatim}
\usepackage{xcolor}
\usepackage{graphicx}
\usepackage[inline]{enumitem}

\numberwithin{equation}{section}

\definecolor{linkred}{rgb}{0.75,0,0}
\definecolor{linkblue}{rgb}{0,0,0.75}

\theoremstyle{plain}
\newtheorem{maintheorem}{Theorem}

\theoremstyle{plain}
\newtheorem{theorem}{Theorem}[section]

\newtheorem{lemma}[theorem]{Lemma}
\newtheorem{proposition}[theorem]{Proposition}
\newtheorem{corollary}[theorem]{Corollary}

\theoremstyle{definition}

\newtheorem{rem}[theorem]{Remark}

\newcommand{\confs}{\mathcal{X}}
\newcommand{\Weylchamber}{\tilde{\mathcal{X}}}

\newcommand{\intZ}{\mathbb{Z}}
\newcommand{\realR}{\mathbb{R}}
\newcommand{\compC}{\mathbb{C}}
\newcommand{\prob}{\mathbb{P}}
\newcommand{\bigO}{\mathcal{O}}
\newcommand{\s}{\sigma}

\newcommand{\Lattice}{\mathbb{L}}
\newcommand{\id}{\mathds{1}}
\renewcommand{\sqcup}{\cup}

\DeclareMathOperator{\sgn}{sgn}

\def\Xint#1{\mathchoice
{\XXint\displaystyle\textstyle{#1}}%
{\XXint\textstyle\scriptstyle{#1}}%
{\XXint\scriptstyle\scriptscriptstyle{#1}}%
{\XXint\scriptscriptstyle\scriptscriptstyle{#1}}%
\!\int}
\def\XXint#1#2#3{{\setbox0=\hbox{$#1{#2#3}{\int}$}
\vcenter{\hbox{$#2#3$}}\kern-.5\wd0}}

\def\dashint{\Xint-}

\newcommand{\prodprime}{\sideset{}{'}\prod}
\newcommand{\qbinom}[2]{\genfrac{[}{]}{0pt}{}{#1}{#2}_q}
\newcommand{\taubinom}[2]{\genfrac{[}{]}{0pt}{}{#1}{#2}_{\tau}}

\begin{document}

\title{Integral formulas of ASEP and $q$-TAZRP on a ring}

\author{Zhipeng Liu}
\affil{Department of Mathematics,
	   University of Kansas, Lawrence, KS 66045, U.S.A.
  \texttt{zhipeng@ku.edu}}

\author{Axel Saenz}
\affil{Mathematics Department,
  Tulane University,
  New Orleans, LA 70118, U.S.A.
  \texttt{saenzaxel@gmail.com}}

\author{Dong Wang}
\affil{Department of Mathematics,
  National University of Singapore,
  119076, Singapore
  \texttt{matwd@nus.edu.sg}}

\date{\today}

\maketitle

\begin{abstract}
  In this paper, we obtain the transition probability formulas for the Asymmetric Simple Exclusion Process (ASEP) and the $q$-deformed Totally Asymmetric Zero Range Process ($q$-TAZRP) on the ring by applying the coordinate Bethe ansatz. We also compute the distribution function for a tagged particle with general initial condition.
\end{abstract}

\section{Introduction}

We investigate the transition probabilities of $1$-dimensional interacting particles systems with spatial periodicity (i.e.~on a ring), by using the coordinate Bethe ansatz. We concentrate on two models: the Asymmetric Simple Exclusion Process (ASEP) and the $q$-deformed Totally Asymmetric Zero Range Process ($q$-TAZRP). We explicitly give the transition probability function and the one-point function for the ASEP and the $q$-TAZRP on a ring as a sum of nested contour integrals. 

The ASEP and the $q$-TAZRP are continuous time Markov processes on a discrete one-dimensional lattice. In particular, we assume that the models are defined on a one dimensional periodic lattice such as $\intZ/L\intZ$. The state of the Markov process is determined by the (random) location of $N$ particles. For the case of the infinite lattice $\intZ$, the ASEP and the $q$-TAZRP have been well-studied resulting in exact transition probability formulas and asymptotic computations \cite{Johansson00,Tracy-Widom08, Tracy-Widom09, Korhonen-Lee14, Lee-Wang17, Borodin-Corwin-Petrov-Sasamoto15a}. These exact transition probability formulas are the $L \to \infty$ limit of our formulas for the periodic lattice $\intZ/L\intZ$. (see Sections \ref{sec:limits} and \ref{sec:one_point_limits}).

In the ASEP, each site may be occupied by at most one particle and particles move independently and asymmetrically to the left or the right, unless they are blocked by a neighboring particle. In the $q$-TAZRP, more than one particle may occupy a site and only the top particle at a site may move to the right, regardless of the neighboring particles. For the precise definition of the models, see Section \ref{subsec:notation}.

The coordinate Bethe ansatz was introduced in \cite{Bethe31} in the context of spin chains (see \cite{Sutherland04, Gaudin14} for a modern review). In the work of Spohn and Gwa \cite{Gwa-Spohn92}, its application to interacting particle systems was observed, and Sch\"utz \cite{Schutz97} applied it to the Totally Asymmetric Simple Exclusion Process (TASEP). \footnote{When we say the coordinate Bethe ansatz is applied to the TASEP and other models, we mean that certain eigenstates for the linear evolution operator of the process may be solved exactly through a ``guess and check" method, with the ``check" part equivalent to the verification of a type of Yang-Baxter equation (i.e.~$k$-particle interactions are equivalent to a sequence of $2$-particle interactions, independent of the order of the sequence). In all these models on $\intZ$, there is no need of solving the Bethe equation, since the boundary condition is trivial.} For one dimensional integrable interacting particle systems on the lattice $\intZ$, the coordinate Bethe ansatz has been applied to wide class of models, most notably the ASEP \cite{Tracy-Widom08} and the $q$-TAZRP \cite{Korhonen-Lee14, Wang-Waugh16}; see also \cite{Borodin-Corwin-Petrov-Sasamoto15a}. For one dimensional periodic integrable interacting particle systems on the lattice $\intZ / L \intZ$, the coordinate Bethe ansatz has only been similarly applied to TASEP \footnote{For TASEP on $\intZ/L\intZ$, the transition probability function is found in the similar method and similar form as the TASEP on $\intZ$. However, the Bethe equation was not solved in the usual sense. This feature is shared in our paper.}, a determinantal model: \cite{Povolotsky-Priezzhev07, Prolhac16, Baik-Liu16, Baik-Liu18, Baik-Liu18a} with all the particles jumping to the right on the ring. Additionally, the Bethe ansatz is useful in computing other observables (e.g.~spectral gap or large deviation functions) for the periodic (T)ASEP, see \cite{Kim95, Schutz95, Derrida-Lebowitz98, Golinelli-Mallick04, Povolotsky04, Golinelli-Mallick05,Prolhac13,Prolhac14,Prolhac15a,Prolhac15,Prolhac16a, Prolhac17,Mallick-Prolhac18,Prolhac20}. Recently, the work of \cite{Feher-Pozsgay18} considers the spin-$\frac{1}{2}$ XXZ chain on the ring, which is related to the ASEP by a non-unitary transformation. The authors of \cite{Feher-Pozsgay18} present a contour integral formula, but only give a precise proof for the two particle case, for the propagator of the spin chain, which corresponds to and is very similar to the transition probability function that we consider in this work.

The strategy to obtain probability formulas through the coordinate Bethe ansatz follows two fundamental steps: \begin{enumerate*}[label=(\arabic*)]
\item
   find eigenvectors for the Markov operator defining the evolution of the process via the coordinate Bethe ansatz, and
\item 
 write a linear transformation from the given eigenvectors to the elementary basis vectors.
\end{enumerate*} In this description, we represent a probability distribution on the configuration space as non-negative vectors with components that add up to one. It follows that an elementary basis vector corresponds to a deterministic distribution. The eigenvectors obtained by the coordinate Bethe ansatz are called \emph{Bethe eigenfunctions}. In general, writing an elementary basis vectors as linear combination of Bethe eigenfunctions is an open problem as it is not clear by the construction of the Bethe ansatz that the Bethe eigenfunctions give a complete basis, and this is beyond the scope of this paper. 

For models on $\mathbb{Z}$ with $N$ particles, the Bethe eigenfunctions are given as sum over the corresponding Weyl group $A_N \cong S_N$ (i.e.~the symmetric group on $N$ elements). In \cite{Tracy-Widom08}, Tracy and Widom  established the linear transformation from the Bethe eigenfunction to the basis vectors as a nested sequence of contour integrals. This linear transformation was crucial in the asymptotic analysis for the ASEP on $\mathbb{Z}$ and it opened the flood-gates to similar asymptotic analysis for non-determinantal models on $\mathbb{Z}$ \cite{Dotsenko10, Calabrese-Le_Doussal-Rosso10, Amir-Corwin-Quastel11, Sasamoto-Spohn10, Borodin-Corwin13, Borodin-Corwin-Sasamoto14}.

For the models on a ring, we present new non-trivial probability formulas through the proper reinterpretation of the Bethe ansatz. We introduce formal Bethe eigenfunctions as a sum over the affine Weyl group $\tilde{A}_{N - 1} \cong S_N \ltimes \intZ^{N - 1}$; a priori these eigenfunctions are not necessarily convergent. Moreover, we establish the linear transformation from the formal Bethe eigenfunctions to the elementary basis vectors as sequence of nested contour integrals. Below, we work directly with the linear transformation of the formal Bethe eigenfunctions and show convergence in that setting. In particular, we don't construct basis vectors that diagonalize the Markov generator for these processes. Instead, through methods inspired by the Bethe ansatz, we directly obtain explicit transition probability formulas.

The main hurdle for the periodic model lies in controlling the spectrum of the Markov operator defining the evolution of the processes. For instance, the spectrum for the ASEP on the ring is discrete whereas the spectrum for the ASEP on the line is continuous. Moreover, the spectrum of the ASEP on the ring depends on the solutions to the Bethe equations, a consistency condition for the periodicity of the process given by a system of algebraic equations with rank proportional to the number of particles and degree proportional to the period of the ring. Thus, exact formulas for the transition probability of the ASEP or the $q$-TAZRP on the ring would require precise information of the Bethe equations; a non-trivial problem in algebraic geometry. We are able to overcome the constraint of the Bethe equations by introducing an additional parameter to the models in the ring: the winding number. As a result, the state space becomes infinite again, as in the models on $\mathbb{Z}$; the spectrum is no longer discrete; and there is an additional parameter in the Bethe equations that may be treated as a deformation parameter to precisely describe the solutions of the Bethe equations.

In the present work, we would like to note the following two contributions:
\begin{enumerate*}[label=(\arabic*)]
\item
  we give a formula for the transition probability function for the models mentioned above that is constructive (see Theorem \ref{thm:main} and Theorem \ref{thm:trans_prob}), and 
\item 
we provide a clear and precise algebreo-geometric structure for non-determinantal periodic interacting particle systems.  
\end{enumerate*}
We construct the transition probability function through the coordinate Bethe ansatz, similar to the construction of the transition probability function for similar models such as the ASEP \cite{Tracy-Widom08} and the $q$-TAZRP \cite{Korhonen-Lee14, Wang-Waugh16} on the integer lattice $\mathbb{Z}$ and the TASEP on the periodic lattice $\mathbb{Z}/L \mathbb{Z}$ \cite{Baik-Liu18, Baik-Liu18a}. In this work, we are able to carry out the coordinate Bethe ansatz through certain algebraic identities (related to the eigenfunctions from the coordinate Bethe ansatz) and through a careful control of the Bethe equations. For models on an infinite lattice, such as the ASEP \cite{Tracy-Widom08} and $q$-TAZRP \cite{Wang-Waugh16} on $\mathbb{Z}$, the Bethe equations become trivial and may be avoided in the coordinate Bethe ansatz method. In the periodic TASEP \cite{Baik-Liu18, Baik-Liu18a}, the Bethe equations decouple and many determinantal formulas apply making the coordinate Bethe ansatz methods straight forward. In this work, we adapt the coordinate Bethe ansatz to a pair of non-determinantal models, readjust the Bethe ansatz to handle the non-trivial Bethe equations via nested contour integrals, and implement the Bethe ansatz to obtain exact transition probability functions.

In the rest of the introduction, we briefly describe the periodic ASEP and the periodic $q$-TAZRP, and we state the main result for each model. Then, at the end of the introduction, we give an outline for the rest of the paper.

\subsection{Description of models} \label{subsec:notation}

For both the ASEP and the $q$-TAZRP, we consider the processes on a ring
\begin{equation}
 \intZ/L\intZ = \{ [1], [2], \dotsc, [L] \},  
\end{equation}
such that the right neighboring site of $[k]$ is $[k + 1]$ and $[L + 1]$ is identified with $[1]$. In fact, by keeping track of the winding number, we actually realize the models on $\intZ$, that is, if a particle moves a whole period monotonically from $[1], [2], [3], \dotsc, [L]$ to $[1]$, we realize it as moving from $kL + 1, kL + 2, \dotsc, (k + 1)L$ to $(k + 1)L + 1$ on $\intZ$ with $k$ any integer. The drawback of lifting up the model to $\intZ$ is that the cyclic symmetry is broken. We discuss this in Section \ref{sec:cyclic_invariance}.

\begin{rem}
  When we realize the periodic model on $\intZ$, we may order the particles two possible ways:
  \begin{enumerate*}[label=(\arabic*)]
  \item
    $x_i \leq x_{i+1}$, or 
  \item
    $x_i \geq x_{i+1}$
  \end{enumerate*}
  (with equality only possible for the $q$-TAZRP).  In the following, we order the particles differently depending on the model in order to align with the convention of the established literature. That is, the notation we use for the periodic ASEP, $x_i <x_{i+1}$, is consistent with the notation for the ASEP on $\intZ$ and the periodic TASEP on $\intZ/ L \intZ$, and the notation we use for  the periodic $q$-TAZRP, $x_i \geq x_{i+1}$, is consistent with the notation for the $q$-TAZRP on $\intZ$. 
\end{rem}

\subsubsection{ASEP on a ring}

The ASEP with $N$ particles moving on a discrete ring of length $L$ is a continuous Markov process with the finite state space 
\begin{equation}
\{ ([x_1], \dots, [x_N] )  \in (\mathbb{Z}/ L \mathbb{Z})^N \mid \text{$[x_1], \dotsc, [x_N]$ has strict cyclic increasing order} \}. 
\end{equation}
We lift the particle configurations and dynamics to the integer lattice by tracking the winding number. For consistency with the notation from the literature on the ASEP on a line \cite{Tracy-Widom08}, we order the particles as
\begin{equation} \label{eq:ordered_pt_ASEP}
  x_1(t) < x_2(t) < \dotsb < x_N(t),
\end{equation}
such that $x_i(t)$ is the position of the $i$-th particle at time $t$. Also, the periodic property of a ring that $x_1$ and $x_N$  block each other is realized by the additional requirement
\begin{equation} \label{eq:ASEP_periodic_cond}
  x_N(t) < x_1(t) + L.
\end{equation}
Then, we denote the configuration space for the ASEP on $\intZ/L\intZ$ as follows:
\begin{equation} \label{eq:ASEP_Weyl_chanmber}
  \confs_N(L) := \{ X = (x_1, \dotsc, x_N) \in \intZ^N \mid x_1 < \cdots < x_N < x_1 + L \}.
\end{equation}
Here, $N$ and $L$ denote the number of particles and the length of the ring, respectively. Note that this configuration space is larger than the configuration space of particles on a ring as it also carries the information of the winding number of the process. 

The evolution of the ASEP is governed by independent exponential clocks of rate $1$, which update the location of the particles, for each particle. When the exponential clock for a particle is activated, the particle will change locations to the right (resp.~left) neighboring site with probability $p$ (resp.~$q = 1 - p$) only if the new location is not occupied by another particle, and the exponential clock for that particle is reset. Otherwise, if the new location is occupied, the particle will not change locations, and the exponential clock is also reset so that the particle attempts to change location once the exponential clock is activated again. In the following, we assume that $p, q \in (0, 1)$, and denote $\tau = q/p$. See Section \ref{subsec:ASEP_master} for more specific details on the dynamics of the ASEP.

\subsubsection{$q$-TAZRP on a ring} \label{subsubsec:q_TAZRP_defn}

The $q$-TAZRP with $N$ particles moving on a discrete ring of length $L$ is a continuous Markov process with the finite state space 
\begin{equation}
\{ ([x_1], \dots, [x_N] )  \in (\mathbb{Z}/ L \mathbb{Z})^N \mid \text{$[x_1], \dotsc, [x_N]$ has a weak cyclic decreasing order} \}.
\end{equation}
In this model, we allow for a spatial inhomogeneity by introducing parameters called conductance on each site. We denote the conductance of site $[i]$ on the ring as $a_{[i]} \in \mathbb{R}_{\geq 0}$. Similar to the ASEP model, we lift the particle configurations and dynamics to the integer lattice keeping track of the winding number of the particles. Accordingly, we take the conductances $a_i$ at site $i \in \mathbb{Z}$ such that $a_{kL + i} = a_{[i]}$ for all $i = 1, 2, \dotsc, L$ and $k \in \intZ$. For consistency with the notation from the literature on the $q$-TAZRP on a line \cite{Korhonen-Lee14, Wang-Waugh16}, we order the particles as
\begin{equation}
  x_1(t) \geq x_2(t) \geq \dotsb \geq x_N(t),
\end{equation}
which resembles the order in \eqref{eq:ordered_pt_ASEP} but with the opposite order. Also, similar to \eqref{eq:ASEP_periodic_cond}, we have the periodic condition
\begin{equation}
  x_1(t) \leq x_N(t) + L.
\end{equation}
Then, analogous to \eqref{eq:ASEP_Weyl_chanmber}, we denote the configuration space
\begin{equation} \label{eq:concrete_X_in_Weylchamber}
  \Weylchamber_N(L) := \{ (x_1, \dotsc, x_N) \in \intZ^N \mid x_N + L \geq x_1 \geq x_2 \geq \dotsb \geq x_N \}.
\end{equation}
Here, $N$ and $L$ denote the number of particles and the length of the ring, respectively. Note that this configuration space is larger than the configuration space of particles on a ring as it also carries the information of the winding number of the process. 

When particles occupy the same site, they have to be labeled consecutively according to the cyclic order of $\intZ/N\intZ$. For instance, take $x_{j + 1}, x_{j + 2}, \dots, x_{j + k}$ with $j + k \leq N$ so that they have identical position $x$, or alternatively, take $x_{j + 1}, x_{j + 2}, \dots, x_{N}, x_1, x_2, \dots, x_{j + k - N}$ with $j + k > N$ so that $x_{j + 1}, \dotsc, x_N$ have position $x$ and $x_1, \dotsc, x_{j + k - N}$ have position $x + L$. Particles stacked on the same site obey the vertical order where $x_j$ is atop of $x_{j + 1}$ for $j = 1, \dotsc, N - 1$ and $x_N$ is atop of $x_1$ if $x_1 = x_N + L$. From the dynamics described at the beginning of the paper, this property remains as the system evolves. 

For $N$ particles in $q$-TAZRP on the ring $\intZ/L\intZ$, we denote $n_{[i]}$ as the number of particles on site $[i]$, and equivalently, for any state in $\Weylchamber_N(L)$, we denote
\begin{equation} \label{eq:defn_n_[i](X)}
  n_{[i]}(X) = \text{\# of particles in the congruence class $\{ kL + i \mid k \in \intZ \}$},
\end{equation}
for any $i \in \intZ$. Although $n_{[i]}$ is defined as an infinite sum, at most two terms in the defining sum are nonzero making the summation well-defined. Moreover, we have $n_{[i + kL]}(X) = n_{[i]}(X)$ and $n_{[1]}(X) + \dotsb + n_{[L]}(X) = N$.

The evolution of the $q$-TAZRP is governed by independent exponential clocks for each particle with the rate of an exponential clock associated to a specific particle depending on the position of the particle and the vertical position of the particle in the stack. If a particle is at site $[i]$ with $k$ particles stacked, the rate of the corresponding exponential clock is $0$ unless the particle is at the top. On the other hand, if a particle is at the top of the stack at site $[i]$, the rate of the corresponding exponential clock is $a_{[i]}(1 - q^k)$. We always assume that $q \in (0, 1)$. Hence, if a particle occupies site $[i]$ by itself, the corresponding exponential clock has rate $a_{[i]}(1 - q)$, which we later denote as $b_{[i]}$. When the exponential clock for a particle is activated, the particle changes location to the right neighboring site and becomes the bottom particle in the new stack. Note that if a particle is in a stack but not the top one, it doesn't change locations.

\subsection{Statement of results}

We denote $\Lattice = (\ell_1, \dotsc, \ell_N) \in \intZ^N$, and write $\Lattice \in \intZ^N(k)$ if and only if $\ell_1 + \dotsb + \ell_N = k$.

\paragraph{ASEP on a ring}

We introduce some notation to state the main result for the ASEP on a ring. Let $\xi_1, \dotsc, \xi_N \in \mathbb{C}$ be some undetermined complex variables. For each pair of integers $\alpha, \beta \in \{ 1, \dots, N\}$, we denote
\begin{equation}\label{eq:S_factor}
  S_{\beta, \alpha} = -\frac{p + q \xi_{\beta} \xi_{\alpha} - \xi_{\beta}}{p + q \xi_{\beta} \xi_{\alpha} - \xi_{\alpha}}.
\end{equation}
For a permutation $\sigma \in S_N$, we denote
\begin{equation} \label{eq:expr_A_sigma_ASEP}
  A_{\sigma}(\xi_1, \dotsc, \xi_N) = \prod_{\substack{(\beta, \alpha) \text{ is an} \\ \text{inversion of } \sigma}}S_{\beta, \alpha} = \sgn(\sigma) \prod_{1 \leq i < j \leq N} \frac{p + q \xi_{\sigma(i)} \xi_{\sigma(j)} - \xi_{\sigma(i)}}{p + q \xi_i \xi_j - \xi_i},
\end{equation}
with $(\beta, \alpha)$ an \emph{inversion} of $\sigma$ if $\alpha$ and $\beta$ are distinct numbers in $\{ 1, \dotsc, N \}$ so that $\beta > \alpha$ and $\sigma^{-1}(\beta) < \sigma^{-1}(\alpha)$. We write $A_{\sigma} = A_{\sigma}(\xi_1, \dotsc, \xi_N)$ if there is no possibility of confusion.

Denote the transition probability function by $\prob_Y(X;t)$, which gives the probability that the process is in state $X=(x_1,\dotsc,x_N)\in \confs_N(L)$ at time $t$ given the initial configuration $Y=(y_1,\dotsc,y_N)\in\confs_N(L)$ at time $t=0$. For $X = (x_1, \dotsc, x_N) \in \intZ^N$, $Y = (y_1, \dotsc, y_N) \in \intZ^N$, $\Lattice \in \intZ^N(0)$ and $\sigma \in S_N$, we set
\begin{equation} \label{eq:int_Lambda_ASEP}
  \Lambda^{\Lattice}_Y(X; t; \sigma) = \dashint_C d\xi_1 \dotsi \dashint_C d\xi_N A_{\sigma} \prod^N_{j = 1} \left[ \xi_{\sigma(j)}^{x_j - y_{\sigma(j)} - 1} e^{\epsilon(\xi_j)t} \right] D_{\Lattice}(\xi_1,\cdots,\xi_N),
\end{equation}
using the notation $\dashint_C dw_i$ as a shorthand for $(2\pi i)^{-1} \oint_C dw_i$ and $C = \{ \lvert z \rvert = r \}$ a counterclockwise contour with $r > 0$ small enough so that the integrand in \eqref{eq:int_Lambda_ASEP} is analytic in the domain $\{0 < \lvert \xi_i \rvert \leq r,\ i = 1, \dotsc, N \}$, with
\begin{equation} \label{eq:HK}
  D_{\Lattice}(\xi_1,\cdots,\xi_N) := \prod^N_{j = 1} \left( \xi^L_j \prod^N_{k = 1} \left( \frac{p + q\xi_k \xi_j - \xi_k}{p + q\xi_k \xi_j - \xi_j} \right) \right)^{\ell_j}, \quad \epsilon(\xi) = p \xi^{-1} + q \xi - 1.
\end{equation}

\begin{maintheorem} \label{thm:main} 
  Let $X, Y \in \confs_N(L)$. The transition probability function for the ASEP is
\begin{equation}
\prob_Y(X;t)= u_Y(X; t),
\end{equation}
with
  \begin{equation} \label{eq:transition_probability}
    u_Y(X; t) = \sum_{\Lattice \in \intZ^N(0)} u^{\Lattice}_Y(X; t), \quad \text{and} \quad u^{\Lattice}_Y(X; t) = \sum_{\sigma\in S_N} \Lambda^{\Lattice}_Y(X; t; \sigma).
  \end{equation}
\end{maintheorem}
Lemma \ref{lm:absolutely_convergent} implies that for any $\sigma \in S_N$, the sum $\sum_{\Lattice \in \intZ^N(0)} \Lambda^{\Lattice}_Y(X; t; \sigma)$ is absolutely convergent. Hence $u_Y(X; t)$ is well defined.

\begin{rem}
  The transition probability formula~\eqref{eq:transition_probability} is a generalization of the transition probability formulas for the ASEP on the line and the TASEP on the ring. The transition probability formula for the ASEP on the line given in \cite[Theorem 2.1]{Tracy-Widom08} is our $u^{(0, \dotsc, 0)}_Y(X; t)$. In the limit $L \rightarrow \infty$, the other terms vanish, and we recover the transition probability formula for the ASEP on the line given in \cite[Theorem 2.1]{Tracy-Widom08}. Setting $p=1$, Theorem \ref{thm:main} degenerates into the transition probability formula for TASEP on the ring given in \cite{Baik-Liu18}. We give more details in section~\ref{sec:limits}.
\end{rem}

From the transition probability formula in Theorem \ref{thm:main}, we can derive the marginal distribution of a tagged particle given any initial condition.
\begin{maintheorem} \label{thm:one_pt_ASEP}
  Let $Y \in \confs_N(L)$ be the initial state of the ASEP on the ring. Then, the distribution of $x_m$ ($m = 1, \dotsc, N$) is
  \begin{multline}
  \label{eq:one_point_distribution}
    \prob_Y(x_m(t) \geq M) = (-1)^{(m - 1)(N - 1)} \frac{p^{N(N - 1)/2}}{2\pi i} \oint_0 \frac{dz}{z^m} C_N(z) \sum^{N - m}_{k = 1 - m} z^{-k} \sum_{\Lattice \in \intZ^N(k)} \\
    \dashint_C \frac{d\xi_1}{1 - \xi_1} \dotsi \dashint_C \frac{d\xi_N}{1 - \xi_N} \prod^N_{j = 1} \xi^{M - y_j - 1}_j e^{\epsilon(\xi_j)t} \prod_{1 \leq i < j \leq N} \frac{\xi_j - \xi_i}{p + q\xi_i\xi_j - \xi_i} D_{\Lattice}(\xi_1, \dotsc, \xi_N),
  \end{multline}
 with the same contour $C$ as the integrals of \eqref{eq:int_Lambda_ASEP}, the terms $\epsilon(\xi_j)$ and $D_{\Lattice}(\xi_1, \dotsc, \xi_N)$ defined by \eqref{eq:HK}, and
  \begin{equation} \label{eq:defn_C_N_ASEP}
    C_N(z) = \prod^{N - 1}_{j = 1} \left( 1 + (-1)^N \tau^j z \right).
  \end{equation}
\end{maintheorem}
The summation over $\Lattice \in \intZ^N(k)$ is discussed further in Lemma \ref{lm:absolutely_convergent}.

\begin{rem}
	The one-point probability formula~\eqref{eq:one_point_distribution} is also a generalization of the one-point probability formulas for the ASEP on the line and the TASEP on the ring. The corresponding results may be found in \cite{Tracy-Widom08} and \cite{Baik-Liu18}, respectively. We give more details in Section~\ref{sec:one_point_limits}.
\end{rem}

\paragraph{$q$-TAZRP on a ring}

We introduce some notation to state the main theorem for the $q$-TAZRP on a ring. For $X \in \Weylchamber_N(L)$, denote
\begin{equation} \label{eq:defn_W(X)}
  W(X) = \prod^L_{i = 1} [n_{[i]}(X)]_q!, \quad \text{with} \quad [k]_q! = (1 - q)(1 - q^2) \dotsm (1 - q^k)/(1 - q)^k.
\end{equation}
Let $w_1, \dotsc, w_n \in \mathbb{C}$ be some undetermined complex variables. For each pair of integers $\alpha , \beta  \in \{ 1, \dots, N\}$, we denote
\begin{equation}
  S_{\beta, \alpha} = -\frac{qw_{\beta} - w_{\alpha}}{qw_{\alpha} - w_{\beta}}.
\end{equation}
For a permutation $\sigma \in S_N$, define
\begin{equation} \label{eq:expr_A_sigma_qTAZRP}
  A_{\sigma}(w_1, \dotsc, w_N) = \prod_{\substack{(\beta, \alpha) \text{ is an} \\ \text{inversion of } \sigma}} S_{\beta, \alpha} = \sgn(\sigma) \prod_{1 \leq i < j \leq N} \frac{qw_{\sigma(i)} - w_{\sigma(j)}}{qw_i - w_j}.
\end{equation}
We write $A_{\sigma} = A_{\sigma}(w_1, \dotsc, w_N)$ if there is no possibility of confusion. Additionally, for $X = (x_1, \dotsc, x_N)$, $Y = (y_1, \dotsc, y_N) \in \intZ^N$, $\Lattice \in \intZ^N(0)$ and $\sigma \in S_N$, we set
\begin{multline} \label{eq:int_Lambda}
  \Lambda^{\Lattice}_Y(X; t; \sigma) = \\
  \left( \prod^N_{k = 1} \frac{-1}{b_{[x_k]}} \right) \dashint_C dw_1 \dotsi \dashint_C dw_N A_{\sigma} \prod^N_{j = 1} \left[ \prodprime^{x_j}_{k = y_{\sigma(j)}} \left( \frac{b_{[k]}}{b_{[k]} - w_{\sigma(j)}} \right) e^{-w_j t} \right] D_{\Lattice}(w_1, \dotsc, w_N),
\end{multline}
with $C$ a positive-oriented circle $\{ \lvert z \rvert = R \}$ for $R$ a large enough positive constant (i.e.~$R > b_{[i]}$ for all $b_{[i]}$),
\begin{equation} \label{eq:general_D_L_qTAZRP}
  D_{\Lattice}(w_1, \dotsc, w_N) = \prod^N_{j = 1} \left( \prod^L_{i = 1} (b_{[i]} - w_j) \prod^N_{k = 1} \left( \frac{qw_k - w_j}{qw_j - w_k} \right) \right)^{\ell_j},
\end{equation}
and $\prod'$ an extension of the usual $\prod$ notation so that
\begin{equation}
  \prodprime^n_{k = m} f(k) =
  \begin{cases}
    \prod^n_{k = m} f(k) & \text{if $n \geq m$}, \\
    1 & \text{if $n = m - 1$}, \\
    \prod^{m - 1}_{k = n + 1} \frac{1}{f(k)} & \text{if $n \leq m - 1$}.
  \end{cases}
\end{equation}
For example, $\prod^{' -3}_{k = 0} f(k) = 1/[f(-1) f(-2)]$. 

\begin{rem}
  The notation for $S_{\beta, \alpha}$, $A_{\sigma}$, $C$, $D_{\Lattice}$ and some other terms defined later have different meanings for the ASEP and the $q$-TAZRP. We prefer to use the same notation for analogous quantities since they should be easily distinguished by the context and hardly confuse the reader.
\end{rem}
Denote the transition probability function $\prob_Y(X;t)$, which gives the probability that the process is in state $X=(x_1,\dotsc,x_N)\in \Weylchamber_N(L)$ at time $t\geq 0$ given an initial configuration $Y=(y_1,\dotsc,y_N)\in\Weylchamber(L)$ at time $t=0$. 
\begin{maintheorem} \label{thm:trans_prob}
  Let $X = (x_1, \dotsc, x_N), Y = (y_1, \dotsc, y_N) \in \Weylchamber_n(L)$. The transition probability function for $q$-TAZRP is 
\begin{equation}
\prob_Y(X; t) = u_Y(X; t) / W(X),
\end{equation}
with
  \begin{equation} \label{eq:defn_u_Y(X)}
    u_Y(X; t) = \sum_{\Lattice \in \intZ^N(0)} u^{\Lattice}_Y(X; t), \quad \text{and} \quad u^{\Lattice}_Y(X; t) = \sum_{\sigma \in S_N} \Lambda^{\Lattice}_Y(X; t; \sigma).
  \end{equation}
\end{maintheorem}
By Lemmas \ref{lem:finiteness_of_large_max(L)} and \ref{lem:qTAZRP_conv}, only finitely many terms $u^{\Lattice}_Y(X; t)$ are nonzero for the sum over $\Lattice \in \intZ^N(0)$. Thus, the function $u_Y(X; t)$ in \eqref{eq:defn_u_Y(X)} is well-defined.

Similar to Theorem \ref{thm:one_pt_ASEP}, we have the marginal distribution for a tagged particle given any initial condition, assuming the technical condition that all conductances are identical.
\begin{maintheorem} \label{thm:one_pt_qTAZRP}
  Take the conductances so that $a_{[1]} = a_{[2]} = \dotsb = a_{[N]} = (1 - q)$. Let $Y \in \Weylchamber_N(L)$ be the initial state of the $q$-TAZRP on the ring. Then, the distribution of $x_{N - m + 1}$ is
  \begin{multline} \label{eq:one_point_distribution_2}
    \prob_Y(x_{N - m + 1}(t) > M) = \frac{1}{2\pi i} \oint_0 \frac{dz}{z^m} C_N(z) \sum^{N - m}_{k = 1 - m} z^{-k} \sum_{\Lattice \in \intZ^N(-k)} \\
    \dashint_C \frac{dw_1}{w_1} \dotsi \dashint_C \frac{dw_N}{w_N} \prod^N_{j = 1} (1 - w_j)^{y_j - M} e^{-w_j t} \prod_{1 \leq i < j \leq N} \frac{w_i - w_j}{qw_i - w_j} D_{\Lattice}(w_1, \dotsc, w_N),
  \end{multline}
  with
  \begin{equation} \label{eq:C_N_qTAZRP}
    C_N(z) = \prod^{N - 1}_{j = 1} (1 + (-1)^N q^j z).
  \end{equation}
\end{maintheorem}
By Lemmas \ref{lem:finiteness_of_large_max(L)} and \ref{lem:qTAZRP_conv}, only finitely many terms in the integral \eqref{eq:one_point_distribution_2} are nonzero for the sum over $\intZ^N(-k)$ making \eqref{eq:one_point_distribution_2} well-defined.

\subsection{Outline}

We prove Theorem~\ref{thm:main} (resp.\ Theorem \ref{thm:trans_prob}) by showing that~\eqref{eq:transition_probability} (resp.\ \eqref{eq:defn_u_Y(X)}) satisfies the (Kolmogorov forward) master equation for the Markov process at hand. More specifically, we take a coordinate Bethe anstaz approach by decomposing the master equation into a non-interacting equation with specific boundary conditions (on a Weyl alcove) that encode the interactions. This way, showing~\eqref{eq:transition_probability} (resp.\ \eqref{eq:defn_u_Y(X)}) satisfies the master equation is much simpler and the more technical point of the proof is showing that formulas indeed satisfy the initial conditions (i.e. $\mathbb{P}_Y(X;0) = \delta_{Y}(X)$). The sensibility of the contour integral formulas is choosing the contours so that the contour integrals converge without picking any stray residues. 

In Section~\ref{sec:asep}, we write down the master equations for the transition probabilities of the ASEP and the $q$-TAZRP on a ring. In both cases, we characterize the transition probabilities $\prob_Y(X; t)$ by a free equation together with boundary conditions, in the spirit of the coordinate Bethe ansatz. In Section \ref{sec:cyclic_invariance}, we introduce identities for both models that are due to the cyclic symmetry of the ring. These identities will simplify many computations throughout the rest of the arguments. Then, Theorems \ref{thm:main} and \ref{thm:trans_prob}, regarding the transition probabilities for both models, are proved in Section \ref{sec:proof}, and Theorems \ref{thm:one_pt_ASEP} and \ref{thm:one_pt_qTAZRP}, regarding the one-point functions for both models, are proved in Section \ref{sec:one_pt}. Additionally, in Sections \ref{sec:limits} and \ref{sec:one_point_limits}, we show that the results from Theorems \ref{thm:main} and \ref{thm:one_pt_ASEP}, regarding the ASEP model, agree with previously known results in \cite{Tracy-Widom08,Baik-Liu18}. Lastly, in Section \ref{sec:root_formula}, we perform some residue computations for the transition formula for the ASEP in Theorem \ref{thm:main}, and we obtain a formula based on the roots of a system of algebraic equations called the \emph{Bethe equations} (see Remark \ref{r:completeness}). At the end, in Appendix \ref{sec:convergence}, we include some bounds on the type of integrands that we use in the contour integral formulas, and these are needed to justify some of the manipulations we do with the contour integral formulas.

\paragraph{Acknowledgments.} We are grateful to Jinho Baik, Ivan Corwin, Leonid Petrov, and Craig Tracy for helpful discussions. The authors would like to thank the organizer of the \emph{Integrable Probability Focused Research Group}, funded by NSF grants DMS-1664531, 1664617, 1664619, 1664650, for organizing stimulating events and the \emph{Park City Mathematics Institute (PCMI)} for organizing ``The $27^{th}$ Annual Summer Session, Random Matrices,'' funded by NSF grant DMS-1441467. Z.L. was supported by the University of Kansas Start Up Grant, the University of Kansas New Faculty General Research Fund, and Simons Collaboration Grant No. 637861. A.S. was partially supported by NSF grants DMS-1664617. D.W. was partially supported by the Singapore AcRF Tier 1 grant R-146-000-217-112 and the Chinese NSFC grant 11871425.

\section{Master equations, boundary conditions and initial conditions}\label{sec:asep}

We set up the master equations - describing the evolution of the probability function - for the two models. Additionally, we express the master equation as a free equation with boundary conditions. For both models, the boundary conditions are reduced to the two-particle boundary conditions due to the integrability of the models. Hence, the proof of Theorems \ref{thm:main} and \ref{thm:trans_prob} is reduced to the verification of
\begin{enumerate*}[label=(\roman*)]
\item 
  the free equation, 
\item
  the boundary conditions, and 
\item
  the initial conditions.
\end{enumerate*}
We discuss the two models separately. In the discussion for either model, we denote 
\begin{equation}
  X_i^{\pm} = (x_1, \dots, x_{i-1}, x_i \pm 1, x_{i+1}, \dots, x_N)
\end{equation}
for $X = (x_1, \dotsc, x_N) \in \intZ^N$. Note that $X^{\pm}_i$ may not belong to $\confs_N(L)$ (resp.~$\Weylchamber_N(L)$) if $X \in \confs_N(L)$ (resp.~$X \in \Weylchamber_N(L)$).

\subsection{ASEP} \label{subsec:ASEP_master}

As a continuous time Markov process, the evolution of the transition probability function for the ASEP is given by the \textit{Kolmogorov forward master equation}
\begin{multline}\label{master}
  \frac{d}{dt} \mathbb{P}_{Y} (X;t) = \sum^N_{i=2} \Big( p\mathbb{P}_{Y} (X_i^-;t) - q\mathbb{P}_{Y} (X;t) \Big) \id_{x_{i - 1} \neq x_i - 1} + \Big( p\prob_Y(X^-_1; t) - q\prob_Y(X; t) \Big) \id_{x_N \neq x_1 + L - 1} \\
  + \sum^{N - 1}_{i = 1} \Big( q\mathbb{P}_{Y} (X_i^+;t) - p\mathbb{P}_{Y} (X;t) \Big) \id_{x_{i + 1} \neq x_i + 1} + \Big( q\prob_Y(X^+_N; t) - p\prob_Y(X; t) \Big) \id_{x_1 \neq x_N - L + 1}.
\end{multline}
Note that $X^{\pm}_i$ may not be $\mathcal{X}_N(L)$, but all $\prob_Y(X^{\pm}_i; t)$ with  $X^{\pm}_i \notin \mathcal{X}_N(L)$ are nullified by the indicator function. The transition probability function $\prob_Y(X; t)$ is uniquely determined by the master equation \eqref{master} given the \emph{initial condition}
\begin{equation}
  \mathbb{P}_{Y}(X;0) = \delta_{Y}(X), \quad \text{for all $X \in \confs_N(L)$}
\end{equation}
with $Y \in \confs_N(L)$.

We further decompose the master equation by the coordinate Bethe ansatz as follows. Suppose $u_Y(X; t)$ are differentiable functions of $t$, indexed by $X, Y \in \intZ^N$. We impose the \textit{free equation} on the functions $u_Y(X; t)$. That is,
\begin{equation}\label{free}
  \frac{d}{dt} u_{Y} (X;t) = \sum_{i =1}^N p u_{Y} (X_i^-;t) + q u_{Y} (X_i^+;t) - u_{Y} (X;t).
\end{equation}
Moreover, we impose the \emph{boundary conditions} on the functions $u_Y(X; t)$. That is,
\begin{equation}\label{boundary1}
  u_Y (X;t) = p u_Y (X_1^-;t) + q u_Y (X_{N}^+;t)
\end{equation}
when $x_N + 1= x_1 + L $, and for $i = 2, \dotsc, N$
\begin{equation}\label{boundary}
  u_{Y} (X;t) = p u_{Y} (X_i^-;t) + q u_{Y} (X_{i-1}^+;t)
\end{equation}
when  $x_i = x_{i-1}+1$. We have that the master equation \eqref{master} is equivalent to the free equation \eqref{free} with the boundary conditions \eqref{boundary1} and \eqref{boundary}.

\begin{lemma} \label{lem:coor_B_ansztz}
  If the differentiable functions $u_Y(X; t)$ with $X, Y \in \intZ^N$ satisfy both the free equation \eqref{free} and the boundary conditions \eqref{boundary1} and \eqref{boundary}, then the functions $u_Y(X; t)$ with $X, Y \in \confs_N(L)$ satisfy the master equation \eqref{master} with $\prob_Y(X; t)$ replaced by $u_Y(X; t)$.
\end{lemma}
\begin{proof}
  Suppose $X = (x_1, \dotsc, x_N) \in \confs_N(L)$. If $X^-_i$ and $X^+_i$ are in $\confs_N(L)$ for all $i = 1, \dotsc, N$, or equivalently, $x_{i + 1} - x_i > 1$ for $i = 1, \dotsc, N - 1$ and $x_1 + L - x_N > 1$, then the free equation \eqref{free} is identical to the master equation \eqref{master} with $\prob_Y(X; t)$ replaced by $u_Y(X; t)$. Otherwise, we can express those $u_Y(X^-_i; t)$ and $u_Y(X^+_i; t)$ on the right side of \eqref{free} that are not in $\confs_N(L)$ by linear combinations of $u_Y(Z_{\alpha}; t)$ with $Z_{\alpha} \in\confs_N(L)$ by using \eqref{boundary1} and \eqref{boundary} recursively. Thus, we can express the right side of \eqref{free} with only terms $u_Y(Z_{\alpha}; t)$ so that $Z_{\alpha} \in \confs_N(L)$, and we find that it agrees with \eqref{master}.

  Let us consider a specific case to see how the recursion mentioned above is implemented. Take the case when we have a configuration $X \in \confs_N(L)$ with $j+1 \geq 2$ consecutive particles, which we call a cluster. For simplicity, we assume that the first or last particles are not part of the cluster and that there is no other cluster in the configuration. That is, we assume $x_{k + i} = x_k + i$ for $1< k <N -j$ and $i = 0, 1, \dotsc, j$ while $x_{i + 1} - x_i > 1$ for $i < k$ or $i \geq k + j$, and $x_1 + L - x_N > 1$ also holds. Then using \eqref{boundary} for $i = k + 1, k + 2, \dotsc, k + j$, we have
  \begin{equation}
    \sum^{k + j}_{i = k} p u_{Y} (X_i^-;t) + q u_{Y} (X_i^+;t) - u_{Y} (X;t) = p u_Y(X^-_k; t) + q u_Y(X^+_{k + j}; t) - u_Y(X; t).
  \end{equation}
  Substituting this into \eqref{free}, we find that for this $X$, we recover \eqref{master}. This argument directly generalizes to the case that there are several clusters or either/both of $x_1, x_N$ belong to a cluster.
\end{proof}

\begin{rem} 
Our proof to Lemma \ref{lem:coor_B_ansztz} uses exactly the same idea as in \cite[Footnote 2]{Tracy-Widom08}.
\end{rem}

Therefore, we compute the transition probability function $\prob_Y(X; t)$ by constructing $u_Y(X; t)$ with $X, Y \in \intZ^N$ that satisfy the free equation, boundary conditions, and the \emph{initial condition}
\begin{equation} \label{initial}
  u_Y(X; 0) = \delta_Y(X), \quad \text{for all $X, Y \in \confs_N(L)$}.
\end{equation}
 We note that the initial condition \eqref{initial} is only required for $X, Y \in \confs_N(L)$.

\subsection{$q$-TAZRP} \label{subsec:q_TAZRP_master}

We introduce some necessary notation. Denote
\begin{equation}
  \ell(X) := \text{$\#$ of congruence classes $\{ kL + i \mid k \in \intZ \}$ that contains at least one particle in state $X$},
\end{equation}
and
\begin{equation}
  v_i(X) := \text{the $i$-th largest value attained by $x_1, \dotsc, x_N$, for $i = 1, \dotsc, \ell(X)$}.
\end{equation}
Then, the congruence classes containing particles are $\{ kL + v_1 \}, \dotsc, \{ kL + v_{\ell(X)} \}$. In particular, note that $x_m = x_n$ if both $x_m$ and $x_n$ belong to the congruence class $\{ k L + v_i\}$ with $i = 2, \dots, \ell(X)$. On the other hand, if $x_m$ and $x_n$ are both in the congruence class $\{ kL + v_1 \}$, then it is possible that $x_m = v_1$ and $x_n = v_1 - L$, given that $m < n$. Recall the notation $n_{[i]}(X)$ defined in \eqref{eq:defn_n_[i](X)}. Then, generally there exist a pair of integers $n'_1(X) > 0$ and $n''_1(X) \geq 0$ so that
\begin{equation}
 n'_1(X) + n''_1(X) = n_{[v_1]}(X), \quad \text{and} \quad x_1 = \dotsb = x_{n'_1} = v_1, \quad x_{N -n''_1 + 1} = \dotsb = x_{N} = v_1 - L.
\end{equation}
We define
\begin{equation}
 N_1(x) = n'_1(X), \quad \text{and} \quad N_i(X) = n'_1(X) + n_{[v_2]}(X) + \dotsb + n_{[v_i]}(X), \quad \text{for $i = 2, \dotsc, \ell(X)$}.
\end{equation}
For $X\in\Weylchamber_N(L)$, the evolution of the transition probability function for the $q$-TAZRP is given by the \emph{Kolmogorov forward master equation}
	\begin{equation} \label{eq:concrete_Master_eq}
	 \frac{d}{dt} \prob_Y(X; t) = \sum^{\ell(X) }_{i = 1} \alpha_i(X)  \prob_Y(X^-_{N_i(X)}; t) - \beta_i(X)  \prob_Y(X; t)
	\end{equation}
with the coefficients given by
\begin{equation}
  \begin{split}
    \frac{\alpha_i (X)}{ a_{[v_i(X) - 1]}} = {}&
    \begin{cases}
      \left(1 - q^{1+ n_{[v_{i + 1}(X)]}}\right) \id_{v_{i + 1}(X) = v_i(X) - 1} + (1 - q) \id_{v_{i + 1}(X) < v_i(X) - 1}, & i \neq \ell(X), \\
      \left(1 - q^{1+ n_{[v_1(X)]} }\right) \id_{v_1(X) = v_{\ell(X)} + L - 1} + (1 - q) \id_{v_1(X) < v_{\ell(X)} + L - 1}, & i = \ell(X),
    \end{cases} \\
    \frac{\beta_i (X)} {a_{[v_i(X)]}} = {}&
    \begin{cases}
      1 - q^{n_{[v_i(X)]}}, &  i \neq \ell(X), \\
      \left(1 - q^{n_{[v_{\ell}(X)]}} \right) \id_{v_1(X) < v_{\ell}(X) + L}, & i = \ell(X).
    \end{cases}
  \end{split}
\end{equation}
The transition probability function is uniquely determined by the master equation \eqref{eq:concrete_Master_eq} given the \emph{initial condition}
\begin{equation}
  \prob_Y(X; 0) = \delta_Y(X), \quad \text{for all $X \in \Weylchamber_N(L)$}
\end{equation}
with $Y\in\Weylchamber_N(L)$.

Similar to the ASEP, we can also decompose the master equation by the Bethe ansatz. Suppose $u_Y(X; t)$ are differentiable functions with respect to $t \geq 0$, indexed by $X, Y \in \intZ^N$. For $b_{[x]} = (1 - q)a_{[x]}$, we impose the \emph{free equation}
\begin{equation} \label{eq:free_eq_q_TAZRP}
  \frac{d}{dt} u(X; t) = \sum^N_{i = 1} b_{[x_i - 1]} u(X^-_{i}; t) - b_{[x_i]} u(X; t).
\end{equation}
Moreover, we impose the boundary conditions, analogous to \eqref{boundary1} and \eqref{boundary},
\begin{equation} \label{eq:free_qTAZRP}
  b_{[x - 1]} u(X^-_N; t) = q b_{[x - 1]} u(X^-_1; t) + (1 - q) b_{[x]} u(X; t)
\end{equation}
with $x_1 = x_N + L = x$, and
\begin{equation} \label{eq:free_qTAZRP_2}
  b_{[x - 1]} u(X^-_i; t) = q b_{[x - 1]} u(X^-_{i + 1}; t) + (1 - q) b_{[x]} u(X; t)
\end{equation}
with $x_i = x_{i + 1} = x$ for $i = 1, \dotsc, N - 1$. Then, analogous to Lemma \ref{lem:coor_B_ansztz}, we have that the master equation \eqref{eq:concrete_Master_eq} is equivalent to the free equation \eqref{eq:free_eq_q_TAZRP} with the boundary conditions \eqref{eq:free_qTAZRP} and \eqref{eq:free_qTAZRP_2}.
\begin{lemma} \label{lem:coord_B_ansatz_TAZRP}
  If the differentiable functions $u_Y(X; t)$ with $X, Y \in \intZ^N$ satisfy both the free equation \eqref{eq:free_eq_q_TAZRP} and the boundary conditions \eqref{eq:free_qTAZRP} and \eqref{eq:free_qTAZRP_2}, then the functions $W(X)^{-1} u_Y(X; t)$ with $W(X)$ defined in \eqref{eq:defn_W(X)} and $X, Y \in \Weylchamber_N(L)$ satisfy the master equation \eqref{eq:concrete_Master_eq} with $\prob_Y(X; t)$ replaced by $W(X)^{-1} u_Y(X; t)$.
\end{lemma}
The proof of this lemma is similar to that of Lemma \ref{lem:coor_B_ansztz} in spirit, and more similar to the corresponding proofs in \cite{Korhonen-Lee14, Wang-Waugh16} for the $q$-TAZRP on $\intZ$.

\begin{proof}[Proof of Lemma \ref{lem:coord_B_ansatz_TAZRP}]
  We write the right-hand side of \eqref{eq:free_eq_q_TAZRP} as
  \begin{equation}
    \sum^{\ell(X)}_{j = 1} D_j, \quad \text{where} \quad D_j = \sum^{N_j}_{i = N_{j - 1} + 1} b_{[x_i - 1]} u(X^-_{i}; t) - b_{[x_i]} u(X; t), \quad \text{for } j = 2, 3, \dotsc, \ell(X),
  \end{equation}
  and $D_1$ is defined by subtracting $D_2, \dotsc, D_{\ell(X)}$ from the right-hand side of \eqref{eq:free_eq_q_TAZRP}. Equivalently, the summation for $D_j$ is over all index $i$ such that $x_i$ is in the congruence class $[v_j(X)]$. Using \eqref{eq:free_qTAZRP} and \eqref{eq:free_qTAZRP_2} for all $i$ in cyclically increasing order such that both $x_i$ and $x_{i + 1}$ are in the congruence class $[v_j(X)]$, we find that
  \begin{equation}
    D_j = (1 + q + \dotsb + q^{n_{[v_i(X)]} - 1}) (b_{[v_i(X) - 1]} u(X^-_{N_j(X)}; t) - b_{[v_i(X)]} u(X; t)).
  \end{equation}
  This leads to the desired result.
\end{proof}

Therefore, similar to the ASEP case, we find the transition probability $\prob_Y(X; t)$ for $q$-TAZRP by constructing $u_Y(X; t)$ with $X, Y \in \intZ^N$ that satisfy the free equation, boundary conditions, and the \emph{initial condition}
\begin{equation} \label{eq:initial_qTAZRP}
  u_Y(X; 0) = W(X) \delta_Y(X), \quad \text{for all $X, Y \in \Weylchamber_N(L)$}.
\end{equation}

\section{Cyclic invariance} \label{sec:cyclic_invariance}

The ASEP and the $q$-TAZRP defined on the line are translation invariant. Namely,
\begin{equation}
  \mathbb{P}_{(y_1, \dots, y_N)} (x_1, \dots, x_N; t)  = \mathbb{P}_{(y_1 + c, \dots, y_N+c)} (x_1+c, \dots, x_N+c; t) 
\end{equation}
for any integer $c \in \mathbb{Z}$. The models defined on the ring have the same translational invariance. Additionally, the geometry of the ring induces a cyclic invariance. Below, we state and prove a technical result for the cyclic invariance of $u_Y(X; t)$ defined in \eqref{eq:transition_probability} and \eqref{eq:defn_u_Y(X)}. Although the cyclic invariance of $u_Y(X; t)$ yields the cyclic invariance of $\mathbb{P}_{(y_1, \dots, y_N)} (x_1, \dots, x_N; t)$ by Theorem \ref{thm:main} or \ref{thm:trans_prob}, before we finally prove the theorems, we cannot use the probability interpretation to justify the cyclic invariance of $u_Y(X; t)$. To the contrary, we will need the cyclic invariance of $u_Y(X; t)$ to prove the theorems. Hence we give an algebraic proof of the cyclic invariance of $u_Y(X; t)$.

We define a cyclic translation for possible configurations of the models, namely operators $\tilde{}$ on $\confs_N(L)$ and $\hat{}$ on $\Weylchamber_N(L)$. For instance, if $Z = (z_1, \dots, z_N) \in \mathbb{Z}^N$, then
\begin{equation}\label{eq:cyclic_trans}
\tilde{Z} = (z_2, \dotsc, z_N, z_1 + L),\quad \text{ and } \quad \hat{Z} = (z_2, \dotsc, z_N, z_1 - L).
\end{equation}

Note that $\tilde{X}, \tilde{Y} \in \confs_N(L)$ if $X, Y \in \confs_N(L)$, and similarly, $\hat{X}, \hat{Y} \in \Weylchamber_N(L)$ if $X, Y \in \Weylchamber_N(L)$. Then, we have that the ASEP on the ring and the $q$-TAZRP on the ring are invariant under the cyclic translations defined above, respectively.

\begin{lemma} \label{lem:cyclic_invariance}
For the ASEP, take the function defined by \eqref{eq:transition_probability} in Theorem \ref{thm:main}. Then, 
  \begin{equation}
    u_{\tilde{Y}}(\tilde{X}; t) = u_{Y}(X; t),
  \end{equation}
for any configurations $X, Y  \in  \confs_N(L)$ and their cyclic translations $\tilde{X}, \tilde{Y}  \in  \confs_N(L)$ given by \eqref{eq:cyclic_trans}.

For the $q$-TAZRP, take the function defined by \eqref{eq:defn_u_Y(X)} in Theorem \ref{thm:trans_prob}. Then, 
  \begin{equation}
    u_{\hat{Y}}(\hat{X}; t) = u_{Y}(X; t),
  \end{equation}
for any configurations $X, Y  \in  \Weylchamber_N(L)$ and their cyclic translations $\hat{X}, \hat{Y}  \in  \Weylchamber_N(L)$ given by \eqref{eq:cyclic_trans}.
\end{lemma}
The proof of the lemma, although formally complicated, is based on the simple idea that the series expression of $u_{\hat{Y}}(\hat{X}; t)$ given by \eqref{eq:transition_probability} in the ASEP case and by \eqref{eq:defn_u_Y(X)} in the $q$-TAZRP case, is the same as the series expression of $u_{Y}(X; t)$ after reordering of terms.
\begin{proof}[Proof of Lemma \ref{lem:cyclic_invariance}]
  We prove only the $q$-TAZRP case. The ASEP case is similar, but the notation is heavier and more cumbersome. 
  We denote $\hat{X} = (\hat{x}_1, \dotsc, \hat{X}_N)$ such that $\hat{x}_N = x_1 - L$ and $\hat{x}_i = x_{i + 1}$ for $i = 1, \dotsc, N - 1$, and similarly denote $\hat{Y} = (\hat{y}_1, \dotsc, \hat{y}_N)$. Additionally, we introduce variables $u_1, \dotsc, u_N$ so that $u_i = w_{i + 1}$ with $i = 1, \dotsc, N - 1$ and $u_N = w_1$. We show below that there is a pair of bijections $\sigma \mapsto F(\sigma)$ and $\mathbb{L} \mapsto \widehat{\mathbb{L}}$, which the latter depends on $\sigma\in S_N$, so that $\Lambda^{\Lattice}_{\hat{Y}}(\hat{X}; t; F(\sigma)) = \Lambda^{\widehat{\Lattice}}_{Y}(X; t; \sigma)$. Then, we have the result
	\[
		u_{\hat{Y}}(\hat{X};t) = \sum_{\Lattice,\sigma } \Lambda^{\Lattice}_{\hat{Y}}(\hat{X}; t; F(\sigma)) = \sum_{\Lattice,\sigma } \Lambda^{\widehat{\Lattice}}_{Y}(X; t; \sigma) = u_{Y}(X;t).
	\]

We define a bijective mapping $F$ from $S_N$ to $S_N$, such that 
  \begin{equation}
    F(\sigma) = c^{-1} \sigma c, \quad \text{with} \quad c(N) = 1 \text{ and } c(i) = i + 1 \text{ for $i = 1, \dotsc, N - 1$}.
  \end{equation}
Below, let $\sigma \in S_N$ and $\lambda = F(\sigma)$.

Through a straight forward substitution in \eqref{eq:expr_A_sigma_qTAZRP}, we have
  \begin{equation}
    A_{\lambda}(u_1, \dotsc, u_N) = A_{\sigma}(w_1, \dotsc, w_N) \prod^N_{k = 1} \frac{(qw_1 - w_k)(qw_{\sigma(k)} - w_{\sigma(1)})}{(qw_k - w_1)(qw_{\sigma(1)} - w_{\sigma(k)})}.
  \end{equation}
for any $\sigma \in S_N$. Moreover, in the case $\sigma(1) = 1$, the product on the right side of the equality above is equal to one, meaning $A_{\lambda}(u_1, \dotsc, u_N) = A_{\sigma}(w_1, \dotsc, w_N)$ if $\sigma(1) = 1$.

If $\sigma(1) = 1$, we have
  \begin{equation}
    \begin{aligned}
      \prodprime^{\hat{x}_N}_{k = \hat{y}_{\lambda(N)}} \left( \frac{b_{[k]}}{b_{[k]} - u_{\lambda(N)}} \right) = {}& \prodprime^{x_1}_{k = y_{\sigma(1)}} \left( \frac{b_{[k]}}{b_{[k]} - w_{\sigma(1)}} \right), && \\
      \prodprime^{\hat{x}_{j - 1}}_{k = \hat{y}_{\lambda(j - 1)}} \left( \frac{b_{[k]}}{b_{[k]} - u_{\lambda(j - 1)}} \right) = {}& \prodprime^{x_j}_{k = y_{\sigma(j)}} \left( \frac{b_{[k]}}{b_{[k]} - w_{\sigma(j)}} \right), && \text{for $j = 2, \dotsc, N$}. 
    \end{aligned}
  \end{equation}
  Hence, for $\Lattice = (\ell_1, \dotsc, \ell_N)$,
  \begin{equation}
    \begin{split}
      \Lambda^{\Lattice}_{\hat{Y}}(\hat{X}; t; \lambda) = {}& \left( \prod^N_{k = 1} \frac{-1}{b_{[\hat{x}_k]}} \right) \dashint_C du_1 \dotsi \dashint_C du_N A_{\lambda}(u_1, \dotsc, u_N) \prod^N_{j = 1} \left[ \prodprime^{\hat{x}_j}_{k = \hat{y}_{\lambda(j)}} \left( \frac{b_{[k]}}{b_{[k]} - u_{\lambda(j)}} \right) e^{-u_j t} \right] \\
    & \times \prod^N_{j = 1} \left( \prod^L_{i = 1} (b_{[i]} - u_j) \prod^N_{k = 1} \left( \frac{qu_k - u_j}{qu_j - u_k} \right) \right)^{\ell_j} \\
    = {}& \left( \prod^N_{k = 1} \frac{-1}{b_{[x_k]}} \right) \dashint_C dw_1 \dotsi \dashint_C dw_N A_{\sigma}(w_1, \dotsc, w_N) \prod^N_{j = 1} \left[ \prodprime^{x_j}_{k = y_{\sigma(j)}} \left( \frac{b_{[k]}}{b_{[k]} - w_{\sigma(j)}} \right) e^{-w_j t} \right] \\
    & \times \left( \prod^L_{i = 1} (b_{[i]} - w_1) \prod^N_{k = 1} \left( \frac{qw_k - w_1}{qw_1 - w_k} \right) \right)^{\ell_N} \prod^N_{j = 2} \left( \prod^L_{i = 1} (b_{[i]} - w_j) \prod^N_{k = 1} \left( \frac{qw_k - w_j}{qw_j - w_k} \right) \right)^{\ell_{j - 1}} \\
    = {}& \Lambda^{\widehat{\Lattice}}_{Y}(X; t; \sigma),
    \end{split}
  \end{equation}
  by \eqref{eq:int_Lambda} and \eqref{eq:general_D_L_qTAZRP}, with $\widehat{\Lattice} = (\ell_N, \ell_1, \ell_2 \dotsc, \ell_{N-1})$. 

If $\sigma(1) \neq 1$, we have
  \begin{equation}
    \begin{gathered}
      \prodprime^{\hat{x}_N}_{k = \hat{y}_{\lambda(N)}} \left( \frac{b_{[k]}}{b_{[k]} - u_{\lambda(N)}} \right) = \prodprime^{x_1 - L}_{k = y_{\sigma(1)}} \left( \frac{b_{[k]}}{b_{[k]} - w_{\sigma(1)}} \right), \quad
      \prodprime^{\hat{x}_{\lambda^{-1}(N)}}_{k = \hat{y}_{N}} \left( \frac{b_{[k]}}{b_{[k]} - u_N} \right) = \prodprime^{x_{\sigma^{-1}(1)}}_{k = y_1 - L} \left( \frac{b_{[k]}}{b_{[k]} - w_1} \right),  \\
      \prodprime^{\hat{x}_{j - 1}}_{k = \hat{y}_{\lambda(j - 1)}} \left( \frac{b_{[k]}}{b_{[k]} - u_{\lambda(j - 1)}} \right) = \prodprime^{x_j}_{k = y_{\sigma(j)}} \left( \frac{b_{[k]}}{b_{[k]} - w_{\sigma(j)}} \right), \quad \text{for $j \in  \{ 1, \dotsc, N \} \setminus \{ 1, \sigma^{-1}(1) \}$},
    \end{gathered}
  \end{equation}
  Hence, for $\Lattice = (\ell_1, \dotsc, \ell_N)$,
  \begin{equation}
    \begin{split}
      \Lambda^{\Lattice}_{\hat{Y}}(\hat{X}; t; \lambda) = {}& \left( \prod^N_{k = 1} \frac{-1}{b_{[\hat{x}_k]}} \right) \dashint_C du_1 \dotsi \dashint_C du_N A_{\lambda}(u_1, \dotsc, u_N) \prod^N_{j = 1} \left[  \prodprime^{\hat{x}_j}_{k = \hat{y}_{\lambda(j)}} \left( \frac{b_{[k]}}{b_{[k]} - u_{\lambda(j)}} \right) e^{-u_j t} \right] \\
      & \times \prod^N_{j = 1} \left( \prod^L_{i = 1} (b_{[i]} - u_j) \prod^N_{k = 1} \left( \frac{qu_k - u_j}{qu_j - u_k} \right) \right)^{\ell_j} \\
      = {}& \left( \prod^N_{k = 1} \frac{-1}{b_{[x_k]}} \right) \dashint_C dw_1 \dotsi \dashint_C dw_N A_{\sigma}(w_1, \dotsc, w_N) \prod^N_{j = 1} \left[ \prodprime^{x_j}_{k = y_{\sigma(j)}} \left( \frac{b_{[k]}}{b_{[k]} - w_{\sigma(j)}} \right) e^{-w_j t} \right] \\
    & \times \left( \prod^L_{i = 1} (b_{[i]} - w_1) \prod^N_{k = 1} \left( \frac{qw_k - w_1}{qw_1 - w_k} \right) \right)^{\ell_N} \prod^N_{j = 2} \left( \prod^L_{i = 1} (b_{[i]} - w_j) \prod^N_{k = 1} \left( \frac{qw_k - w_j}{qw_j - w_k} \right) \right)^{\ell_{j - 1}} \\
    & \times \left( \prod^L_{i = 1} (b_{[i]} - w_{\sigma(1)}) \prod^N_{k = 1} \frac{(qw_{\sigma(k)} - w_{\sigma(1)})}{(qw_{\sigma(1)} - w_{\sigma(k)})} \right) \left( \prod^L_{i = 1} (b_{[i]} - w_1) \prod^N_{k = 1} \left( \frac{qw_k - w_1}{qw_1 - w_k} \right) \right)^{-1} \\
    = {}& \Lambda^{\widehat{\Lattice}}_{Y}(X; t; \sigma),
    \end{split}
  \end{equation}
  by \eqref{eq:int_Lambda} and \eqref{eq:general_D_L_qTAZRP}, with $\widehat{\Lattice} = (\ell_N -1, \ell_1, \dotsc, \ell_{\sigma(1) - 2}, \ell_{\sigma(1) - 1} + 1, \ell_{\sigma(1)}, \dotsc, \ell_{N - 1})$.

Thus, we set
	\begin{equation}
	\widehat{\Lattice} = \begin{cases} (\ell_N, \ell_1, \ell_2 \dotsc, \ell_{N-1}), \hspace{55mm} \sigma(1) =1 \\ (\ell_N -1, \ell_1, \dotsc, \ell_{\sigma(1) - 2}, \ell_{\sigma(1) - 1} + 1, \ell_{\sigma(1)}, \dotsc, \ell_{N - 1}), \quad \sigma(1) \neq 1 \end{cases}
	\end{equation}
for $\Lattice = (\ell_1, \ell_2, \dotsc, \ell_N)$. Then, the result follows by the argument at the beginning of the proof.
\end{proof}

\section{Proof of Theorems \ref{thm:main} and \ref{thm:trans_prob}}\label{sec:proof}

In this section, we show that the function $u_Y(X; t)$ defined in \eqref{eq:transition_probability} satisfies the assumptions of Lemma \ref{lem:coor_B_ansztz}, and by the statement in Lemma \ref{lem:coor_B_ansztz}, this proves Theorem \ref{thm:main}. Also, at the same time, we show that the function $u_Y(X; t)$ defined in \eqref{eq:defn_u_Y(X)} satisfies the assumptions in Lemma \ref{lem:coord_B_ansatz_TAZRP}, proving Theorem \ref{thm:trans_prob}. More specifically, we show that in either case $u_{Y}(X;t)$ satisfies the \textit{free equation}, the \textit{initial conditions}, and the \textit{boundary conditions} introduced in Section \ref{sec:asep}. For the most part, we treat both models in parallel. Checking the initial conditions for the ASEP on the ring is the most technically challenging aspect of these series of arguments as it was also the case in the work of Tracy and Widom \cite{Tracy-Widom08} for the probability function formula of the ASEP on the line. As a consequence, the arguments for checking the initial conditions of the ASEP and the $q$-TAZRP are different, but the essential idea of residue calculations is the same for both the ASEP and the $q$-TAZRP with the residues arising in the ASEP being more subtle to control. We begin by looking at the free equation in Section \ref{sec:free}, move onto the initial conditions in Section \ref{sec:IC}, and end with the boundary conditions in Section \ref{sec:boundary}.

\subsection{Free equation}\label{sec:free}

\begin{lemma}\label{lm:free}

The function $u_{Y}(X;t)$ given in \eqref{eq:transition_probability} satisfies the free equation~\eqref{free}, and the function $u_Y(X; t)$ given in \eqref{eq:defn_u_Y(X)} satisfies the free equation \eqref{eq:free_eq_q_TAZRP}.
\end{lemma}
\begin{proof}
  First we prove the ASEP case. For each set of fixed parameters $\Lattice \in \intZ^N$, $\sigma \in S_N$ and $X \in \intZ^N$, it is straight forward to check that
  \begin{equation}
    \frac{d}{dt} \Lambda^{\Lattice}_Y(X; t; \sigma) = \sum^N_{i = 1} \left( p \Lambda^{\Lattice}_Y(X^i_-; t; \sigma) + q \Lambda^{\Lattice}_Y(X^i_+; t; \sigma) - \Lambda^{\Lattice}_Y(X; t; \sigma) \right).
  \end{equation}
That is, the function $\Lambda^{\Lattice}_Y(X; t; \sigma)$ satisfies the free equation. Then, we have that
	\begin{equation}
	\frac{d}{dt} u_Y(X; t) = pu_Y(X^i_-; t) + qu_Y(X^i_+; t) - u_Y(X; t)
	\end{equation}
by summing both sides of the previous identity over $\sigma \in S_N$ and $\Lattice \in \intZ^N(0)$. The convergence of the sums over $\sigma \in S_N$ and $\Lattice \in \intZ^N(0)$ are guaranteed by Lemma \ref{lm:absolutely_convergent}. Hence, we derive that $u_Y(X; t)$ satisfies the free equation. 

  We prove the $q$-TAZRP case analogously. For each fixed set of parameters $\Lattice \in \intZ^N$, $\sigma \in S_N$ and $X \in \intZ^N$, we have
  \begin{equation} \label{eq:free_eq_part_qTAZRP}
    \frac{d}{dt} \Lambda^{\Lattice}_Y(X; t; \sigma) = \sum^N_{i = 1} \left( b_{[x_i - 1]} \Lambda^{\Lattice}_Y(X^i_-; t; \sigma) - b_{[x_i]} \Lambda^{\Lattice}_Y(X; t; \sigma) \right).
  \end{equation}
That is, the function $\Lambda^{\Lattice}_Y(X; t; \sigma)$ satisfies the free equation. Summing over $\sigma \in S_N$ and $\Lattice \in \intZ^N(0)$, we derive that $u_Y(X; t)$ satisfies the free equation. We note that the convergence is not an issue since both sides of \eqref{eq:free_eq_part_qTAZRP} are nontrivial only for finitely many $\Lattice$ by Lemma \ref{lem:qTAZRP_conv}.
\end{proof}

\subsection{Boundary conditions}\label{sec:boundary}

\begin{lemma}\label{lm:boundary}
The function $u_{Y}(X;t)$ given in \eqref{eq:transition_probability} satisfies the boundary conditions~\eqref{boundary1} and~\eqref{boundary}, and the function $u_{Y}(X;t)$ given in \eqref{eq:defn_u_Y(X)} satisfies the boundary conditions \eqref{eq:free_qTAZRP} and \eqref{eq:free_qTAZRP_2}.
\end{lemma}

We establish some notation before giving the proof. We define 
  \begin{equation}
    u^{\Lattice}_Y(X; t) = \sum_{\sigma \in S_N} u^{\Lattice}_Y(X; t; \sigma),
  \end{equation}
  in both the ASEP case and the $q$-TAZRP case with $\Lambda^{\Lattice}_Y(X; t; \sigma)$ defined by \eqref{eq:int_Lambda_ASEP} for the ASEP and \eqref{eq:int_Lambda} for the $q$-TAZRP.

\begin{proof}
  First, we give the proof for the ASEP. From \cite[Theorem 2.1]{Tracy-Widom08}, we have
  \begin{equation} \label{eq:sub_boundary_ASEP}
    u^{\Lattice}_Y(X; t) = pu^{\Lattice}_Y(X^-_i; t) + qu^{\Lattice}_Y(X^+_{i - 1}; t).
  \end{equation}
for all $i = 2, \dotsc, N$ and all $X, Y \in \intZ^N$ if $\Lattice = (0, \dotsc, 0)$. For a general $\Lattice \in \intZ^N(0)$, the same arguments found in \cite[Theorem 2.1]{Tracy-Widom08} may be applied. Thus, we have that \eqref{eq:sub_boundary_ASEP} is true for for all $i = 2, \dotsc, N$, all $X, Y \in \intZ^N$, and any $\Lattice \in \intZ^N(0)$.  Hence, by summing over all $\Lattice \in \intZ^N(0)$ on both side of \eqref{eq:sub_boundary_ASEP}, we prove that $u_Y(X; t)$ satisfies the boundary condition \eqref{boundary}.

 We prove that $u_Y(X; t)$ also satisfies the boundary condition \eqref{boundary1} by use of the cyclic invariance of $u_Y(X; t)$ discussed in Section \ref{sec:cyclic_invariance}. Suppose $X = (x - L + 1, x_2, \dotsc, x_{N - 1}, x), Y = (y_1, \dotsc, y_N) \in \intZ^N$. Then, by Lemma \ref{lem:cyclic_invariance}, we have that $u_Y(X; t) = u_{\tilde{Y}}(\tilde{X}; t)$ with $\tilde{X} = (x_2, \dotsc, x_{N - 1}, x, x + 1)$ and $\tilde{Y} = (\dotsc, y_{N - 1}, y_N, y_1 + L)$. Also, we have $u_Y(X^-_1; t) = u_{\tilde{Y}}(\tilde{X}^-_N; t)$ and $u_Y(X^+_N; t) = u_{\tilde{Y}}(\tilde{X}^+_{N - 1}; t)$. Hence, the boundary condition \eqref{boundary1} is equivalent to the condition 
	\begin{equation}
	u_{\tilde{Y}}(\tilde{X}; t) = p u_{\tilde{Y}}(\tilde{X}^-_N; t) + q u_{\tilde{Y}}(\tilde{X}^+_{N - 1}; t),
	\end{equation}
which is the $i = N$ case of \eqref{boundary}. Therefore, we have that $u_Y(X; t)$ also satisfies the boundary condition \eqref{boundary1}.

  The same arguments of the ASEP work for the $q$-TAZRP. From \cite{Korhonen-Lee14} and \cite{Wang-Waugh16}, we have
  \begin{equation}\label{eq:sub_boundary_TAZRP}
    b_{[x - 1]} u^{\Lattice}(X^-_i; t) = q b_{[x - 1]} u^{\Lattice}(X^-_{i + 1}; t) + (1 - q) b_{[x]} u^{\Lattice}(X; t)
  \end{equation}
  for all $i = 1, \dotsc, N - 1$ if $\Lattice = (0, \dotsc, 0)$. It is straightforward to generalize the identity to all $\Lattice \in \intZ^N(0)$. Hence, by summing over all $\Lattice \in \intZ^N(0)$ on both side of \eqref{eq:sub_boundary_TAZRP}, we prove that $u_Y(X; t) $ satisfies \eqref{eq:free_qTAZRP_2}.

  Next, we want to check that $u_Y(X; t)$ satisfies \eqref{eq:free_qTAZRP} with $X = (x + L, x_2, \dotsc, x_{N - 1}, x)$ and $Y = (y + L, y_2, \dotsc, y_{N - 1}, y)$. Note that, by Lemma \ref{lem:cyclic_invariance}, $u_Y(X; t)$ satisfies \eqref{eq:free_qTAZRP} if and only if $u_{\hat{Y}}(\hat{X}; t)$ satisfies the identity
  \begin{equation}
    b_{[x - 1]}u_{\hat{Y}}(\hat{X}^-_{N - 1}; t) = qb_{[x - 1]} u_{\hat{Y}}(\hat{X}^-_N; t) + (1 - q) b_{[x]} u_{\hat{Y}}(\hat{X}; t)
  \end{equation}
  with $\hat{X} = (x_2, x_3, \dotsc, x_{N - 1}, x, x)$ and $\hat{Y} = (y_2, y_3, \dotsc, y_{N - 1}, y, y)$. Since this is simply the statement that $u_{\hat{Y}}(\hat{X}; t)$ satisfies the boundary condition \eqref{eq:free_qTAZRP_2} with $i = N - 1$ and $X$ replaced by $\hat{X}$, which was proved in the previous argument, we also know that the boundary condition \eqref{eq:free_qTAZRP} is satisfied.
\end{proof}

\subsection{Initial conditions}\label{sec:IC}

	Proving that the function $u_Y(X; 0)$, given by \eqref{eq:defn_u_Y(X)}, satisfies the initial conditions \eqref{eq:initial_qTAZRP} for the $q$-TAZRP case is more straightforward than proving that the function $u_Y(X; 0)$, given by \eqref{eq:transition_probability}, satisfies the initial conditions \eqref{initial} for the ASEP case. In both cases, we have that the function $u_Y(X; 0)$ is a linear combination of $N$ nested contour integrals that depend on the parameters $\mathbb{L} \in \mathbb{Z}^{N}(0)$ and $\sigma  \in S_N$. We show, in both cases, that these integrals vanish (for $t=0$) unless $\mathbb{L}= \vec{0}= (0, \dots, 0)$. In the $q$-TAZRP case, we show that each specific nested integral vanishes (for $t=0$) unless $\mathbb{L}= \vec{0}$, but in the ASEP case, we show that the nested integrals vanish in pairs depending on the parameter $\sigma \in S_N$. The argument for the ASEP case is similar to the argument given by Tracy and Widom in \cite{Tracy-Widom08, Tracy-Widom11} where the authors show that their formula for the probability function of the ASEP on the integer line also satisfies similar initial conditions, but our arguments are a bit more involved since we are dealing with infinite sums of nested contour integrals and our integrand gives rise to more complicated residues. In the following, we leave some of the more technical proofs to the end of the section in an attempt to give more streamlined arguments.\par

\subsubsection{Notation}

We introduce some notation before the proof of the initial conditions (Lemmas \ref{lm:initial_qTAZRP} and \ref{lm:initial_ASEP}). Take $\Lattice = (\ell_1, \dotsc, \ell_N) \in \intZ^N$, and set 
\begin{equation}\label{eq:defn_mM}
  \begin{split}
    \min(\Lattice) = \min(\ell_i), & \quad \max(\Lattice) = \max(\ell_i)\\
    m = m(\Lattice) \in \{1, \dots, N \} \quad \text{such that}& \quad \ell_m = \min(\Lattice) \quad \text{and} \quad  \ell_j > \ell_m \text{ if }m < j \leq N\\
    M = M(\Lattice) \in \{ 1, \dotsc, N \} \quad \text{such that}& \quad \ell_M = \max(\Lattice) \quad \text{and} \quad \ell_i < \ell_M \text{ if } 1 \leq i < M.
  \end{split}
\end{equation}
In particular, note that $\ell_M \geq 1$ and $\ell_m \leq -1$ if $\Lattice \in \intZ^N(0)$ and $\Lattice \neq (0, \dotsc, 0)$. Additionally, for some parameters $M,m\in \mathbb{Z}$, define the subset 
	\begin{equation}
	\intZ^N(0; m; M) \subset \mathbb{Z}^N(0)
	\end{equation}
so that $(\ell_1, \dotsc, \ell_N) \in \intZ^N(0; m; M)$ if and only if 
\begin{enumerate}
\item[(i)]
  $\sum^N_{i = 0} \ell_i = 0$,
\item[(ii)]
  $\ell_m = M$, and
\item[(iii)]
  $\ell_i \geq M$ if $1 \leq i < m$ and $\ell_i > M$ if $m < i \leq N$.
\end{enumerate}
It is clear that $\intZ^N(0; m; M) = \emptyset$ for some values of $M$ and $m$, but we don't use this property in the following arguments.

Lastly, throughout this section, we use the short-hand notation $A_{\sigma} = A_{\sigma}(w_1, \dotsc, w_N)$ for the $q$-TAZRP case and $A_{\sigma} = A_{\sigma}(\xi_1, \dotsc, \xi_N)$ for the ASEP case.

\subsubsection{Initial conditions for the $q$-TAZRP}

\begin{lemma}\label{lm:initial_qTAZRP}
The function $u_Y(X; t)$ given by \eqref{eq:defn_u_Y(X)} satisfies the initial condition \eqref{eq:initial_qTAZRP}.
\end{lemma}

\begin{proof}
We assume, without loss of generality, that $Y = (y_1, \dotsc, y_N) \in \Weylchamber_N(L)$ and $y_1 - y_N < L$ by the cyclic invariance of the model (see Section \ref{sec:cyclic_invariance}). Then, by the result in \cite[Lemma 4.2]{Wang-Waugh16}, we have that
  \begin{equation}
    \sum_{\sigma \in S_N} \Lambda^{(0, \dotsc, 0)}_Y(X; 0; \sigma) = 
    \begin{cases}
      W(Y) & \text{if $X = Y$}, \\
      0 & \text{otherwise}.
    \end{cases}
  \end{equation}
 for all $X \in \Weylchamber_N(L)$. Therefore, to show the initial conditions for the $q$-TAZRP case, it suffices to show that $\Lambda^{\Lattice}_Y(X; 0; \sigma) = 0$ for all $X \in \Weylchamber_N(L)$ and $\sigma \in S_N$ if $\Lattice \neq (0, \dotsc, 0)$.

  If $N = 1$, the statement is trivial. So, we take $N \geq 2$ from now on. Recall that
  \begin{multline} \label{eq:contour_int_formula_Lambda^Lattice_Y(X0sigma)}
    \Lambda^{\Lattice}_Y(X; 0; \sigma) = \left( \prod^N_{k = 1} \frac{-1}{b_{[x_k]}} \right) \dashint_C dw_1 \dotsi \dashint_C dw_N A_{\sigma} \prod^N_{j = 1} \left[ \prodprime^{x_j}_{k = y_{\sigma(j)}} \left( \frac{b_{[k]}}{b_{[k]} - w_{\sigma(j)}} \right) \right] \\
  \times \prod^N_{j = 1} \left( \prod^L_{i = 1} (b_{[i]} - w_j) \prod^N_{k = 1} \left( \frac{qw_k - w_j}{qw_j - w_k} \right) \right)^{\ell_j},
  \end{multline}
  with the contour $C = \{ \lvert z \rvert = R \}$ for $R$ large enough (i.e.~$R > b_{[i]}$ for all $b_{[i]}$). Let $m$ and $M$ be defined as in \eqref{eq:defn_mM}. We consider two cases: the integral with respect to $w_M$ or $w_m$. 

  Consider first integrating with respect to $w_{M}$, regarding all other $w_i$ as fixed parameters with absolute values equal to $R$. Note that the integrand does not have poles at $w_M = qw_j$ with $j \neq M$. More precisely, the power of the linear factor $(qw_j - w_M)$ in the integrand can only be
  \begin{enumerate}
  \item
    $\ell_M - \ell_j + 1$ if $j > M$ and $\sigma^{-1}(j) < \sigma^{-1}(M)$;
  \item
    $\ell_M - \ell_j - 1$ if $j < M$ and $\sigma^{-1}(j) > \sigma^{-1}(M)$;
  \item
    $\ell_M - \ell_j$ otherwise,
  \end{enumerate}
  and by the definition of $M$, we find that the power is always non-negative so that there is no pole at $w_M = qw_j$ with $j \neq M$. The only poles of $w_M$ that are inside the contour $C$ are those from the term $\prod'^{x_{\sigma^{-1}(M)}}_{k = y_M} (b_{[k]}/(b_{[k]} - w_M)) \prod^L_{i = 1} (b_{[i]} - w_M)^{\ell_M}$. If $x_{\sigma^{-1}(M)} < y_M + \ell_M L$, we have that the integral with respect to $w_M$ vanishes, and then $\Lambda^{\Lattice}_Y(X; 0; \sigma) = 0$. Hence, the function $\Lambda^{\Lattice}_Y(X; 0; \sigma)$ may be non-trivial only if
  \begin{equation} \label{eq:bound_M_qTAZRP}
    x_{\sigma^{-1}(M)} \geq y_M + \ell_M L.
  \end{equation}

On the other hand, suppose we first integrate with respect to $w_m$ and regard all other $w_i$ as fixed parameters with absolute values equal to $R$. We have that all poles of $w_m$ are inside the contour $C$. Then, we take all other $w_i$ fixed while deforming the contour for $w_m$ into a larger circular contour $C' = \{ \lvert z \rvert = R' \}$ with $R' > R$. Taking $R' \to \infty$, we have that the integrand of $\Lambda^{\Lattice}_Y(X; 0; \sigma)$ is of order $\bigO((R')^{\ell_m L + y_m - x_{\sigma^{-1}(m)} - 1})$, which is uniform in $w_m \in C'$ and other $w_i \in C$. Thus, the integral with respect to $w_m$ vanishes and $\Lambda^{\Lattice}_Y(X; 0; \sigma) = 0$ by letting $R' \to \infty$ if $x_{\sigma^{-1}(m)} > \ell_m L + y_m$. Hence, the function $\Lambda^{\Lattice}_Y(X; 0; \sigma)$ may be non-trivial only if
  \begin{equation} \label{eq:bound_m_qTAZRP}
    x_{\sigma^{-1}(m)} \leq y_m + \ell_m L.
  \end{equation}

Now, for a contradiction, assume that $\Lambda^{\Lattice}_Y(X; 0; \sigma) \neq 0$ for some $\Lattice \neq (0, \dotsc, 0)$. By the previous two arguments, we must have that 
  \begin{equation} \label{eq:ineq}
    y_1 - y_N \geq y_m - y_M \geq x_{\sigma^{-1}(m)} - x_{\sigma^{-1}(M)} + (\ell_M - \ell_m)L.
  \end{equation}
Since we are considering $\Lattice \neq (0, \dotsc, 0)$, we have $\ell_M \geq 1$ and $\ell_m \leq -1$, and generally we have $x_{\sigma^{-1}(m)} - x_{\sigma^{-1}(M)} \geq -L$. Then, by the inequalities \eqref{eq:ineq}, we may conclude that $y_1 - y_N \geq L$, but this is a contradiction to the periodic assumption that $y_1 - y_N < L$. Therefore, we have that $\Lambda^{\Lattice}_Y(X; 0; \sigma) = 0$ for all $X \in \Weylchamber_N(L)$ and $\sigma \in S_N$ if $\Lattice \neq (0, \dotsc, 0)$. This establishes the initial conditions for the $q$-TAZRP case by the discussion in the beginning of the proof.
\end{proof}

\subsubsection{Initial conditions for the ASEP} \label{subsubsec:IC_ASEP}

The proof for the initial conditions for the ASEP case is based on residue computations as is the proof of the initial conditions for the $q$-TAZRP case (Lemma \ref{lm:initial_qTAZRP}). There is a big difference in the technical difficulty between the proof of the two cases since the residues of the ASEP case are more difficult to control. To this end, we begin with some preliminary results (Lemmas \ref{lem:u_Y(X0)_and_I} and \ref{lem:ASEP_IC}) that allow us to better control the residues in the formulas for the ASEP case. Since these results are mostly technical, we leave the proofs of the statements to Sections \ref{sec:u_Y(X0)_and_I} and \ref{sec:ASEP_IC_proof}.

We note that in Sections \ref{subsubsec:IC_ASEP} -- \ref{sec:u_Y(X0)_and_I}, all integrals $\dashint d\xi_i$ are over the circular contour $\lvert \xi_i \rvert = r$ with a small enough $r > 0$.
\begin{lemma}\label{lem:u_Y(X0)_and_I}
For $t=0$, the function $u_Y(X;t)$, given by \eqref{eq:transition_probability}, may be simplified as follows,
\begin{equation} \label{eq:u_Y(X0)_and_I}
  u_Y(X; 0) = \sum_{M \in \intZ, \, m = 1, \dotsc, N} (-1)^{m(N - m)} \oint \frac{dz}{2\pi i}  z^{(M + 1)N - m - 1} \sum_{\sigma \in S_N} I_{\sigma}(z; X; Y^{(m, M)}).
\end{equation}
for any $X = (x_1, \dotsc, x_N), Y = (y_1, \dotsc, y_N) \in \confs_N(L)$ and $\sigma \in S_N$ with
\begin{equation} \label{eq:contour_alt}
  I_{\sigma}(z; X; Y) = \dashint d\xi_1 \dotsi \dashint d\xi_N \prod^{N - 1}_{j = 1} \left( 1 - z\xi^L_j \prod^N_{k = 1} \frac{p + q \xi_k \xi_j - \xi_k}{p + q \xi_k \xi_j - \xi_j} \right)^{-1} A_{\sigma} \prod^N_{j = 1} \xi^{x_j - y_{\sigma(j)} - 1}_{\sigma(j)}.
\end{equation}
and the shifted configuration $Y^{(m, M)} = (y^{(m, M)}_1, \dotsc, y^{(m, M)}_N) \in \confs_N(L)$ given by  
\begin{equation} \label{eq:defn_Y^(mM)}
  y^{(m, M)}_i =
  \begin{cases}
    y_{i + m} - (M + 1)L & \text{if $i = 1, \dotsc, N - m$}, \\
    y_{i - N + m} - ML & \text{if $i = N - m + 1, \dotsc, N$}.
  \end{cases}
\end{equation}
\end{lemma}

The form of $u_Y(X;t)$ given by \eqref{eq:u_Y(X0)_and_I} is better suited for residue computations. We defer the proof of Lemma \ref{lem:u_Y(X0)_and_I} to Section \ref{sec:u_Y(X0)_and_I}. In the following Lemma \ref{lem:ASEP_IC}, we compute some residues for the integrals $I_{\sigma}$, given by \eqref{eq:contour_alt}. In particular, we group different integrals $I_{\sigma}$ together to cancel out some residues. In the following, we introduce more notation necessary for the statement of Lemma \ref{lem:ASEP_IC}.

Take two subsets $A, B \subset \{ 1, \dots , N - 1 \}$ with $\lvert A \rvert = \lvert B \rvert = n$, for some $0 \leq n \leq N-1$, and a bijective map $\varphi: A \to B$. Then, we denote
\begin{equation}\label{eq:sym_subsets_AB}
  S^{\varphi}_{A, B} = \{ \sigma \in S_N \mid \sigma(i) = \varphi(i) \text{ for all $i \in A$ and $\sigma(N) = N$} \}.
\end{equation}
Additionally, given some subset $B \subset \{ 1, \dots , N - 1 \}$, we introduce an integral $I^{(B)}_{\sigma}$ that generalizes the integral $I_{\sigma}$ given by \eqref{eq:contour_alt}. We set
\begin{equation}\label{eq:contour_alt_B}
  I^{(B)}_{\sigma}(z; X; Y) = \dashint d\xi_1 \dotsi \dashint d\xi_N \prod_{j \in \{1, \dotsc, N - 1 \} \setminus B} \left( 1 - z\xi^L_j \prod^N_{k = 1} \frac{p + q \xi_k \xi_j - \xi_k}{p + q \xi_k \xi_j - \xi_j} \right)^{-1} A_{\sigma} \prod^N_{j = 1} \xi^{x_j - y_{\sigma(j)} - 1}_{\sigma(j)}.
\end{equation}
Note that $I_{\sigma}(z; X; Y) = I^{(B)}_{\sigma}(z; X; Y)$ for $B = \emptyset$. 

\begin{lemma} \label{lem:ASEP_IC}
For any pair of subsets $A, B \subset \{ 1, \dots, N - 1 \}$, so that $0 \leq \lvert A \rvert = \lvert B \rvert  \leq N-1$, and any bijective map $\varphi: A \to B$, we have 
  \begin{equation} \label{eq:sum_SAB}
    \sum_{\sigma \in S^{\varphi}_{A, B}} I^{(B)}_{\sigma}(z; X; Y) = \sum_{\sigma \in S^{\varphi}_{A, B}} \dashint d\xi_1 \dotsi \dashint d\xi_N A_{\sigma} \prod^N_{j = 1} \xi^{x_j - y_{\sigma(j)} - 1}_{\sigma(j)}.
  \end{equation}
with $S^{\varphi}_{A, B}$ given by \eqref{eq:sym_subsets_AB} and $I^{(B)}_{\sigma}$ given by \eqref{eq:contour_alt_B}.
\end{lemma}

This Lemma \ref{lem:ASEP_IC} is the corner stone for the proof of the initial condition for the ASEP case. Moreover, this Lemma \ref{lem:ASEP_IC} is reminiscent of \cite[Lemma 2.1]{Tracy-Widom08} that is proved in \cite{Tracy-Widom11}, which was also crucial for the proof of the initial conditions of the ASEP on the integer lattice by Tracy and Widom. In fact, the proof of Lemma \ref{lem:ASEP_IC} relies on residue computations and cancellations very similar to the ones given in the proof of \cite[Lemma 2.1]{Tracy-Widom08} in \cite{Tracy-Widom11} by Tracy and Widom. We defer the proof of Lemma \ref{lem:ASEP_IC} to Section \ref{sec:ASEP_IC_proof}, and now, we state and prove the initial conditions for the ASEP case.

\begin{lemma}\label{lm:initial_ASEP}
The function $u_{Y}(X;t)$ given by \eqref{eq:transition_probability} satisfies the initial conditions~\eqref{initial}
\end{lemma}

\begin{proof}
We show
\begin{equation} \label{eq:sum_I}
  \sum_{\sigma \in S_N} I_{\sigma}(z; X; Y) = \delta_{X, Y}.
\end{equation}
for any $X, Y \in \confs_N(L)$. This implies the desired result that $u_Y(X; 0) = \delta_{X, Y}$ by \eqref{eq:u_Y(X0)_and_I} and \eqref{eq:defn_Y^(mM)}.

We compute the left side of \eqref{eq:sum_I} by decomposing the summation into $N$ distinct summations. In particular, we consider the summations over the subsets
\begin{equation}
  S_N(n) = \{ \sigma \in S_N \mid \sigma(n) = N \} \subset S_N.
\end{equation}
First, we take $n = N$ and consider the summation of $I_{\sigma}$ over $S_N(N)$. Note that we may write $S_N(N) = S^{0}_{\emptyset, \emptyset}$ using the notation from \eqref{eq:sym_subsets_AB} with $0$ representing the trivial mapping from $A = \emptyset$ to $B = \emptyset$. Then, by Lemma \ref{lem:ASEP_IC}, we have
\begin{equation}\label{eq:sum_I_SNN}
  \sum_{\sigma \in S_N(N)} I_{\sigma}(z; X; Y) = \sum_{\sigma \in S_{\emptyset, \emptyset}^0} I^{(\emptyset)}_{\sigma}(z; X; Y) = \sum_{\sigma \in S_N(N)} \dashint d\xi_1 \dotsi \dashint d\xi_N A_{\sigma} \prod^N_{j = 1} \xi^{x_j - y_{\sigma(j)} - 1}_{\sigma(j)} = \delta_{X, Y},
\end{equation}
with the last equality following from the initial condition result for the ASEP on the integer line (see Theorem 2.1 \cite{Tracy-Widom08}).

Now, we consider the summation of $I_{\sigma}$ over $S_N(n)$ for $n \neq N$. We decompose $S_N(n)$ further into subsets
	\begin{equation}
	S_{I} = S_{(i_{n + 1}, \dotsc, i_N)} = \{\sigma \in S_N(n) | \sigma(k) = i_k \text{ for } k =n+1, \dots, N \}
	\end{equation}
for any subset $I= \{i_{n+1}, \dots, i_{N}\} \subset \{1, \dots, N-1 \}$ of $N-n$ distinct integers. Then, for each $\sigma \in S_{I}$ with $I = \{ i_{n + 1}< \dotsb < i_N\}$, we define $\tau \in S_N(N)$ so that
\begin{equation}
  \tau(j) =
  \begin{cases}
    \sigma(j + n) = i_{j + n} & j = 1, \dotsc, N - n, \\
    \sigma(j + n - N) & j = N - n + 1, \dotsc, N.
  \end{cases}
\end{equation}
Moreover, given $X = (x_1, \dotsc, x_N) \in \confs_N(L)$, we introduce the shifted configuration $X^{(n)} = (x^{(n)}_1, \dotsc, x^{(n)}_N) \in \confs_N(L)$ with
\begin{equation}
  x^{(n)}_j =
  \begin{cases}
    x_{j + n} - L & j = 1, \dotsc, N - n, \\
    x_{j + n - N} & j = N - n + 1, \dotsc, N.
  \end{cases}
\end{equation}
Then, we write
\begin{equation}
  A_{\sigma} \prod^N_{j = 1} \xi^{x_j - y_{\sigma(j)} - 1}_{\sigma(j)} = (-1)^{n(N - n)} A_{\tau} \prod^N_{j = 1} \xi^{x^{(n)}_j - y_{\tau(j)} - 1}_{\tau(j)} \prod_{j \in I} \left( \xi^L_j \prod_{k \in I^c} \frac{p + q\xi_j \xi_k - \xi_k}{p + q\xi_j \xi_k - \xi_j} \right).
\end{equation}
for $I^c = \{ 1, \dots, N\} \setminus I$. In particular, we have
\begin{equation} \label{eq:sigma_by_tau_last}
\begin{split}
  \sum_{\sigma \in S_{I}} I_{\sigma}(z; X; Y)& = (-1)^{n} (-z)^{n - N} \sum_{\tau \in S_I'} \dashint d\xi_1 \dotsi \dashint d\xi_N A_{\tau} \prod^N_{j = 1} \xi^{x^{(n)}_j - y_{\tau(j)} - 1}_{\tau(j)}\\
	& \times \prod_{j \in I} \left( z\xi^L_j \prod_{k =1}^N \frac{p + q\xi_j \xi_k - \xi_k}{p + q\xi_j \xi_k - \xi_j} \right) \prod^{N - 1}_{j = 1} \left( 1 - z\xi^L_j \prod^N_{k = 1} \frac{p + q \xi_k \xi_j - \xi_k}{p + q \xi_k \xi_j - \xi_j} \right)^{-1}
\end{split}
\end{equation}
for any $I= \{ i_{n + 1}< \dotsb < i_N\} \subset \{1, \dots, N - 1 \}$ and 
\begin{equation}
  S_I' = \{ \tau \in S_N(N) | \tau(j) = i_{j+n} \text{ for } j = 1, \dots, N-n\} \subset S_N(N).
\end{equation}
Now, consider the identity
\begin{equation}
\left( 1 - z\xi^L_j \prod^N_{k = 1} \frac{p + q \xi_k \xi_j - \xi_k}{p + q \xi_k \xi_j - \xi_j} \right)^{-1}  - 1= \left(z\xi^L_j \prod^N_{k = 1} \frac{p + q \xi_k \xi_j - \xi_k}{p + q \xi_k \xi_j - \xi_j}\right) \left( 1 - z\xi^L_j \prod^N_{k = 1} \frac{p + q \xi_k \xi_j - \xi_k}{p + q \xi_k \xi_j - \xi_j} \right)^{-1}.
\end{equation}
for any $1 \leq j \leq N-1$. Then, we simply \eqref{eq:sigma_by_tau_last} further as
\begin{multline}
  (-1)^{n} (-z)^{N-n}\sum_{\sigma \in S_{I}} I_{\sigma}(z; X; Y) = \sum_{\tau \in S_I'} \sum^{N - n}_{l = 0} (-1)^l \sum_{S \subseteq I \text{ and } \lvert S \rvert = l} \dashint d\xi_1 \dotsi \dashint d\xi_N A_{\tau} \\
  \times \prod^N_{j = 1} \xi^{x^{(n)}_j - y_{\tau(j)} - 1}_{\tau(j)} \prod_{j \in \{ 1, \dotsc, N - 1 \} \setminus S} \left( 1 - z\xi^L_j \prod^N_{k = 1} \frac{p + q \xi_k \xi_j - \xi_k}{p + q \xi_k \xi_j - \xi_j} \right)^{-1}.
\end{multline}
So, if we consider the summation of $I_{\sigma}$ over $S_N(n)$, we may write 
\begin{equation} \label{eq:in_sum_of_I_S}
  (-1)^{n} (-z)^{N-n} \sum_{\sigma \in S_N(n)} I_{\sigma}(z; X; Y) = \sum^{N - n}_{l = 0} (-1)^l \sum_{S \subseteq \{ 1, \dotsc, N - 1 \} \text{ and } \lvert S \rvert = l} I_S,
\end{equation}
with
\begin{multline}
  I_S = \sum_{\tau \in S_N(N) \text{ and } \tau(\{ 1, \dotsc, N - n \}) \supseteq S} \dashint d\xi_1 \dotsi \dashint d\xi_N A_{\tau} \\
  \times \prod^N_{j = 1} \xi^{x^{(n)}_j - y_{\tau(j)} - 1}_{\tau(j)}\prod_{j \in \{ 1, \dotsc, N - 1 \} \setminus S} \left( 1 - z\xi^L_j \prod^N_{k = 1} \frac{p + q \xi_k \xi_j - \xi_k}{p + q \xi_k \xi_j - \xi_j} \right)^{-1}.
\end{multline}
Moreover, by Lemma \ref{lem:ASEP_IC}, we have
\begin{equation} \label{eq:expr_I_S}
  I_S = \sum_{\tau \in S_N(N) \text{ and } \tau(\{ 1, \dotsc, N - n \}) \supseteq S} \dashint d\xi_1 \dotsi \dashint d\xi_N A_{\tau} \prod^N_{j = 1} \xi^{x^{(n)}_j - y_{\tau(j)} - 1}_{\tau(j)}.
\end{equation}
Finally, we compute the summation of $I_{\sigma}$ over $S_N(n)$ by using formula \eqref{eq:in_sum_of_I_S} and the expression \eqref{eq:expr_I_S} for $I_S$. We have
\begin{equation}
  \begin{split}
  {}&(-1)^{n} (-z)^{N-n} \sum_{\sigma \in S_N(n)} I_{\sigma}(z; X; Y) \\
    = {}& \sum_{\tau \in S_N(N)} \sum^{N - n}_{l = 0} (-1)^l \sum_{S \subseteq \tau(\{ 1, \dotsc, N - n \}) \text{ and } \lvert S \rvert = l} \dashint d\xi_1 \dotsi \dashint d\xi_N A_{\tau} \prod^N_{j = 1} \xi^{x^{(n)}_j - y_{\tau(j)} - 1}_{\tau(j)} \\
    = {}& \sum_{\tau \in S_N(N)}  \left(\dashint d\xi_1 \dotsi \dashint d\xi_N A_{\tau} \prod^N_{j = 1} \xi^{x^{(n)}_j - y_{\tau(j)} - 1}_{\tau(j)}\right)  \sum^{N - n}_{l = 0} (-1)^l \frac{(N - n)!}{l!(N - n - l)!}  \\
    = {}& \sum_{\tau \in S_N(N)}  \left(\dashint d\xi_1 \dotsi \dashint d\xi_N A_{\tau} \prod^N_{j = 1} \xi^{x^{(n)}_j - y_{\tau(j)} - 1}_{\tau(j)} \right)(1 -1)^{N-n}\\
   = {}& 0.
  \end{split}
\end{equation}
Thus, we showed that the sum of $I_{\sigma}(z; X; Y)$ over $S_N(n)$ vanishes if $n < N$. Together with \eqref{eq:sum_I_SNN}, this show that the sum of $I_{\sigma}$ over $S_N$ is $\delta_{X, Y}$, completing the proof.
\end{proof}

\subsubsection{Proof for Lemma \ref{lem:u_Y(X0)_and_I}}\label{sec:u_Y(X0)_and_I}

We use the indicator function
	\begin{equation}
	\mathds{1}(\mathbb{L})= \oint \frac{dz}{2 \pi i z} \prod_{j=1}^N z^{\ell_j} = \begin{cases} 1 \quad \ell_1 + \cdots + \ell_N  = 0 \\0 \quad \ell_1 + \cdots + \ell_N  \neq  0\end{cases},
	\end{equation}
with the contour of $z$ a small enough circle centered at the origin, for any $\mathbb{L} = (\ell_1, \dots, \ell_N) \in \mathbb{Z}^N$. Then, we write
\begin{equation} \label{eq:express_IC}
  u_Y(X; 0) =  \sum_{\sigma \in S_N} \sum_{\Lattice \in \intZ^N(0)} u^{\Lattice}_Y(X; 0; \sigma) = \sum_{M \in \intZ, \, m = 1, \dotsc, N} \sum_{\sigma \in S_N} \sum_{\Lattice \in \intZ^N(0; m, M)} u^{\Lattice}_Y(X; 0; \sigma) \mathds{1}(\mathbb{L}),
\end{equation}
since $\intZ^N(0) \subset \left( \cup_{m=1}^N \cup_{M = - \infty}^{\infty}\intZ^N(0; m, M)\right)$. In particular, we have
\begin{equation}
  u^{\Lattice}_Y(X; 0; \sigma)\mathds{1}(\mathbb{L}) = \dashint d\xi_1 \dotsi \dashint d\xi_N \oint \frac{dz}{2\pi i z} \prod^N_{j = 1} \left( z\prod^N_{j = 1} \xi^L_j \prod^N_{k = 1} \frac{p + q\xi_j \xi_k - \xi_k}{p + q\xi_k \xi_k - \xi_j} \right)^{\ell_j} A_{\sigma} \prod^N_{j = 1} \xi^{x_j - y_{\sigma(j)} - 1}_{\sigma(j)},
\end{equation}
and moreover, we have
\begin{multline} \label{eq:sum_u^L_Y(X)_Z(0mn)}
  \sum_{\Lattice \in \intZ^N(0; m, M)} u^{\Lattice}_Y(X; 0; \sigma)\mathds{1}(\mathbb{L}) = \dashint d\xi_1 \dotsi \dashint d\xi_N \oint \frac{dz}{2\pi i} z^{(M + 1)N - m - 1} A_{\sigma} \prod^N_{j = 1} \xi^{x_j - y_{\sigma(j)} + ML - 1}_{\sigma(j)} \\
  \times  \prod_{\substack{j = 1, \dotsc, N \\ j \neq m}} \left( 1 - z\xi^L_j \prod^N_{k = 1} \frac{p + q\xi_j \xi_k - \xi_k}{p + q\xi_j \xi_k - \xi_j} \right)^{-1}  \prod^N_{j = m + 1} \left( \xi^L_j \prod^m_{k = 1} \frac{p + q\xi_j \xi_k - \xi_k}{p + q\xi_j \xi_k - \xi_j} \right) .
\end{multline}
Now, we introduce the change of variables
\begin{equation}
  \xi_j =
  \begin{cases}
    \zeta_{j + N - m} & \text{if $j = 1, \dotsc, m$}, \\
    \zeta_{j - m} & \text{if $j = m + 1, \dotsc, N$}
  \end{cases}
\end{equation}
for any $m \in \{ 1, \dotsc, N \}$, and the permutation 
\begin{equation} \label{eq:tau_and_sigma}
  \tau(j) =
  \begin{cases}
    \sigma(j) + N - m & \text{if $\sigma(j) = 1, \dotsc, m$}, \\
    \sigma(j) - m & \text{if $\sigma(j) = m + 1, \dotsc, N$}
  \end{cases}
\end{equation}
for the same $m$ as in the change of variables. Then, we have
\begin{equation}
 (-1)^{m(N - m)} A_{\tau}(\zeta_1, \dotsc, \zeta_N) =  \prod^m_{k = 1} \prod^N_{j = m + 1} \frac{p + q \xi_k \xi_j - \xi_k}{p + q \xi_k \xi_j - \xi_j} A_{\sigma}(\xi_1, \dotsc, \xi_N).
\end{equation}
Additionally, for any configuration $Y = (y_1, \dotsc, y_N) \in \confs_N(L)$ and any $M \in \mathbb{Z}$, we introduce the shifted configuration $Y^{(m, M)} = (y^{(m, M)}_1, \dotsc, y^{(m, M)}_N) \in \confs_N(L)$ given by  
\begin{equation} 
  y^{(m, M)}_i =
  \begin{cases}
    y_{i + m} - (M + 1)L & \text{if $i = 1, \dotsc, N - m$}, \\
    y_{i - N + m} - ML & \text{if $i = N - m + 1, \dotsc, N$}.
  \end{cases}
\end{equation}
Then, we have 
\begin{equation}
  \prod^N_{j = m + 1} \xi^L_j \prod^N_{j = 1} \xi^{x_j - y_{\sigma(j)} + ML - 1}_{\sigma(j)} = \prod^N_{j = 1} \zeta^{x_j - y^{(m, M)}_{\tau(j)} - 1}_{\tau(j)},
\end{equation}
Hence, for any $X = (x_1, \dotsc, x_N), Y = (y_1, \dotsc, y_N) \in \confs_N(L)$ and $\sigma \in S_N$, we write \eqref{eq:sum_u^L_Y(X)_Z(0mn)} as
\begin{equation} \label{eq:contour_int_of_I_tau(z)}
  \sum_{\Lattice \in \intZ^N(0; m, M)} u^{\Lattice}_Y(X; 0; \sigma) \mathds{1}(\mathbb{L}) = (-1)^{m(N - m)} \oint \frac{dz}{2\pi i}  z^{(M + 1)N - m - 1} I_{\tau}(z; X; Y^{(m, M)}),
\end{equation}
with $\tau, \sigma \in S_N$ given by the one-to-one mapping in \eqref{eq:tau_and_sigma} and $I_{\sigma}(z; X; Y)$ is defined in \eqref{eq:contour_alt}. Therefore, by \eqref{eq:express_IC} and \eqref{eq:contour_int_of_I_tau(z)}, we have
\begin{equation} 
  u_Y(X; 0) = \sum_{M \in \intZ, \, m = 1, \dotsc, N} (-1)^{m(N - m)} \oint \frac{dz}{2\pi i}  z^{(M + 1)N - m - 1} \sum_{\sigma \in S_N} I_{\sigma}(z; X; Y^{(m, M)}).
\end{equation}
This completes the proof of Lemma \ref{lem:u_Y(X0)_and_I}.

\subsubsection{Proof for Lemma \ref{lem:ASEP_IC}}\label{sec:ASEP_IC_proof}

We prove Lemma \ref{lem:ASEP_IC} by induction on $n = N - 1 - \lvert A \rvert = N - 1 - \lvert B \rvert$. Note that the statement of the lemma is trivial for $\lvert A \rvert = \lvert B \rvert = N - 1$ and $n = 0$. In this subsubsection, we take a slightly different choice of the contour for $\dashint d\xi_i$. Instead of assuming $\lvert \xi_i \rvert = r$ for all $i$ in Sections \ref{subsubsec:IC_ASEP} -- \ref{sec:u_Y(X0)_and_I}, we assume that $\lvert x_i \rvert$ are around a small $r$ but distinct. This is harmless for the $n = 1$ case and useful for the $n \geq 2$ case. See Remark \ref{rem:distinct_radii}.

We start with $\lvert A \rvert = \lvert B \rvert = N-2$ and $n = 1$. Then, there is a single $a \in \{ 1, \dotsc, N - 1 \} \setminus A$ and a single $b \in \{ 1, \dotsc, N - 1 \} \setminus B$, and there is a single $\sigma \in S^{\varphi}_{A, B}$ such that $\sigma(a) = b$ for a given $\varphi: A \rightarrow B$. Hence, in this special case, the summation on the left side of \eqref{eq:sum_SAB} (which we want to prove) has a single term. In particular, for $n = 1$, we want to show 
\begin{equation} \label{eq:|B|=N-2}
  \dashint \xi_1 \dotsi \dashint \xi_N \frac{A_{\sigma} \prod^N_{j = 1} \xi^{x_j - y_{\sigma(j)} - 1}_{\sigma(j)}}{ 1 - z \xi^L_b \prod^N_{k = 1} \frac{p + q\xi_k \xi_b - \xi_k}{p + q\xi_k \xi_b - \xi_b}}  = \dashint \xi_1 \dotsi \dashint \xi_N A_{\sigma} \prod^N_{j = 1} \xi^{x_j - y_{\sigma(j)} - 1}_{\sigma(j)}
\end{equation}
for any $\sigma \in S_N$ with $\sigma(N)=N$. So, we introduce the change of variable
\begin{equation} \label{eq:change_variable_eta}
  \xi_N = \frac{\eta}{\xi_1 \dotsm \xi_{N - 1}},
\end{equation}
and, using a geometric series since we may take $\lvert z \rvert$ arbitrarily small, we rewrite the left side of \eqref{eq:|B|=N-2} as
\begin{equation} \label{eq:eta_change_|B|=N-2}
  \dashint \xi_1 \dotsi \dashint \xi_{N - 1} \dashint d\eta \frac{A_{\sigma} \prod^{N - 1}_{j = 1} \xi^{x_j - y_{\sigma(j)} - x_N + y_N - 1}_{\sigma(j)} \eta^{x_N - y_N - 1}_N}{\left( 1 - z \xi^L_b \frac{p\xi_1 \dotsm \xi_{N - 1} + q\xi_b \eta - \eta}{(p - \xi_b) \xi_1 \dotsm \hat{\xi}_b \dotsm \xi_{N - 1} + q\eta} \prod^{N - 1}_{k = 1} \frac{p + q\xi_k \xi_b - \xi_k}{p + q\xi_k \xi_b - \xi_b}  \right)^{-1}}  = \sum^{\infty}_{l = 0} (-z)^l I_l, 
\end{equation}
with $\hat{\xi}_b$ denoting that the term $\xi_b$ is removed in the product and 
\begin{multline} \label{eq:defn_I_l}
  I_l = \dashint \xi_1 \dotsi \dashint \xi_{N - 1} \dashint d\eta \left(\xi^L_b \frac{p\xi_1 \dotsm \xi_{N - 1} + q\xi_b \eta - \eta}{(p - \xi_b) \xi_1 \dotsm \hat{\xi}_b \dotsm \xi_{N - 1} + q\eta} \prod^{N - 1}_{k = 1} \frac{p + q\xi_k \xi_b - \xi_k}{p + q\xi_k \xi_b - \xi_b}  \right)^l \\
  \times A_{\sigma} \prod^{N - 1}_{j = 1} \xi^{x_j - y_{\sigma(j)} - x_N + y_N - 1}_{\sigma(j)} \eta^{x_N - y_N - 1}_N.
\end{multline}
Note that, in \eqref{eq:eta_change_|B|=N-2} and \eqref{eq:defn_I_l}, the factor $A_{\sigma}$ is independent of $\xi_N$ since $\sigma(N) = N$, meaning that the term $A_{\sigma}$ is not affected by the change of variable \eqref{eq:change_variable_eta}. Also, note that $\xi_b = 0$ is the only possible pole  in the contour integral formula for $I_l$ \eqref{eq:defn_I_l} if we integrate over $\xi_b$ first. Lastly, note that
\begin{equation} \label{eq:pole_ineq}
  x_{\sigma^{-1}(b)} - y_b - x_N + y_N - 1 + l(L - 1) \geq 0,
\end{equation}
if $l \geq 1$. So, the integral over $\xi_b$ vanishes for $l \geq 1$, and we have that $I_l = 0$ for all $l \geq 1$. Then, the integral in left side of \eqref{eq:eta_change_|B|=N-2}, which is equivalent to the left side of \eqref{eq:|B|=N-2}, is equal to $I_0$ defined in \eqref{eq:defn_I_l}, which is equivalent to the right-hand side of \eqref{eq:|B|=N-2}. This proves the lemma for the case $n = N-2$.

Now, consider the case $n \geq 2$, that is, $\lvert A \rvert = \lvert B \rvert \leq N - 3$, and assume the lemma is true for smaller $n$. To make notation simpler, we assume without loss of generality that $A = B = \{ n + 1, \dotsc, N - 1 \}$ and $\varphi(i) = i$ for $i \in A$. Then, we apply the change of variable given by \eqref{eq:change_variable_eta}, and we obtain
\begin{multline} \label{eq:sum_I^(B)}
  \sum_{\sigma \in S^{\varphi}_{A, B}} I^{(B)}_{\sigma}(z; X; Y) = \sum_{\sigma \in S^{\varphi}_{A, B}} \dashint d\xi_1 \dotsi \dashint d\xi_{N - 1} \dashint d\eta A_{\sigma} \prod^{N - 1}_{j = 1} \xi^{x_j - y_{\sigma(j)} - x_N + y_N - 1}_{\sigma(j)} \eta^{x_N - y_N - 1}_N \\
  \times \prod^n_{j = 1} \left( 1 - z \xi^L_j \frac{p\xi_1 \dotsm \xi_{N - 1} + q\xi_j \eta - \eta}{(p - \xi_j) \xi_1 \dotsm \hat{\xi}_j \dotsm \xi_{N - 1} + q\eta} \prod^{N - 1}_{k = 1} \frac{p + q\xi_k \xi_j - \xi_k}{p + q\xi_k \xi_j - \xi_j}  \right)^{-1}.
\end{multline}
For each $\sigma \in S^{\varphi}_{A, B}$, we integrate over $\xi_n$ first. Then, we need to consider $n$ possible poles. There are two type of locations for the poles:
\begin{enumerate*}[label=(\arabic*)]
\item  
  at $\xi_n = 0$, depending on the exponent of the variable $\xi_n$; and
\item
  at $\xi_n = \xi^{(j)}_n$, for $j =1, \dots, n-1$, so that $\xi_n$ is a solution of the equation
\end{enumerate*}
\begin{equation} \label{eq:defn_xi^j_n}
  1 - z\xi^{L - 1}_j \frac{p \xi_1 \dotsm \xi_{N - 1} + q\xi_j \eta - \eta}{(p - \xi_j) \xi_1 \dotsm \hat{\xi}_j \dotsm \xi_{N - 1} + q\eta} \prod^{N - 1}_{k = 1} \frac{p + q \xi_k \xi_j - \xi_k}{p + q \xi_k \xi_j - \xi_j}  = 0.
\end{equation}
Note that $\xi^{(j)}_n$ depends on $\xi_1, \dotsc, \xi_{n - 1}, \xi_{n + 1}, \dotsc, \xi_{N - 1}, \eta$ analytically. Also, it turns out that the residues at $\xi^{(j)}_n$ are non-trivial. Below we show that the residues at $\xi^{(j)}_n$ associated to each $\sigma$ cancel out to zero as a summation over all $\sigma \in S^{\varphi}_{A, B}$. In particular, we will compute the residue at $\xi^{(j)}_n$ as an $(N - 1)$-fold contour integral with respect to the variables $\xi_1, \dotsc, \xi_{n - 1}, \xi_{n + 1}, \dotsc, \xi_{N - 1}, \eta$, and then we take another integral (but now over $\xi_j$) to obtain the cancellation.

Now, consider the residue at $\xi_n = \xi^{(j)}_n$ more carefully. We write
\begin{equation} \label{eq:defn_xi^j_n_2}
\begin{split}
  &\left( 1 - z\xi^{L - 1}_j \frac{p \xi_1 \dotsm \xi_{N - 1} + q\xi_j \eta - \eta}{(p - \xi_j) \xi_1 \dotsm \hat{\xi}_j \dotsm \xi_{N - 1} + q\eta} \prod^{N - 1}_{k = 1} \frac{p + q \xi_k \xi_j - \xi_k}{p + q \xi_k \xi_j - \xi_j} \right)^{-1} \\
&= \frac{\left( (p - \xi_j) \xi_1 \dotsm \hat{\xi}_j \dotsm \xi_{N - 1} + q\eta\right) \left( p + q \xi_n \xi_j - \xi_j\right)}{f (\xi_1, \dots, \xi_{N-1},\eta )}
\end{split}
\end{equation}
with
	\begin{equation}
	\begin{split}
	f (\xi_1, \dots, \xi_{N-1},\eta ) &= \left( (p - \xi_j) \xi_1 \dotsm \hat{\xi}_j \dotsm \xi_{N - 1} + q\eta\right)  \left( p + q \xi_n \xi_j - \xi_j\right)\\
	& -  z\xi^{L - 1}_j \left(p \xi_1 \dotsm \xi_{N - 1} + q\xi_j \eta - \eta\right) \left(p + q \xi_n \xi_j - \xi_n\right) \prod^{N - 1}_{\substack{k = 1\\ k \neq n}} \frac{p + q \xi_k \xi_j - \xi_k}{p + q \xi_k \xi_j - \xi_j}.
	\end{split}
	\end{equation}
Note that the function $f(\xi_1, \dots, \xi_{N-1},\eta )$ is a quadratic polynomial with respect to $\xi_n$, and so, we write 
\begin{equation}
  f(\xi_1, \dots, \xi_{N-1},\eta ) = C (\xi_n - \xi^{(j)}_n)(\xi_n - \tilde{\xi}^{(j)}_n), \quad C = [q (p - \xi_j) - pz \xi^{L - 1}_j(q\xi_j - 1)] \xi_1 \cdots \hat{\xi_n} \cdots \xi_{N-1}
\end{equation}
Also, for the contours of integrations small enough, one may check that one of the roots of $f(\xi_1, \dots, \xi_{N-1},\eta )$ with respect to $\xi_n$ will be inside the contour of integration and the other root will be outside of the contour of integration. We take $\xi^{(j)}_n$ to be the root inside the contour of integration and the root $\tilde{\xi}^{(j)}_n$ to be outside of the contour of integration. Then, we write the residue at $\xi_n = \xi^{(j)}_n$ as
\begin{equation}
\begin{split}
&\dashint d\xi_1 \dotsi \dashint d\xi_{N - 1} \dashint d\eta A_{\sigma} \prod^{N - 1}_{s= 1} \xi^{x_s - y_{\sigma(s)} - x_N + y_N - 1}_{\sigma(s)} \eta^{x_N - y_N - 1}_N \\
  &\times \prod^n_{\substack{s = 1 \\ s \neq j}} \left( 1 - z \xi^L_s \frac{p\xi_1 \dotsm \xi_{N - 1} + q\xi_s \eta - \eta}{(p - \xi_s) \xi_1 \dotsm \hat{\xi}_s \dotsm \xi_{N - 1} + q\eta} \prod^{N - 1}_{k = 1} \frac{p + q\xi_k \xi_s - \xi_k}{p + q\xi_k \xi_s - \xi_s}  \right)^{-1}\\
& \times \frac{\left((p - \xi_j) \xi_1 \dotsm \hat{\xi}_j \dotsm \xi_{N - 1} + q\eta\right) \left( p + q \xi_n \xi_j - \xi_j\right)}{C(\xi_n^{(j)} - \tilde{\xi}_n^{(j)})}
\end{split}
\end{equation}
with the integral with respect to $\xi_n$ omitted and all $\xi_n$ evaluated as $\xi^{(j)}_n$. Alternatively, we may write the residue at $\xi_n = \xi^{(j)}_n$ as
\begin{equation}\label{eq:first_res}
\begin{split}
&\dashint d\xi_1 \dotsi \dashint d\hat{\xi}_n \dotsi \dashint d\xi_{N - 1} \dashint d\eta A_{\sigma} \prod^{N - 1}_{s = 1} \xi^{x_s - y_{\sigma(s)} - x_N + y_N - 1}_{\sigma(s)} \eta^{x_N - y_N - 1}_N \\
  &\times \prod_{s = 1, \dotsc, n \text{ and } s \neq j} \left( 1 - z \xi^L_s \frac{p\xi_1 \dotsm \xi_{N - 1} + q\xi_s \eta - \eta}{(p - \xi_s) \xi_1 \dotsm \hat{\xi}_s \dotsm \xi_{N - 1} + q\eta} \prod^{N - 1}_{k = 1} \frac{p + q\xi_k \xi_s - \xi_k}{p + q\xi_k \xi_s - \xi_s}  \right)^{-1}\\
& \times \frac{z\xi^{L - 1}_j \left(p \xi_1 \dotsm \xi_{N - 1} + q\xi_j \eta - \eta\right) \left(p + q \xi_n \xi_j - \xi_n\right)}{C(\xi_n^{(j)} - \tilde{\xi}_n^{(j)})} \prod^{N - 1}_{\substack{k = 1\\ k \neq n}} \frac{p + q \xi_k \xi_j - \xi_k}{p + q \xi_k \xi_j - \xi_j}
\end{split}
\end{equation}
with the integral with respect to $\xi_n$ removed and all $\xi_n$ evaluated as $\xi^{(j)}_n$, since $f(\xi_1, \dots, \xi_{N-1},\eta ) = 0$ is the defining equation of $\xi^{(j)}_n$. In particular, note that the exponent of $\xi_j$ of the integrand in \eqref{eq:first_res} is $x_{\sigma^{-1}(j)}- y_j - x_N + y_N + L -2$. Then, by an inequality similar to \eqref{eq:pole_ineq}, we have the integrand in \eqref{eq:first_res} is analytic at the origin with respect to $\xi_j$.

Now, we integrate \eqref{eq:first_res} with respect to $\xi_j$, and we have to consider $n-1$ residues. We have the residue at $\xi_j = \xi^{(n, i)}_j$, for $i \in \{ 1, \dotsc, n  \} \setminus \{ j \}$, given by the solution to
\begin{equation} \label{eq:defn_xi^ni_j}
  1 - z\xi^{L - 1}_i  \frac{p \xi_1 \dotsm \xi_{N - 1} + q\xi_i \eta - \eta}{(p - \xi_i) \xi_1 \dotsm \hat{\xi}_i \dotsm \xi_{N - 1} + q\eta} \prod^{N - 1}_{k = 1} \frac{p + q \xi_k \xi_i - \xi_k}{p + q \xi_k \xi_i - \xi_i} = 0,
\end{equation}
with respect to $\xi_i$, taking $\xi_1, \dotsc, \hat{\xi}_j, \dotsc, \hat{\xi}_n, \dotsc, \xi_{N - 1}, \eta$ as fixed parameters and $\xi_n = \xi^{(j)}_n$ as an analytic function depending on $\xi_j$. Note that the residue at $\xi_j = \xi^{(n, i)}_{j}$ is an $(N - 2)$-fold contour integral with respect to the variables $\xi_1, \dotsc, \hat{\xi}_j, \dotsc, \hat{\xi}_n, \dotsc, \xi_{N - 1}, \eta$. Moreover, by the defining equations \eqref{eq:defn_xi^j_n_2} and \eqref{eq:defn_xi^ni_j}, we have
\begin{equation} \label{eq:xi^(nj)_n-1=xi_j}
  \xi^{(n, i)}_j = \xi_i,
\end{equation}
for $i \in \{ 1, \dotsc, n - 1 \} \setminus \{ j \}$, and 
\begin{equation} \label{eq:xi^(nn)_n-1=xi^n-1_n}
  \xi^{(n, n)}_j = \xi^{(j)}_n.
\end{equation}
\begin{rem} \label{rem:distinct_radii}
  We let $r$ be a small enough constant, and $r_1 > r_2 > \dotsb > r_N$ be around $r$, say, $0.9 r < r_N < r_1 < 1.1 r$, and then let $\lvert \xi_i \rvert = r_{\sigma^{-1}(i)}$, so that the contours are not through the poles.
\end{rem}
Now, we pair up different residues to cancel them out. In particular, we denote the double residue at $\xi_n = \xi_n^{(j)}$ ($j \in \{1, \dots, n-1\}$) and at $\xi_j = \xi_j^{(n,i)}$ ($i \in \{ 1, \dots, n\} \setminus \{ j\} $) by 
\begin{equation}
	\underset{(\xi_n  , \xi_j  ) = \left(\xi_n^{(j)} , \xi_j^{(n,i)}\right) } {\text{Res}} \mathcal{I}_{\sigma}^{(B)}(z; X;Y)
\end{equation}
with $\mathcal{I}_{\sigma}^{(B)}(z; X;Y)$ denoting the integrand on the right side of \eqref{eq:sum_I^(B)}. We note that this residue vanishes if $\sigma^{-1}(i) < \sigma^{-1}(j)$ since the pole $\xi^{(n, j)}_j$ is outside of the contour for $\xi_j$, but otherwise the residue is nontrivial. Then, we have two cases:
\begin{enumerate}[label=(\alph*)]
\item \label{enu:cancellation_of_perms_1}
  if $i, j \in \{ 1, \dotsc, n - 1 \}$, 
  \begin{equation}
    \underset{(\xi_n  , \xi_j  ) = \left(\xi_n^{(j)} , \xi_j^{(n,i)}\right) } {\text{Res}} \mathcal{I}_{\sigma}^{(B)}(z; X;Y)= - \underset{(\xi_n  , \xi_i ) = \left(\xi_n^{(i)} , \xi_i^{(n,j)}\right) } {\text{Res}} \mathcal{I}_{\sigma'}^{(B)}(z; X;Y)
  \end{equation}
  for $\sigma , \sigma'  \in S^{\varphi}_{A, B}$ so that  $\sigma' = (ij) \circ \sigma$; and
\item \label{enu:cancellation_of_perms_2}
  if $j \in \{ 1, \dotsc, n - 1 \}$, 
  \begin{equation}
    \underset{(\xi_n  , \xi_j  ) = \left(\xi_n^{(j)} , \xi_j^{(n,n)}\right) }{\text{Res}} \mathcal{I}_{\sigma}^{(B)}(z; X;Y)= - \underset{(\xi_n  , \xi_j  ) = \left(\xi_n^{(j)} , \xi_j^{(n,n)}\right) }{\text{Res}}   \mathcal{I}_{\sigma''}^{(B)}(z; X;Y)
  \end{equation}
  for $\sigma ,\sigma'' \in S^{\varphi}_{A, B}$ so that $\sigma'' = (jn) \circ \sigma$.
\end{enumerate}
These cases follow by an argument similar to the argument given in the proof of \cite[Sublemma 2.3]{Tracy-Widom11} by Tracy and Widom. In particular for case \ref{enu:cancellation_of_perms_1} above, when we compare the residue of $\mathcal{I}_{\sigma}^{(B)}(z; X;Y)$ at $(\xi_n  , \xi_j  ) = (\xi_n^{(j)} , \xi_j^{(n,i)})$ and the residue of $\mathcal{I}_{\sigma'}^{(B)}(z; X;Y)$ at $(\xi_n  , \xi_i  ) = (\xi_n^{(i)} , \xi_i^{(n,j)})$, we have that all of the terms in the resulting $(N-2)$-fold contour integral will be the same except for possibly the terms $A_{\sigma}$ and $A_{\sigma'}$. Then, in fact, we have
	\begin{equation}\label{eq:a_minus}
	A_{\sigma}\Big|_{(\xi_n  , \xi_j  ) = (\xi_n^{(j)} , \xi_j^{(n,i))})}= - A_{\sigma'}\Big|_{(\xi_n  , \xi_i  ) = (\xi_n^{(i)} , \xi_i^{(n,j))})}
	\end{equation}
by \eqref{eq:xi^(nj)_n-1=xi_j}. The full argument for \eqref{eq:a_minus} breaks down into smaller sub-cases depending on the relative positions of the indices $i, j$ of the residues. For sake of brevity, we refer the reader to the full argument in the proof of \cite[Sublemma 2.3]{Tracy-Widom11}. Also, we have that case \ref{enu:cancellation_of_perms_2} above follow by a similar argument.

Hence, we conclude that all these $(N - 2)$-fold contour integrals cancel out, and we only need to consider the sum of all residues at $\xi_n = 0$ that are associated to all $\sigma \in S^{\varphi}_{A, B}$.

Now, in order to carry out the induction argument, we consider
\begin{multline} \label{eq:sum_I^(Bn)}
  \sum_{\sigma \in S^{\varphi}_{A, B}} I^{(B \cup \{ n \})}_{\sigma}(z; X; Y) = \sum_{\sigma \in S^{\varphi}_{A, B}} \dashint d\xi_1 \dotsi \dashint d\xi_{N - 1} \dashint d\eta A_{\sigma} \prod^{N - 1}_{j = 1} \xi^{x_j - y_{\sigma(j)} - x_N + y_N - 1}_{\sigma(j)} \eta^{x_N - y_N - 1}_N \\
  \times \prod^{n - 1}_{j = 1} \left( 1 - z \xi^L_j \prod^{N - 1}_{k = 1} \frac{p + q\xi_k \xi_j - \xi_k}{p + q\xi_k \xi_j - \xi_j} \frac{p\xi_1 \dotsm \xi_{N - 1} + q\xi_j \eta - \eta}{(p - \xi_j) \xi_1 \dotsm \hat{\xi}_j \dotsm \xi_{N - 1} + q\eta} \right)^{-1}.
\end{multline}
We evaluate $I^{(B \cup \{ n \})}_{\sigma}(z; X; Y)$ similar to how we just evaluated $I^{(B)}_{\sigma}(z; X; Y)$ above. First, we integrate over $\xi_n$ first. Again, there are $n$ possible poles: $\xi_n = 0$ and $\xi_n = \xi^{(j)}_n$ ($j = 1, \dotsc, n - 1$) which are defined by \eqref{eq:defn_xi^j_n}. Note that $j =n$ is not a pole. Then, we proceed to compute the residue at $\xi_n =\xi^{(j)}_n$, and find that the residue is an $(N - 1)$-fold contour integral with integral variables $\xi_1, \dotsc, \xi_{n - 1}, \xi_{n + 1}, \dotsc, \xi_N, \eta$. For the residue at $\xi_n =\xi^{(j)}_n$, we further integrate over $\xi_j$, and find that the poles are $\xi_j= \xi^{(n, i)}_j$ for $i \in \{ 1, \dotsc, n \} \setminus \{j \}$. Here, the locations $\xi^{(n, i)}_j$ are the same as those defined above for $I^{(B)}_{\sigma}(z; X; Y)$ in \eqref{eq:defn_xi^ni_j}. Also like in the case for $I^{(B)}_{\sigma}(z; X; Y)$, the residue at $\xi^{(n, i)}_j$ is expressed by an $(N - 2)$-fold contour integral with respect to the variables $\xi_1, \dotsc, \hat{\xi}_j, \dotsc, \hat{\xi}_n, \dotsc, \xi_{N - 1}, \eta$, and either $\xi^{(n, n)}_j = \xi_n^{(j)}$ or $\xi^{(n, i)}_j = \xi_i$ (for $i \neq n$) in the integrand. We note that, if we sum over all $\sigma \in S^{\varphi}_{A, B}$, then the sum of the $(N - 2)$-fold contour integrals for the residue at $\xi^{(n, i)}_j$ cancel in pair, so the right side of \eqref{eq:sum_I^(Bn)} is equal to the sum of all residues at $\xi_n = 0$ that are associated to all $\sigma \in S^{\varphi}_{A, B}$. This is similar to the conclusion for the right side of \eqref{eq:sum_I^(B)}, but in this case, we evaluate the residue at $\xi_n$ for the integral $I^{(B \cup \{ n \})}_{\sigma}(z; X; Y)$ instead of the integral $I^{(B)}_{\sigma}(z; X; Y)$.

Now, we compare the residues at $\xi_n = 0$ for $I^{(B \cup \{ n \})}_{\sigma}(z; X; Y)$ and $I^{(B)}_{\sigma}(z; X; Y)$. For any $\sigma \in S^{\varphi}_{A, B}$, the residue at $\xi_n = 0$ for $I^{(B \cup \{ n \})}_{\sigma}(z; X; Y)$ is the same as the residue at $\xi_n = 0$ for $I^{(B)}_{\sigma}(z; X; Y)$. The special case with $\lvert B \rvert = N - 2$ is given by \eqref{eq:|B|=N-2}, which we already showed to be true. Then, the general case is proved similarly, and we omit the detail. Hence, we have
\begin{equation}
  \sum_{\sigma \in S^{\varphi}_{A, B}} I^{(B)}_{\sigma}(z; X; Y) = \sum_{\sigma \in S^{\varphi}_{A, B}} I^{(B \cup \{ n \})}_{\sigma}(z; X; Y).
\end{equation}
Also, note that
\begin{equation}
  S^{\varphi}_{A, B} = \bigcup^n_{k = 1} S^{\varphi_k}_{A \cup \{ k \}, B \cup \{ n \}},
\end{equation}
with the mapping $\varphi_k: A \cup \{ k \} \to B \cup \{ n \}$ defined by $\varphi_k(i) = i$ if $i \in A$ and $\varphi_k(k) = n$. Then,
\begin{equation}
  \sum_{\sigma \in S^{\varphi}_{A, B}} I^{(B \cup \{ n \})}_{\sigma}(z; X; Y) = \sum^n_{k = 1} \sum_{\sigma \in S^{\varphi_k}_{A \cup \{ k \}, B \cup \{ n \}}} I^{(B \cup \{ n \})}_{\sigma}(z; X; Y).
\end{equation}
Moreover, by the inductive assumption, we have
\begin{equation}
  \sum_{\sigma \in S^{\varphi_k}_{A \cup \{ k \}, B \cup \{ n \}}} I^{(B \cup \{ n \})}_{\sigma}(z; X; Y) = \sum_{\sigma \in S^{\varphi_k}_{A \cup \{ k \}, B \cup \{ n \}}} \dashint d\xi_1 \dotsi \dashint d\xi_N A_{\sigma} \prod^N_{j = 1} \xi^{x_j - y_{\sigma(j)} - 1}_{\sigma(j)}.
\end{equation}
Therefore, we have
\begin{equation}
  \sum_{\sigma \in S^{\varphi}_{A, B}} I^{(B)}_{\sigma}(z; X; Y) = \sum_{\sigma \in S^{\varphi}_{A, B}} \dashint d\xi_1 \dotsi \dashint d\xi_N A_{\sigma} \prod^N_{j = 1} \xi^{x_j - y_{\sigma(j)} - 1}_{\sigma(j)},
\end{equation}
and this completes the proof for Lemma \ref{lem:ASEP_IC}.

\section{Formulas for TASEP on the Ring and ASEP on the Line}\label{sec:limits}

We use the formula introduced in Theorem~\ref{thm:main} to recover transition probability formulas for the TASEP on the ring obtained originally by Liu and Baik \cite{Baik-Liu18} and for the ASEP on the line obtained originally by Tracy and Widom \cite{Tracy-Widom08}.

\paragraph{TASEP on the Ring}

In \cite{Baik-Liu18}, the authors give contour integral formulas for the TASEP on the ring. In particular, the integrands in \cite{Baik-Liu18} depend on the solutions to a system of algebraic equations, called the \emph{Bethe equations}. In the case of the TASEP, the Bethe system of equations become decoupled, as opposed to the Bethe equations (see \eqref{eq:bth} in Section \ref{sec:root_formula}) for the ASEP with generic $p \neq 1$. We denote the set of \emph{Bethe roots} by
\begin{equation}\label{eq:bethe_roots}
  R_{z} := \{w \in \mathbb{C} | z^L =  w^N(w +1)^{L-N}  \}
\end{equation}
for any $z \in \mathbb{C}$. 
\begin{corollary}
  If $p=1$, then the transition probability function~\eqref{eq:transition_probability} is equivalent to the transition probability function for the TASEP on the ring  in \cite[Proposition 5.1]{Baik-Liu18}, that is,
  \begin{equation}
    \prob_Y(X;t) = \oint_0 \frac{dz}{2\pi i z} \det\left[\sum_{w \in R_z} \frac{w^{j-k+1}(1+w)^{-x_k+y_j+k-j}e^{tw}}{L(w+\rho)}\right]^N_{j, k = 1},
  \end{equation}
  with $\rho = N/L$.
\end{corollary}

\begin{proof}
  If we take $p=1$ in~\eqref{eq:transition_probability}, we have the transition probability of TASEP in $\confs_N(L)$. We let $K$ be a positive integer such that $-K(L - N) \leq y_1 - x_N - N$. Recall the notation $\min(\Lattice) = \min(\ell_i)$ for $\Lattice = (\ell_1, \dotsc, \ell_N) \in \mathbb{Z}^N$ and the indices $m , M \in \{ 1, \dots, N\}$ defined in \eqref{eq:defn_mM}. Then, by Theorem \ref{thm:main}, we have
  
  \begin{equation}
    \begin{split}
      \prob_Y(X;t)
      = {}& \sum_{\Lattice \in \intZ^N(0)}\sum_{\sigma\in S_N}\sgn(\sigma) \prod^N_{j = 1} \left( \dashint_C \xi^{x_{\sigma^{-1}(j)} - y_j + \ell_j L - 1} (1 - \xi)^{j - \sigma^{-1}(j) -\ell_j N} e^{t(\xi^{-1} - 1)} d \xi \right) \\
      = {}& \sum_{\substack{\Lattice \in \intZ^N(0) \\ \min(\Lattice) \geq -K}} \sum_{\sigma\in S_N}\sgn(\sigma) \prod^N_{j = 1} \left( \dashint_C \xi^{x_{\sigma^{-1}(j)} - y_j + \ell_j L - 1} (1 - \xi)^{i - \sigma^{-1}(j) -\ell_j N} e^{t(\xi^{-1} - 1)} d\xi \right) \\
      = {}& \sum_{\sigma\in S_N}\sgn(\sigma) \oint_0 \frac{dz}{2\pi i z} \prod^N_{j = 1} \left( \dashint_C \xi^{x_{\sigma^{-1}(i)} - \sigma^{-1}(i) - (y_i - i)} e^{t(\xi^{-1} - 1)} \frac{\left( \frac{(1 - \xi)^N}{\xi^L z^L} \right)^K}{1 - \frac{\xi^L}{(1 - \xi)^N} z^L} d\xi \right),
    \end{split}
  \end{equation}
  with the integral with respect to $z$ over a small enough circle centered at zero such that $\lvert (z \xi)^L_i/(1 - \xi_i)^N \rvert < 1$ always holds. Note that the integral with respect to $\xi_m$ vanishes if $\min(\Lattice) < -K$. Thus, the restriction $\min(\Lattice) \geq -K$ in the second line does not affect the summation.
  
  Now, we simplify the summation over $S_N$ as a determinant. We have
  \begin{equation} \label{eq:det_trans_ASEP}
    \prob_Y(X;t) = \oint_0 \frac{dz}{2\pi i z} \det\left[ \dashint_C \xi^{x_k - y_j - 1}(1 - \xi)^{j - k} e^{t(\xi^{-1} - 1)} \frac{\left( \frac{(1 - \xi)^N}{\xi^L z^L} \right)^K}{1 - \frac{\xi^L}{(1 - \xi)^N} z^L} d \xi \right]_{j, k =1}^N.
  \end{equation}
  Also, we introduce the change of variables 
  \begin{equation} \label{eq:TASEP_change_var}
    \xi \rightarrow \frac{1}{w +1},
  \end{equation}
  and we have
  \begin{equation}
    \prob_Y(X;t) = \oint_0 \frac{dz}{2\pi i z} \det\left[ \dashint_{\lvert w \rvert = R} w^{j - k} (1 + w)^{y_j - x_k + k - j - 1} e^{wt} \frac{(w^N(1 + w)^{L - N} z^{-L})^K}{z^L(1 + w)^{-L + N} w^{-N} - 1} dw \right]_{j, k =1}^N,
  \end{equation}
so that we may take the contour $|w| = R$ with $R$ is large enough since we can take $|\xi_j|$ to be arbitrarily small. It is straightforward to convert the contour integral over $C$ into a discrete sum by using residue computations. 
  \begin{multline} \label{eq:residue_TASEP_trans}
    \dashint_{\lvert w \rvert = R} w^{j - k} (1 + w)^{y_j - x_k + k - j - 1} e^{wt} \frac{(w^N(1 + w)^{L - N} z^{-L})^K}{z^L(1 + w)^{-L + N} w^{-N} - 1} dw = \\
    \sum_{w \in R_z} \frac{w^{j-k+1}(1+w)^{-x_k+y_j+k-j}e^{tw}}{L(w+\rho)}.
  \end{multline}
  Hence we prove the result.
\end{proof}

\paragraph{ASEP on the Line}

\begin{corollary}
	In the limit as the size of the ring goes to infinity, the transition probability function~\eqref{eq:transition_probability} for the ASEP on the ring is equal to the transition probability function for the ASEP on the line that was obtained in \cite{Tracy-Widom08}. That is,
	\begin{equation}
	\lim_{L \rightarrow \infty} \sum_{\Lattice  \in \intZ^N(0)} u_{Y}^{\Lattice}(X;t) = u_{Y}^{(0, \dots, 0)}(X;t)
	\end{equation}
\end{corollary}

\begin{proof}
  It suffices to show that
  \begin{equation}
    \lim_{L \rightarrow \infty} \Lambda_Y^{\Lattice} (X; t ; \sigma) = 0,
  \end{equation}
  for any $\Lattice \in \intZ^N(0) \setminus \{ (0, \dots, 0) \}$ and the convergence becomes exponentially fast as $\max(\Lattice) \to \infty$. By Lemma~\ref{lem:ASEP_conv_tech}, we have that
  \begin{equation}
    \left|  \Lambda_Y^{\Lattice} (X; t ; \sigma) \right| \leq \frac{1}{[\max(\Lattice)L/2]!}
\end{equation}
for large enough $L$. The desired result then follows Lemma \ref{lem:finiteness_of_large_max(L)}.
\end{proof}

\section{One point distribution}\label{sec:one_pt}

Our strategy is to prove the $m = 1$ case of Theorems \ref{thm:one_pt_ASEP} and \ref{thm:one_pt_qTAZRP}, and then use the cyclic invariance property to generalize the result to all $m$.

\subsection{Distribution of $x_1(t)$ for ASEP} \label{subsec_1pt_ASEP}

First, we compute $\prob_Y(x_1(t) = a)$, for some $a \in \intZ$, by taking the sum of the joint probability function $\mathbb{P}_Y(X;t)$ over all configurations $X \in \confs_N(L)$ with $x_1(t) =a$, and then sum over all $a \geq M$ to obtain $\prob_Y(x_1(t) \geq  M)$. The main obstacle in this strategy to compute the one-point function is the summation over all configurations $X \in \confs_N(L)$ with $x_1(t) =a$, which is more difficult that the analogous case for the ASEP on $\intZ$. In the periodic ASEP, we overcome this obstacle through Lemma \ref{lem:Baik-Liu}, introduced in \cite{Baik-Liu18a}. For the ASEP on $\intZ$, the proof of the one-point function of $x_1$ was due to Lemma \ref{lem:Tracy_Widom_09}, introduced in \cite{Tracy-Widom08}. 

In this subsection, the contour integral $\dashint d\xi_i$ is over a counterclockwise contour $\{ \lvert z \rvert = r \}$ with a small enough $r$ unless otherwise stated. Also, we denote
\begin{equation}
  \tilde A_\sigma(\xi_1,\cdots,\xi_n) := \sgn(\sigma) \prod_{1 \leq i < j \leq n} \left(p+q\xi_{\sigma(i)}\xi_{\sigma(j)}-\xi_{\sigma(i)}\right) \prod^N_{j = 1} \xi_{\sigma(j)}^j
\end{equation}
for $\sigma \in S_N$, and 
\begin{equation} \label{eq:defn_G_ASEP}
  G(J_1, \dotsc, J_k) = \prod_{1 \leq i < j \leq k} \left( \prod_{\alpha \in J_i,\ \beta \in J_j} (p + q \xi_{\alpha} \xi_{\beta} - \xi_{\alpha}) \prod_{\substack{\alpha > \beta \\   \alpha \in J_i,\ \beta \in J_j}  } (-1) \right)
\end{equation}
for $J_1, \dotsc, J_k$ disjoint finite subsets of positive integers.

By Theorem \ref{thm:main} and expressions \eqref{eq:int_Lambda_ASEP} and \eqref{eq:expr_A_sigma_ASEP}, we write the one-point probability function
\begin{equation} \label{eq:aux_003}
  \begin{split}
    \prob_Y(x_1(t) = a) = {}& \sum_{X\in\confs_N(L),\ x_1=a} \prob_Y(X;t) \\
    = {}& \sum_{\Lattice \in\intZ^N(0)}\dashint d\xi_1 \dotsi \dashint d\xi_N D_{\Lattice}(\xi_1,\cdots,\xi_N) \sum_{\substack{X\in\confs_N(L) \\ x_1 = 0}} C(\xi_1, \dotsc, \xi_N; X) \\
    & \times \prod_{1 \leq i < j \leq N}(p+q\xi_i\xi_j-\xi_i)^{-1} \prod^N_{k = 1} \xi^{a - y_k - 1}_k e^{\epsilon(\xi_k)t}, 
  \end{split}
\end{equation}
with $D_{\Lattice}(\xi_1,\cdots,\xi_N)$ defined in \eqref{eq:HK} and
\begin{equation}
  C(\xi_1, \dotsc, \xi_N; X) = \sum_{\s\in S_N} \tilde A_{\sigma}(\xi_1,\cdots,\xi_N) \prod_{j=1}^N\xi_{\sigma(j)}^{x_j-j}.
\end{equation}

First, we evaluate the summation over $C(\xi_1, \dotsc, \xi_N; X)$ in \eqref{eq:aux_003} with the following lemma.
\begin{lemma}[\protect{\cite[Lemma 5.3]{Baik-Liu18a}}] \label{lem:Baik-Liu}
  Let $F_{m,n}=\prod_{j=m}^n \xi_j$ for all $m\le n$. Then
  \begin{multline} \label{eq:Baik_Liu17}
    \sum_{X\in\confs_N(L),\ x_1 = 0} \left( \prod_{j=1}^N \xi_j^{x_j-j} \right) = \\
    \sum_{k=0}^{N-1} (-1)^k \sum_{1 = s_0 <s_1<\cdots<s_k < s_{k + 1} = N + 1} \frac{1 - F_{1, s_1 - 1}}{F_{1, s_1 - 1}} (F_{s_1, N})^{L - N} \prod^k_{i = 0} \frac{1}{\prod_{j=s_i}^{s_{i + 1}-1}(1-F_{j,s_{i + 1}-1})}.
  \end{multline}
\end{lemma}
We assume $F_{m, m - 1} = 1$ here and in similar cases below.

By changing the order of summation and applying Lemma \ref{lem:Baik-Liu}, we have
\begin{equation} \label{eq:aux_004}
  \sum_{X\in\confs_N(L),\ x_1 = 0} C(\xi_1, \dotsc, \xi_N; X) = \sum_{k=0}^{N-1} (-1)^k \sum_{1  < s_1 < \dotsb < s_k <  N + 1} C_{s_1, \dotsc, s_k}(\xi_1, \dotsc, \xi_N),
\end{equation}
with
\begin{multline}
  C_{s_1, \dotsc, s_k}(\xi_1, \dotsc, \xi_N) = \\
  \sum_{\sigma\in S_N} \left( \tilde A_{\sigma}(\xi_1,\cdots,\xi_N) \frac{1 - F^{\sigma}_{1, s_1 - 1}}{F^{\sigma}_{1, s_1 - 1}} (F^{\sigma}_{s_1, N})^{L - N} \prod^k_{i = 0} \frac{1}{\prod_{j=s_i}^{s_{i + 1}-1}(1-F^{\sigma}_{j,s_{i + 1}-1})}\right), 
\end{multline}
and $F^{\sigma}_{m,n}:=\prod_{j=m}^n\xi_{\sigma(j)}$. If $k = 0$, then $s_1 = N + 1$ in the expression of $C_-$.

Next we simplify the expression of $C_{s_1, \dotsc, s_k}(\xi_1, \dotsc, \xi_N)$ for fixed $k = 0, 1, \dotsc, N - 1$ and $1 = s_0 <s_1 < \dotsb < s_k < s_{k + 1} = N + 1$. Each permutation $\sigma \in S_N$ is uniquely determined by the restriction of the permutation $\sigma$ to the sets $\{1, \dotsc, s_1 - 1 \}$ and $\{ s_1, \dotsc, N \}$, separately. Moreover, the restriction to the latter set is determined by the restriction of the permutation $\sigma$ to the sets $\{ s_j, s_j +1, \dots, s_{j+1}-1\}$ for each $j = 1, \dotsc, k$. We use the notation
	\begin{equation}
	f(\xi_{I}) = f(\xi_{i_1} , \xi_{i_2}, \dots , \xi_{i_s})  \hspace{5mm} \text{or} \hspace{5mm} f(\xi_{\sigma (I)}) = f(\xi_{i_{\sigma(1)}} , \xi_{i_{\sigma(2)}}, \dots , \xi_{i_{\sigma(s)}})
	\end{equation}
for any subset $I \subset \mathbb{N}$ with $|I| =s$ and $I = \{ i_1 < i_2 < \cdots < i_s \}$, for any function $f$ on $s$ arguments, and for any $\sigma$ a permutation of the set $\{1, 2, \dots, s \}$. For example, we have 

	\begin{equation}
	F_{j , |I|}(\xi_{\sigma(I)}) = \prod_{k = j}^{|I|} \xi_{i_{\sigma(k)}}
	\end{equation}

for $I = \{ i_1 < i_2 < \cdots < i_{|I|} \}$. We introduce the functions

\begin{equation}
  E(I) = \left( \prod_{\alpha \in I} \frac{1}{\xi_{\alpha}} - 1 \right) \sum_{\lambda \in S_{\lvert I \rvert}}  \frac{\tilde{A}_{\lambda}(\xi_{I})}{\prod^{\lvert I \rvert}_{j = 1} (1 - F_{j, \lvert I \rvert}(\xi_{\lambda(I)}))},
\end{equation}
and
\begin{equation}
  D_{s_1, \dotsc, s_k}(I^c) = \sum_{\tau \in S_{N - \lvert I \rvert}} \tilde{A}_{\tau}(\xi_{I^c}) \prod^k_{i=1} \frac{\left( F_{s_i,s_{i+1}-1}(\xi_{\tau(I^c)})\right)^{L-N + s_1 - 1}}{\prod_{j=s_i+1}^{s_{i+1}-1}(1-F_{j,s_{i+1}-1}(\xi_{\tau(I^c)}))}.
\end{equation}
for fixed $k = 0, 1, \dotsc, N - 1$, $1 = s_0 <s_1 < \dotsb < s_k < s_{k + 1} = N + 1$, and any subset $I \subset \{1, \dots, N \}$ with $|I|=s_1-1$ and $I^c = \{ 1, \dots, N\} \setminus I$. In the case $k = 0$ and $s_1 = N + 1$, we assume that $D_-(\emptyset) = 1$. Then, we write

\begin{equation} \label{eq:form_C_s}
  C_{s_1, \dotsc, s_k}(\xi_1, \dotsc, \xi_N) = \sum_{\substack{|I| =s_1 -1 \\ I \subset \{1, \dots, N \}}} G(I, I^c) D_{s_1, \dotsc, s_k}(I^c) E(I),
\end{equation}
so that the sum is over all subsets $I \subset \{1, \dots, N \}$ with $|I| = s_1 -1$ and $I^c = \{1 , \dots, N\} \setminus I$.

We simplify $E(I)$ and $D_{s_1, \dotsc, s_k}(I^c)$ with the following lemma. Additionally, for notational convenience, we introduce 
\begin{equation}
  R^{(m)}_n(z_1, \dotsc, z_n) = p^{n(n - 1)/2} \prod^n_{i = 1} z^m_i \prod_{1 \leq i < j \leq n} (z_j - z_i), \quad \text{and} \quad R^{(m)}_0(-) = 1
\end{equation}
for $n$ complex variables $z_1, \dotsc, z_n$.

\begin{lemma}[Equation (1.6) in \cite{Tracy-Widom08}]\label{lem:Tracy_Widom_09}
  \begin{equation}
    \sum_{\s\in S_n} \frac{ \tilde{A}_{\sigma}(\xi_1, \dotsc, \xi_n)}{\prod^n_{j = 1} (1- F^{\sigma}_{j, n})} = R^{(0)}_n(\xi_1, \dotsc, \xi_n) \prod^n_{j = 1} \frac{\xi_j}{1-\xi_j}, \quad \text{with} \quad F^{\sigma}_{j, n} = \prod^n_{l = j} \xi_{\sigma(l)}.
  \end{equation}
\end{lemma}
Then, we have
\begin{equation}\label{eq:E_simple}
  E(I) = R^{(0)}_{\lvert I \rvert} (\xi_{i_1}, \dotsc, \xi_{i_{\lvert I \rvert}}) \prod_{\alpha \in I} \frac{1}{1 - \xi_{\alpha}} \left( 1 - \prod_{\alpha \in I} \xi_{\alpha} \right)
\end{equation}
with the help of Lemma \ref{lem:Tracy_Widom_09}.

Next, we simplify $D_{s_1, \dotsc, s_k}(I^c)$ by taking the summation over all possible partitions $J_1 \cup \cdots \cup J_k$ of $I^c$ with $\lvert J_l \rvert = s_{l + 1} - s_l$. We have

\begin{equation}
  D_{s_1, \dotsc, s_k}(I^c) = \sum_{\substack{|J_l| = s_{l+1} -s_l \\ J_1 \sqcup \dots \sqcup J_k = I^c}} G(J_1, \dotsc, J_k)  \prod^k_{l = 1} \left( \sum_{\lambda^{(l)} \in S_{\lvert J_l \rvert}}  \frac{ \tilde{A}_{\lambda^{(l)}}(\xi_{J_l}) F_{1, \lvert J_l \rvert}(\lambda^{(l)}(J_l))^{L - N + s_l - 1}}{\prod^{\lvert J_l \rvert}_{j = 1} \left(1 - F_{j, \lvert J_l \rvert}(\lambda^{(l)}(J_l) )\right)} \right).
\end{equation}

 Furthermore, by applying Lemma \ref{lem:Tracy_Widom_09} to each summation over $\lambda^{(l)}$, we have

\begin{equation}\label{eq:D_simple}
  D_{s_1, \dotsc, s_k}(I^c) = \prod_{\beta \in I^c} \xi^{L - N + s_1}_{\beta} \frac{1}{1 - \xi_{\beta}}  \sum_{\substack{|J_l| = s_{l+1} -s_l \\ J_1 \sqcup \dots \sqcup J_k = I^c}} G(J_1, \dotsc, J_k) \prod^k_{l = 1}  R^{(s_l - s_1)}_{\lvert I_l \rvert}(\xi_{J_l}) .
\end{equation}

At this point, equations \eqref{eq:E_simple} and \eqref{eq:D_simple} seem to give us the simplest expression for $C_{s_1, \dotsc, s_k}(\xi_1, \dotsc, \xi_N)$ in \eqref{eq:form_C_s} with fixed $1 < s_1< \dotsc< s_k \leq N$. So, we now consider and simplify the summation of the terms $C_{s_1, \dotsc, s_k}(\xi_1, \dotsc, \xi_N)$ over all possible $k = 0, 1, \dotsc, N - 1$ and $1 < s_1< \dotsc< s_k \leq N$.

As an intermediate step, we fix $s_1 \leq N$ and sum over indices $s_2, \dotsc, s_k$, or equivalently, sum over all possible partitions of $I^c$. We let $s_1 = s + 1$. That is, take $\lvert I \rvert = s$ and $\lvert I^c \rvert = N - s$. We set $I = \{ i_1 < \dotsb < i_s \}$ and $I^c = \{ j_1 < \dotsb < j_{N - s} \}$. The index $k$ may vary between $1$ and $N - s$. Then, we have
\begin{equation} \label{eq:sum_D_s}
  \sum^{N - s}_{k = 1} (-1)^k \sum_{s + 1 = s_1 < s_2 < \dotsb < s_k \leq N} D_{s_1, \dotsc, s_k}(I^c) = (-1)^{N - s} \tau^{\frac{(N - s)(N - s - 1)}{2}} R^{(L)}_{N - s}(\xi_{I^c}) \prod_{\beta \in I^c} \frac{1}{1 - \xi_{\beta}},
\end{equation}
with $\tau = q/p$ as defined in Section \ref{subsec:notation}. This is a consequence of the following lemma whose proof is given in the end of this subsection in Section \ref{sec:misc_proof1}.
\begin{lemma} \label{lem:sum_D(I^c)}
  For any positive integer $n$, we have

  \begin{equation} \label{eq:identity}
    \sum^n_{k = 1} (-1)^k \sum_{I_1 \sqcup \cdots \sqcup I_k = \{1, \dots, n \}} G(I_1, \dotsc, I_k) \prod^k_{\alpha=1} R^{(\sum^{\alpha - 1}_{i = 1} \lvert I_i \rvert)}_{\lvert I_{\alpha} \rvert}( \xi_{I_{\alpha}})   = (-1)^n \tau^{\frac{n(n - 1)}{2}} R^{(n - 1)}_n(\xi_1, \dotsc, \xi_n)
  \end{equation}

so that the inner-most sum on the left side of \eqref{eq:identity} is over all disjoint sets $I_1, \dots, I_k \subset \{1, \dots, n\}$ with $I_1 \cup \cdots  \cup I_k = \{ 1, \dots, n\}$.
\end{lemma}

Then, using \eqref{eq:sum_D_s} and \eqref{eq:form_C_s} in \eqref{eq:aux_004}, we simplify the summation in \eqref{eq:aux_003} as

\begin{multline}
  \sum_{\substack{X \in \confs_N(L)\\ x_1 = 0}} C(\xi_1, \dotsc, \xi_N; X) = \\
 \prod^N_{j = 1} \frac{1}{1 - \xi_j} \sum^N_{s = 1} \sum_{\substack{|I|=s \\ I \subset \{1, \dots, N \}}} R^{(0)}_s(\xi_{I}) \left( 1 - \prod_{\alpha \in I} \xi_{\alpha} \right) G(I, I^c) (-1)^{N - s} \tau^{\frac{(N - s)(N - s - 1)}{2}} R^{(L)}_{N - s}(\xi_{I^c}).
\end{multline}
 
We note that 
\begin{equation} \label{eq:change_level_id_ASEP}
  G(I, I^c) R^{(L)}_{N - s}(\xi_{I^c}) D_{\Lattice}(\xi_1, \dotsc, \xi_N) = G(I^c, I) R^{(0)}_{N - s}(\xi_{I^c}) D_{\Lattice(I^c)}(\xi_1, \dotsc, \xi_N),
\end{equation}
for any $\Lattice = (\ell_1, \dotsc, \ell_N)$ and $\Lattice(I^c) = (\ell'_1, \dotsc, \ell'_N)$ so that $\ell'_i = \ell_i$ if $i \in I$ and $\ell'_j = \ell_j + 1$ if $j \in I^c$. Also, note that $\Lattice(I^c) \in \intZ^N(N-s)$ if and only if $\Lattice \in \intZ^N(0)$. Hence, we have 

\begin{equation} \label{eq:sum_DC}
  \begin{split}
    &
    \begin{aligned}
      \sum_{\Lattice \in \intZ^N(0)} \dashint d\xi_1 \dotsi \dashint d\xi_N & D_{\Lattice}(\xi_1, \dotsc, \xi_N) \sum_{\substack{X \in \confs_N(L) \\ x_1 = 0}} \frac{C(\xi_1, \dotsc, \xi_N; X)} {\prod_{1 \leq i < j \leq N} (p + q \xi_i \xi_j - \xi_i)} \prod^N_{k = 1} \xi^{a - y_k - 1}_k e^{\epsilon(\xi_k) t}
    \end{aligned} \\
    = {}& \sum^N_{s = 1} (-1)^{N - s}  \tau^{\frac{(N - s)(N - s - 1)}{2}} \sum_{\substack{|I|=s \\ I \subset \{1, \dots, N \}}} \sum_{\Lattice \in \intZ^N(N-s)} \dashint d\xi_1 \dotsi \dashint d\xi_N D_{\Lattice}(\xi_1, \dotsc, \xi_N)    \\
    & \times  \left( 1 - \prod_{\alpha \in I} \xi_{\alpha} \right)\frac{ G(I^c, I)  R^{(0)}_s(\xi_{I})  R^{(0)}_{N - s}(\xi_{I^c}) }{ \prod_{1 \leq i < j \leq N} (p + q \xi_i \xi_j - \xi_i)}\prod^N_{k = 1} \frac{\xi^{a - y_k - 1}_k}{1 - \xi_k} e^{\epsilon(\xi_k) t} \\
    = {}& \sum^N_{s = 1} (-1)^{N - s} p^{\frac{s(s - 1)}{2}} q^{\frac{(N - s)(N - s - 1)}{2}} \sum_{\substack{|I| =s \\ I \subset \{ 1, \dots, N \} }} \sum_{\Lattice \in \intZ^N(N-s)} \dashint d\xi_1 \dotsi \dashint d\xi_N D_{\Lattice}(\xi_1, \dotsc, \xi_N) \\
    & \times \left( 1 - \prod_{\alpha \in I} \xi_{\alpha} \right) \prod_{\alpha \in I, \beta \in I^c} \frac{p + q \xi_{\alpha} \xi_{\beta} - \xi_{\beta}}{\xi_{\beta} - \xi_{\alpha}} \prod_{1 \leq i < j \leq N} \frac{\xi_j - \xi_i}{p + q \xi_i \xi_j - \xi_i} \prod^N_{k = 1} \frac{\xi^{a - y_k - 1}_k}{1 - \xi_k} e^{\epsilon(\xi_k) t}.
  \end{split}
\end{equation}

We remark that for each $s = 1, \dotsc, N$ and fixed $I$ with $\lvert I \rvert = s$, the sum over $\Lattice \in \intZ^N(N-s)$ in \eqref{eq:sum_DC} converges absolutely by Lemma \ref{lm:absolutely_convergent}. We simplify the right side of \eqref{eq:sum_DC} further by the following lemma.
\begin{lemma}[\protect{\cite[Formula (1.9)]{Tracy-Widom08}}]\label{lem:TW09}
  For $s = 1, \dotsc, N$,
  \begin{equation}
    \sum_{\substack{|I| = s \\I \subset \{1, \dots, N \}}}  \left( 1 - \prod_{\alpha \in I} \xi_{\alpha} \right) \prod_{\substack{\alpha \in I \\  \beta \in I^c}} \frac{p + q \xi_{\alpha} \xi_{\beta} - \xi_{\beta}}{\xi_{\beta} - \xi_{\alpha}} =  (-p)^{(N - s)s} (q/p)^{N - s}  \taubinom{N - 1}{s - 1} \left(1-\prod_{j=1}^N\xi_j\right)
  \end{equation}
so that the sum on the left side is over all subsets $I \subset \{1, \dots, N \}$ with $|I| =s$.
\end{lemma}
By applying Lemma \ref{lem:TW09} on the right side of \eqref{eq:sum_DC}, we have

\begin{multline} \label{almost_m=1_ASEP}	
  \sum_{\Lattice \in \intZ^N(0)} \dashint d\xi_1 \dotsi \dashint d\xi_N D_{\Lattice}(\xi_1, \dotsc, \xi_N) \sum_{\substack{X \in \confs_N(L) \\ x_1 = 0}} \frac{C(\xi_1, \dotsc, \xi_N; X)}{ \prod_{1 \leq i < j \leq N} (p + q \xi_i \xi_j - \xi_i)}  \prod^N_{k = 1} \xi^{a - y_k - 1}_k e^{\epsilon(\xi_k) t} \\ 
 = \sum^N_{s = 1} c(s) \sum_{\Lattice \in \intZ^N(N-s)} \dashint d\xi_1 \dotsi \dashint d\xi_N D_{\Lattice}(\xi_1, \dotsc, \xi_N) \\
 \times  \prod_{1 \leq i < j \leq N} \frac{\xi_j - \xi_i}{p + q \xi_i \xi_j - \xi_i} \prod^N_{k = 1} \frac{\xi^{a - y_k - 1}}{1 - \xi_k} e^{\epsilon(\xi_k) t} \left( 1 - \prod^N_{j = 1} \xi_j \right),
\end{multline}

with
\begin{equation}
  c(s) = (-1)^{(N - s)(s - 1)} p^{(2N - s)(s - 1)/2} q^{(N - s)(N - s + 1)/2} \taubinom{N - 1}{s - 1}.
\end{equation}
We note that sum on the left side of equation \eqref{almost_m=1_ASEP} is absolutely convergent for $\Lattice \in \intZ^N(0)$, and on the right side of equation \eqref{almost_m=1_ASEP}, the sum over $\Lattice \in \intZ^N(N-s)$ is absolutely convergent for each $s$ by Lemma \ref{lm:absolutely_convergent} and Lemma \ref{lem:ASEP_conv_tech}.

Lastly, we use the identity
\begin{equation} \label{eq:tau_identity}
  \prod^{N - 1}_{i = 1} (1 + \tau^i x) = \sum^{N - 1}_{k = 0} q^{(k^2 + k)/2} \taubinom{N - 1}{k} x^k,
\end{equation}
found in \cite[Formula (10.0.9)]{Andrews-Askey-Roy99} so that we may write
\begin{equation} \label{eq:gene_fun_ASEP}
\sum^N_{s = 1} c(s) z^{N - s} = p^{N(N - 1)/2} C_N(z),
\end{equation}
with $C_N(z)$ defined in \eqref{eq:defn_C_N_ASEP}.

From \eqref{eq:aux_003}, we recognize that the left side of \eqref{almost_m=1_ASEP} is $p^{-N(N - 1)/2} \prob_Y(x_1(t) = a)$. Then, using identity \eqref{eq:gene_fun_ASEP}, we have
\begin{multline} \label{eq:x_1=a_ASEP}
  \prob_Y(x_1(t)=a) = \frac{p^{N(N - 1)/2}}{2\pi i} \int_0 \frac{dz}{z} C_N(z) \sum^{N - 1}_{k = 0} z^{-k} \sum_{\Lattice \in \intZ^N(k)} \dashint d \xi_1\cdots\dashint d \xi_N D_{\Lattice}(\xi_1,\cdots,\xi_N) \\
  \times \prod_{1 \leq i < j \leq N} \frac{\xi_j - \xi_i}{p + q \xi_i \xi_j - \xi_i} \prod^N_{k = 1}\frac{\xi^{- y_k}_k}{1 - \xi_k} e^{\epsilon(\xi_k) t} \left( \prod^N_{j = 1} \xi^{a-1}_j - \prod^N_{j = 1} \xi^{a}_j \right).
\end{multline}
Hence,
\begin{equation} \label{eq:x_1>=M_ASEP}
  \begin{split}
    \prob_Y(x_1(t) \geq M) = {}& \sum^{\infty}_{a = M} \prob_Y(x_1(t)=a) \\
    = {}& \frac{p^{N(N - 1)/2}}{2\pi i} \int_0 \frac{dz}{z} C_N(z) \sum^{N - 1}_{k = 0} z^{-k} \sum^{\infty}_{a = M} \sum_{\Lattice \in \intZ^N(k)} \dashint d \xi_1\cdots\dashint d \xi_N D_{\Lattice}(\xi_1,\cdots,\xi_N) \\
    & \times \prod_{1 \leq i < j \leq N} \frac{\xi_j - \xi_i}{p + q \xi_i \xi_j - \xi_i} \prod^N_{k = 1}\frac{\xi^{ - y_k }_k}{1 - \xi_k} e^{\epsilon(\xi_k) t} \left( \prod^N_{j = 1} \xi^{a-1}_j - \prod^N_{j = 1} \xi^{a }_j \right).
  \end{split}
\end{equation}
Since the infinite summation over $a \geq M$ and $\Lattice \in \intZ^N(k)$ on the right-hand side of \eqref{eq:x_1>=M_ASEP} is absolutely convergent by Lemma \ref{lm:absolutely_convergent}, we conclude the proof of the $m = 1$ case of Theorem \ref{thm:one_pt_ASEP} by changing the order of summation so that the sum over $a$ is taken first.

\subsubsection{Miscellaneous Proof}\label{sec:misc_proof1}

\begin{proof}[Proof of Lemma \ref{lem:sum_D(I^c)}]
 We prove the lemma by induction on $n$. The $n = 1$ case holds trivially. Suppose the lemma holds if $n$ is replaced by $1, 2, \dotsc, n - 1$. For notational convenience, we write $I_1 = I$ with $\lvert I_1 \rvert = s \geq 1$. Then, the left side of \eqref{eq:identity} becomes
  \begin{equation} \label{eq:intermediate_ASEP_1pt_lemma}
    (-1) \sum^n_{s = 1} \sum_{|I| = s \text{ and } I \subset \{1, \dots, n \}} G(I, I^c) R^{(0)}_s \left( \xi_{I}\right) H_{n - s}(\xi_{I^c}),
  \end{equation}
with $I_2, \dotsc, I_k$ defined as in \eqref{eq:identity},

  \begin{equation}
    H_{n - s}(\xi_{I^c}) =  \prod_{j \in I^c} \xi^s_j \sum^{N - s + 1}_{k = 2} (-1)^{k - 1} \sum_{I_2,\dotsc,I_k} G(I_2, \dotsc, I_k) \prod^k_{\alpha = 2} R^{(\sum^{\alpha - 1}_{i = 1} \lvert I_i \rvert)}_{\lvert I_{\alpha} \rvert}( \xi_{I_{\alpha}}) .
  \end{equation}

for $s = 1, \dotsc, n - 1$, and $H_{n - s}(\xi_{I^c}) = H_0(-) = 1$ for $s=n$. By the induction assumption,
  \begin{equation}
    H_{n - s}(\xi_{I^c}) = (-1)^{n - s + 1} \tau^{\frac{(n - s)(n - s - 1)}{2}} R^{(n - 1)}_{n - s}(\xi_{I^c}).
  \end{equation}
 for $1 \leq s < n$. We note that the summation over $s$ in \eqref{eq:intermediate_ASEP_1pt_lemma} is from $1$ to $n$. That is, $I_1$ can be any subset of $\{ 1, \dotsc, n \}$ except for the empty set. Then, it suffices to show the following identity

  \begin{equation} \label{eq:comb_for_ASEP}
    \sum^N_{s = 0} \sum_{|I| =s \text{ and } I \subset \{1, \dots, n \}} (-1)^{n - s} \tau^{\frac{(n - s)(n - s - 1)}{2}} G(I, I^c) R^{(0)}_s(\xi_{I}) R^{(n - 1)}_{n - s}(\xi_{I^c}) = 0,
  \end{equation}

since the right side of \eqref{eq:identity} corresponds to the $I = \emptyset$ term in \eqref{eq:comb_for_ASEP} and the left side of \eqref{eq:identity} corresponds to the negative of the sum of the terms with $I \neq \emptyset$ in \eqref{eq:comb_for_ASEP}.
        
First, we apply the change of variables $\xi_i=(1+z_i)/(1+\tau z_i)$ to \eqref{eq:comb_for_ASEP}. Then, we multiply both sides of \eqref{eq:comb_for_ASEP} by $(-1)^{n(n + 1)/2} (p - q)^{n(n - 1)/2} \prod^n_{k = 1} (1 - \tau z_k)^{2n - 2}$ on both sides so that \eqref{eq:comb_for_ASEP} is equivalent to

  \begin{multline} \label{eq:001}
    \sum^n_{s = 0} (-1)^s \sum_{\substack{\lvert I \rvert = s \\ I\subseteq\{1, \dotsc, n \} }} \left[ \prod_{\substack{\alpha \in I,\ \beta \in I^c \\ \alpha > \beta}} (-1) \prod_{\substack{\alpha, \alpha' \in I \\ \alpha < \alpha'}} (z_{\alpha'} - z_{\alpha}) \prod_{\substack{\beta, \beta' \in I^c \\ \beta < \beta'}} (\tau z_{\beta'} - \tau z_{\beta}) \right. \\
    \times \left. \prod_{\alpha \in I, \beta \in I^c} (\tau z_{\beta} - z_{\alpha}) \prod_{\alpha \in I}(1 + \tau z_{\alpha})^{n - 1} \prod_{\beta \in I^c}(1 +z_{\beta})^{n - 1} \right] = 0.
  \end{multline}

  Applying the Cauchy-Binet formula, we have that the left side of \eqref{eq:001} may be written as
  \begin{equation} \label{eq:002}
    \det \left( (1 + \tau z_i)^{n - 1} z^{j - 1}_i - (1 + z_i)^{n - 1} (\tau z_i)^{j - 1} \right)^n_{i, j = 1}.
  \end{equation}
  This determinant vanishes, because the $n \times n$ matrix has rank $\leq n - 1$. (To see this, we note that the entries in the last row of the matrix have degree $n - 2$, so the last row is a linear combinations of the other $n - 1$ rows.) Hence, this prove the lemma.
\end{proof}

\subsection{Distribution of $x_N(t)$ for $q$-TAZRP}

For technical reasons, we assume that $a_{[1]} = \dotsb = a_{[L]} = (1 - q)^{-1}$ for the conductance, or equivalently, $b_{[1]} = \dotsb = b_{[L]} = 1$, in the definition of the model, see Section \ref{subsec:notation}. Note that in this subsection we use a few symbols identical to those in Section \ref{subsec_1pt_ASEP} with similar but different meanings. We assume in this subsection that all integral $\dashint$ are over a big counterclockwise circle $C = \{ \lvert z \rvert = R \}$ with $R$ large enough. Analogous to \eqref{eq:defn_G_ASEP}, we define 
\begin{equation}
  G(J_1, \dotsc, J_n) = \prod_{i \leq i < j \leq n} \left( \prod_{\alpha \in J_i,\ \beta \in J_j, \text{ and } \alpha > \beta} -\frac{q w_{\alpha} - w_{\beta}}{q w_{\beta} - w_{\alpha}} \right).
\end{equation}
for finite disjoint subsets $J_1, \dotsc, J_n$ of positive integers. Also, we use the notation
	\begin{equation}
	f(w_{I}) = f(w_{i_1} , w_{i_2}, \dots , w_{i_s})  \hspace{5mm} \text{or} \hspace{5mm} f(w_{\sigma (I)}) = f(w_{i_{\sigma(1)}} , w_{i_{\sigma(2)}}, \dots , w_{i_{\sigma(s)}})
	\end{equation}
for any function $f$ on $s$ arguments, $\sigma$ a permutation of the set $\{1, 2, \dots, s \}$, and any subset $I \subset \mathbb{N}$ with $|I| =s$ and $I = \{ i_1 < i_2 < \cdots < i_s \}$.

In this subsection, we need the following result that is analogous to Lemma \ref{lem:Tracy_Widom_09} and formula \eqref{eq:comb_for_ASEP} for the ASEP case.

\begin{lemma} \label{lem:B_lem}
  Let $e_k(x_1, \dotsc, x_n)$ be the elementary symmetric polynomial and
  \begin{equation}
    B_n(w_1, \dotsc, w_n) = \prod_{1 \leq i < j \leq n} \frac{w_i - w_j}{qw_i - w_j}, \quad \text{so that} \quad B_0(-) = 1.
  \end{equation}
Then,
\begin{equation} \label{eq:sum_of_BB}
  \sum_{|I| = s \text{ and } I \subset \{ 1, \dots, n\} } B_s(w_{I}) B_{n - s}(w_{I^c}) G(I, I^c) = \qbinom{n}{s} B_n(w_1, \dotsc, w_n),
\end{equation}
  \begin{multline} \label{eq:lem_id_TAZRP_1}
    \sum_{|I| = s \text{ and } I \subset \{ 1, \dots, n\} } B_s(w_I) \prod_{i \in I} w_{i} B_{n - s}(w_{I^c}) \prod_{j \in I^c} (1 - tw_{j}) G(I, I^c)\\ 
=   (-1)^s B_n(w_1, \dotsc, w_n) \sum^n_{k = s} (-1)^k q^{s(n - k)} \qbinom{k}{s} e_k(w_1, \dotsc, w_n) t^{k - s},
  \end{multline}
  \begin{multline} \label{eq:lem_id_TAZRP_2}
    \sum_{|I| = s \text{ and } I \subset \{ 1, \dots, n\} } B_s(w_{I}) B_{n - s}(w_{I^c}) \prod_{j \in I^c} (1 - t w_{j}) G(I, I^c) = \\
    B_n(w_1, \dotsc, w_n) \sum^{n - s}_{k = 0} (-1)^k \qbinom{n - k}{s} e_k(w_1, \dotsc, w_n) t^k
  \end{multline}
so that all the sums on the left side of the identities are over all subsets $I \subset \{1, \dots, n \}$ with $|I| =s$ and $I^c = \{1, \dots, n \} \setminus I$.
\end{lemma}

The proof of the Lemma \ref{lem:B_lem} is given at the end of this subsection in Section \ref{sec:misc_proofs}.

We compute the one-point function for the periodic $q$-TAZRP analogous to \eqref{eq:aux_003}. We have
\begin{equation} \label{eq:prob_x_N=a}
  \begin{split}
    \prob_Y(X_N(t) = a) = {}& \sum_{X \in \Weylchamber_N(L),\ x_N = a} \prob_Y(X; t) \\
    = {}& (-1)^N \sum_{\Lattice \in \intZ^N(0)} \dashint dw_1 \dotsi \dashint dw_N D_{\Lattice}(w_1, \dotsc, w_N) \\
    & \times \sum_{X \in \Weylchamber_N(L) \, x_N = 0} C(w_1, \dotsc, w_N; X) \prod^N_{j = 1} (1 - w_j)^{y_j - a - 1} e^{-w_j t}
  \end{split}
\end{equation}
with
\begin{equation}
  C(w_1, \dotsc, w_N; X) = \frac{1}{W(X)} \sum_{\sigma \in S_N} A_{\sigma}(w_1, \dotsc, w_N) \prod^N_{j = 1} (1 - w_{\sigma(j)})^{-x_j},
\end{equation}
parameters $b_{[1]} = b_{[2]} = \dotsb = b_{[N]} = 1$, $\Lattice = (\ell_1, \dotsc, \ell_N)$, and $D_{\Lattice}(w_1, \dotsc, w_N)$ defined by \eqref{eq:general_D_L_qTAZRP}.

Let
\begin{equation} \label{eq:defn_F_n(z_1...z_n)}
  F_n(z_1, \dotsc, z_n) = \sum_{L > x_1 \geq x_2 \geq \dotsb \geq x_n > 0} \frac{1}{W(X)} \prod^n_{j = 1} z^{x_j}_j.
\end{equation}
Then, consider the summation in the last line of \eqref{eq:prob_x_N=a}. We have
\begin{multline} \label{eq:sum_of_C_TAZRP}
  \sum_{X \in \Weylchamber_N(L),\ x_N = 0} C(w_1, \dotsc, w_N; X) = \sum_{\sigma \in S_N} A_{\sigma}(w_1, \dotsc, w_N) \\
  \times \sum^N_{r = 1} \sum^{N - r}_{p = 0} \frac{1}{[p + r]_q!} \prod^p_{k = 1} (1 - w_{\sigma(k)})^{-L} F_{N - p - r} \left( \frac{1}{1 - w_{\sigma(p + 1)}}, \dotsc, \frac{1}{1 - w_{\sigma(N - r)}} \right).
\end{multline}
Set $I_1 = \{ \sigma(1), \dotsc, \sigma(p) \}$, $I_2 = \{ \sigma(N - r + 1), \dotsc, \sigma(N) \}$ and $I_3 = \{ 1, \dotsc, N \} \setminus (I_1 \cup I_2)$ so that we may rewrite the right side of \eqref{eq:sum_of_C_TAZRP}. Using the identity (\cite[Formula (5.1)]{Wang-Waugh16})
\begin{equation}
  \sum_{\sigma \in S_n} A_{\sigma}(w_1, \dotsc, w_n) = [n]_q! B_n(w_1, \dotsc, w_n),
\end{equation}
we have that

\begin{multline} \label{eq:C(w_1w_NX)_in_F_n}
  \sum_{X \in \Weylchamber_N(L),\ x_N = 0} C(w_1, \dotsc, w_N; X)    
  = \sum^N_{r = 1} \sum^{N - r}_{p = 0} \sum_{\substack{|I_1| = p, |I_3| =r, \\ I_1 \sqcup I_2 \sqcup I_3 = \{ 1, \dots, N\}}}  B_p(w_{I_1}) B_r(w_{I_3}) \prod_{1 \leq j<k \leq 3} G(I_j, I_k)  \\
  \times \prod_{\alpha \in I_1} (1 - w_{\alpha})^{-L} \frac{[p]_q! [r]_q!}{[p + r]_q!} \sum_{\mu \in S_{N - p - r}} A_{\mu}(w_{I_3}) F_{N - p - r} \left( \frac{1}{1 - w_{\mu (I_3)}}\right).
\end{multline}

so that the third sum on the right side is over all disjoint subsets $I_1, I_2, I_3 \subset \{1, \dots, N\}$ with $|I_1| =p$, $|I_2|=r$, and $|I_3| = N-p-r$.

Now, we set up some preliminary definitions to simplify $F_{N - p - r}$ in \eqref{eq:C(w_1w_NX)_in_F_n}. For any positive integer $n$, a \emph{$k$-composition} $C^{(k)}(n; s_1, s_2, \dotsc, s_{k - 1})$ is given by the partition $\{ s_0 + 1 = 1, 2, \dotsc, s_1 \}$, $\{ s_1 + 1, \dotsc, s_2 \}$, \dots, $\{ s_{k - 1} + 1, \dotsc, s_k = n \}$ with $0 = s_0 < s_1 < s_2 < \dotsb < s_{k - 1} < s_k = n$. We denote
\begin{equation}
\Gamma_n = \{ \text{all $k$-compositions of $n$ with $k = 1, \dotsc, n$} \}.
\end{equation}
 For a composition $\sigma$ of $n$, we write $\ell(\sigma) = k$ if $\sigma$ is a $k$-composition. The total number of compositions of $n$ is $2^{n - 1}$, and for each $k = 1, \dotsc, n$, the number of $k$-compositions of $n$ is $\binom{n - 1}{k - 1}$. Let $\sigma = C^{(k)}(n; s_1, \dotsc, s_{k - 1})$ be a $k$-composition of $n$ and $\tau = C^{(k)}(n; t_1, \dotsc, t_{l - 1})$ be an $l$-composition of $n$; we say $\tau$ is a \emph{refinement} of $\sigma$, or $\tau \prec \sigma$, if  $\{ s_1, \dotsc, s_{k - 1} \} \subseteq \{ t_1, \dotsc, t_{l - 1} \}$. It is clear that $\prec$ is a partial order, and $\tau \prec \sigma$ implies $\ell(\tau) \geq \ell(\sigma)$. The unique maximum among compositions of $n$ is the $1$-composition $C^{(1)}(n; -)$ and the unique minimum is the $n$-composition $C^{(n)}(n; 1, 2, \dotsc, n - 1)$.

We define real constants $c_1, c_2, \dotsc$ recursively by the identities
\begin{equation} \label{eq:id_defn_c_k}
  \sum_{\sigma \in \Gamma_n} c(\sigma) = \frac{1}{[n]_q!}, \quad \text{with} \quad c(C^{(k)}(n; s_1, \dotsc, s_{k - 1})) = \prod^k_{j = 1} c_{s_j - s_{j - 1}}.
\end{equation}

\begin{lemma} \label{eq:sum_of_comp_prec_C^k}
  For any $k$-composition $\sigma = C^{(k)}(n; s_1, \dotsc, s_{k - 1})$ of $n$,
  \begin{equation} \label{eq:general_sumation_c(tau)}
    \sum_{\tau \in \Gamma_n \text{ and } \tau \prec \sigma} c(\tau) = \prod^k_{j = 1} \frac{1}{[s_j - s_{j - 1}]_q!}.
  \end{equation}
\end{lemma}
\begin{proof}
  It is clear that \eqref{eq:general_sumation_c(tau)} is true for $k = 1$, which is directly given by identity \eqref{eq:id_defn_c_k}. For $k > 1$, we have a one-to-one mapping between the set of compositions $\tau$ which are refinements of $C^{(k)}(n; s_1, \dotsc, s_{k - 1})$ and the set of $k$-tuples of compositions $(\tau_1, \dotsc, \tau_k)$ so that $\tau_i$ is a composition of $s_i - s_{i - 1}$. In particular, if $ C^{(l)}(n; t_1, \dotsc, t_{l - 1}) = \tau \prec \sigma = C^{(k)}(n; s_1, \dotsc, s_{k - 1})$ with $s_1 = t_{i_1}, s_2 = t_{i_2}, \dotsc, s_{k - 1} = t_{i_{k - 1}}$, then
  \begin{equation}
    \tau_j = C^{(i_j - i_{j - 1})}(s_j - s_{j - 1}; t_{i_{j - 1} + 1} - s_{j - 1}, t_{i_{j - 1} + 2} - s_j, \dotsc, t_{i_j - 1} - s_j), \quad j = 1, 2, \dotsc, k
  \end{equation}
with $i_0 = 0$ and $i_k = l$. It is clear that $c(\tau) = c(\tau_1) \dotsm c(\tau_k)$. Hence,
  \begin{equation}
    \sum_{\tau \in \Gamma_n \text{ and } \tau \prec \sigma} c(\tau) = \prod^k_{j = 1} \left( \sum_{\tau_j \in \Gamma_{s_j - s_{j-1}}}c(\tau_j)  \right) = \prod^k_{j = 1} \frac{1}{[s_j - s_{j - 1}]_q!}.
  \end{equation}
\end{proof}

Next, we define
\begin{equation} \label{eq:defn_f_tau}
  f_{\tau}(z_1, \dotsc, z_n) = \sum_{L > x_{s_1} \geq \cdots \geq x_{s_k}>0} \prod_{i=1}^k \left( \prod_{j = s_{i-1}+1}^{s_i} z_j\right)^{x_{s_i}}
\end{equation}
for any $k$-composition $\tau = C^{(k)}(n; s_1, \dotsc, s_{k - 1})$ of $n$.
\begin{lemma} \label{lem:F_n_in_f_tau}
Take $F_n$ as given in \eqref{eq:defn_F_n(z_1...z_n)} and $f_{\tau}$ as given in \eqref{eq:defn_f_tau}. Then,
  \begin{equation} \label{eq:decomposition_of_partial_F}
    F_n(z_1, \dotsc, z_n) = \sum_{\tau \in \Gamma_n} c(\tau) f_{\tau}(z_1, \dotsc, z_n).
  \end{equation}
\end{lemma}
\begin{proof}
  Consider any monomial $z^{x_1}_1 \dotsm z^{x_n}_n$ such that 
  \begin{equation}
    x_1 = \dotsb = x_{s_1} = y_1, \quad x_{s_1 + 1} = \dotsb = x_{s_2} = y_2, \quad \dotsc, \quad x_{s_{k - 1} + 1} = \dotsb = x_n = y_k,
  \end{equation}
  with $L > y_1 > \dotsb > y_k > 0$. Then, this monomial occurs only once in the polynomial $f_{\tau}(z_1, \dotsc, z_n)$ on the right side of \eqref{eq:decomposition_of_partial_F} if $\tau \prec C^{(k)}(n; s_1, \dotsc, s_{k - 1})$, and otherwise, the monomial does not occur in the polynomial $f_{\tau}(z_1, \dotsc, z_n)$. Thus, the coefficient of this $z^{x_1}_1 \dotsm z^{x_n}_n$ on the right side of \eqref{eq:decomposition_of_partial_F} is $\prod^k_{j = 1} ([s_j - s_{j - 1}]_q!)^{-1}$ by Lemma \ref{eq:sum_of_comp_prec_C^k}, and this agrees with the coefficient on the left side of \eqref{eq:decomposition_of_partial_F} given by \eqref{eq:defn_F_n(z_1...z_n)}. 
\end{proof}
We want to express $f_{\tau}(z_1, \dotsc, z_n)$ in a compact and explicit form. It suffices to consider the case with $\tau = C^{(n)}(n; 1, 2, \dotsc, n - 1)$ since
\begin{equation}
  f_{C^{(k)}(n; s_1, \dotsc, s_{k - 1})}(z_1, \dotsc, z_n) = f_{C^{(k)}(k; 1, 2, \dotsc, k - 1)} \left( \prod^{s_1 - 1}_{i = 1} z_i, \prod^{s_2 - 1}_{i = s_1} z_i, \dotsc, \prod^n_{i = s_{k - 1}} z_i \right). 
\end{equation}
by \eqref{eq:defn_f_tau}. To this end, we define
\begin{equation}
  g_n(z_1, \dotsc, z_n) = \prod^n_{j = 1} \frac{z_j}{1 - z_j z_{j - 1} \dotsm z_1} \quad \text{for $n = 1, 2, \dotsc$}, \quad \text{so that} \quad g_0(-) = 1
\end{equation}
and also
\begin{multline}
  G_n(z_1, \dotsc, z_n) = \\
  \sum^n_{k = 0} (-1)^k \sum_{s_0 <  \dotsb < s_k } g_{n - s_k}(z_{s_k + 1}, \dotsc, z_n) \prod^{s_k}_{i = 1} z^{L - 1}_i \prod^k_{j = 1} g_{s_j - s_{j - 1}}(z_{s_{j - 1} + 1}, \dotsc, z_{s_j}).
\end{multline}
Then, we have
\begin{equation}
  f_{C^{(n)}(n; 1, 2, \dotsc, n - 1)}(z_1, \dotsc, z_n) = G_n(z_1, \dotsc, z_n).
\end{equation}
by using Lemma \ref{lem:Baik-Liu} and the expression
\begin{equation} \label{eq:sumation_f_n_min_comp}
  f_{C^{(n)}(n; 1, 2, \dotsc, n - 1)}(z_1, \dotsc, z_n) = \sum_{L > x_1 \geq \dotsb \geq x_n > 0} z^{x_j}_j = \prod^n_{i = 1} z^{i - n}_i \sum_{L + n - 1 > y_1 > \dotsb > y_n > 0} z^{y_j}_j,
\end{equation}
with $y_j = x_j - j + n$. So, it is straightforward to check that
\begin{multline} \label{expr_of_f}
  c(C^{(l)}(n; t_1, \dotsc, t_{l - 1})) f_{C^{(l)}(n; t_1, \dotsc, t_{l - 1})}(z_1, \dotsc, z_n) = \\
    \sum^l_{k = 0} \sum_{\substack{s_0 <  \dotsb < s_k  \\ \{ s_1, \dotsc, s_k \} \subseteq \{ t_1, \dotsc, t_l \}}} h^{(\tau; s_k + 1, n)}_{n - s_k}(z_{s_k + 1}, \dotsc, z_n) \prod^{s_k}_{i = 1} z^{L - 1}_i \prod^k_{j = 1} h^{(\tau; s_{j - 1} + 1, s_j)}_{s_j - s_{j - 1}}(z_{s_{j - 1} + 1}, \dots, s_j),
  \end{multline}
for an $l$-composition $C^{(l)}(n; t_1, \dotsc, t_{l - 1})$ with 
\begin{equation} \label{eq:h_in_g}
  h^{(\tau; t_i + 1, t_j)}_{t_j - t_i}(z_{t_i + 1}, \dotsc, z_{t_j}) = \prod^j_{k = i + 1} c_{t_k - t_{k - 1}} g_{j - i} \left( \prod^{t_{i + 1}}_{k = t_i + 1} z_k, \prod^{t_{i + 2}}_{k = t_{i + 1} + 1} z_k, \dotsc, \prod^{t_j}_{k = t_{j - 1} + 1} z_k \right),
\end{equation}
for $0 \leq i < j \leq l$  so that $h^{(\tau; n + 1, n)}_0(-) = 1$ and $\tilde{\tau} = C^{(j - i)}(t_j - t_i; t_{i + 1} - t_i, \dotsc, t_{j - 1} - t_i)$ a $(j - i)$-composition of $t_j - t_i$. Let

\begin{equation} \label{eq:defn_h_n}
  h_n(z_1, \dotsc, z_n) = \sum_{\substack{\tau \in \Gamma_n \\ \tau = C^{(l)}(n; t_1, \dotsc, t_{l - 1}) } } c(\tau) g_l \left( \prod^{t_1}_{k = 1} z_k, \prod^{t_2}_{k = t_1 + 1} z_k, \dotsc, \prod^{t_l}_{k = t_{l - 1} + 1} z_k \right),
\end{equation}
and
\begin{equation} \label{eq:defn_H_n}
  H_n(z_1, \dotsc, z_n) = \sum^n_{k = 1} (-1)^k \sum_{ s_0 < \dotsb  < s_{k}} \left(\prod^k_{i = 1} h_{s_i - s_{i - 1}}( z_{s_{i - 1} + 1}, \dotsc, z_{s_i}) \right).
\end{equation}

\begin{lemma} \label{lem:Baik-Liu_like}
Take $F_n$ as given in \eqref{eq:defn_F_n(z_1...z_n)}, $h_n$ as given in \eqref{eq:defn_h_n} and $H_n$ as given in \eqref{eq:defn_H_n}. Then,
  \begin{equation} \label{eq:sum_for_whole}
    F_n(z_1, \dotsc, z_n) = \sum^n_{s = 0} \prod^s_{i = 1} z^{L - 1}_i H_s(z_1, \dotsc, z_s) h_{n - s}(z_{s + 1}, \dotsc, z_n).
  \end{equation}
\end{lemma}
\begin{proof}
First, we introduce some notation. Let $Z_i = \prod^{t_i}_{j = t_{i - 1} + 1} z_j$ for $i = 1, \dotsc, l$ with $t_0 = 0$ and $t_l = n$, and define
\begin{equation}
	I_{(\tau, i_1, \dotsc, i_k)} = g_{l - i_k}(Z_{i_k + 1}, \dotsc, Z_l) \prod^{i_k}_{j = 1} Z^{L - 1}_j \prod^k_{j = 1} g_{i_j - i_{j - 1}}(Z_{i_{j - 1} + 1}, \dotsc, Z_{i_j})
\end{equation}
with $\tau = (t_1, \dotsc, t_{l - 1})$ an $l$-composition of $n$ and $\{ i_1 < \dotsb < i_k \} \subseteq \{ 1, 2, \dotsc, l \}$.

  By Lemma \ref{lem:F_n_in_f_tau}, the left side of \eqref{eq:sum_for_whole} is a sum of terms $c(\tau) f_{\tau}(z_1, \dotsc, z_n)$, which are weighted sums of $I_{(\tau; i_1, \dotsc, i_k)}$ by \eqref{expr_of_f} and \eqref{eq:h_in_g}. On the other hand, the right side of \eqref{eq:sum_for_whole} may also be expressed as a weighted sum of $I_{(\tau; i_1, \dotsc, i_k)}$. By comparing the coefficients of $I_{(\tau; i_1, \dotsc, i_k)}$, the lemma follows.
\end{proof}
We also use the following identity for the functions $h_n$ and $H_n$.
\begin{lemma} \label{lem:h_n_H_n}
  Take $h_n$ and $H_n$ as defined in \eqref{eq:defn_h_n} and \eqref{eq:defn_H_n}. We have
  \begin{equation} \label{eq:Lee-Wang17_result}
    \sum_{\sigma \in S_n} A_{\sigma}(w_1, \dotsc, w_n) h_n \left( \frac{1}{1 - w_{\sigma(1)}}, \dotsc, \frac{1}{1 - w_{\sigma(n)}} \right) = (-1)^n \prod^n_{k = 1} \frac{1}{w_k} B_n(w_1, \dotsc, w_n).
  \end{equation}
  and
  \begin{equation} \label{eq:intemediate_L_sum}
    \sum_{\sigma \in S_n} A_{\sigma}(w_1, \dotsc, w_n) H_n \left( \frac{1}{1 - w_{\sigma(1)}}, \dotsc, \frac{1}{1 - w_{\sigma(n)}} \right) = q^{n(n - 1)/2} \prod^n_{k = 1} \frac{1}{w_k} B_n(w_1, \dotsc, w_n).
  \end{equation}
\end{lemma}

\begin{proof}[Proof of Lemma \ref{lem:h_n_H_n}]
Note that we may write the multi-series expansion
  \begin{equation}
    g_n(z_1, \dotsc, z_n) = \sum_{1 \leq x_n \leq x_{n - 1} \leq \dotsb \leq x_1 < \infty} \prod^n_{k = 1} z^{x_k}_k,
  \end{equation}
 if $\lvert z_k \rvert < 1$ for all $k = 1, \dotsc, n$. Then,
  \begin{equation}
    h_n(z_1, \dotsc, z_n) = \sum_{1 \leq x_n \leq x_{n - 1} \leq \dotsb \leq x_1 < \infty} \frac{1}{W(X)} \prod^n_{k = 1} z^{x_k}_k
  \end{equation}
for $X = (x_1, \dotsc, x_n)$.

Identity \eqref{eq:Lee-Wang17_result} is proved in \cite{Lee-Wang17}. More specifically, \cite[Proposition 2.1]{Lee-Wang17} yields that
  \begin{multline} \label{eq:Lee_Wang_2.1}
    (-1)^n \dashint \frac{dw_1}{w_1} \dotsi \dashint \frac{dw_n}{w_n} \prod^n_{k = 1} (1 - w_k)^{y_j - 1} \sum_{\sigma \in S_n} A_{\sigma}(w_1, \dotsc, w_n) h_n \left( \frac{1}{1 - w_{\sigma(1)}}, \dotsc, \frac{1}{1 - w_{\sigma(n)}} \right) = \\
    \dashint \frac{dw_1}{w_1} \dotsi \dashint \frac{dw_n}{w_n} \prod^n_{k = 1} (1 - w_k)^{y_j - 1} \prod^n_{k = 1} \frac{1}{w_k} B_n(w_1, \dotsc, w_n),
  \end{multline}
with $N$ replaced by $n$, $M = 0$, $b_i = 1$ for all $i \in \intZ$ and any $y_1 > y_2 > \dotsb > y_n$.  The left side of \eqref{eq:Lee_Wang_2.1} is the  probability that the leftmost particle is $> 0$ in the $n$-particle $q$-TAZRP with initial condition $Y = (y_1, \dotsc, y_n)$. The proof of \cite[Proposition 2.1]{Lee-Wang17} actually verifies \eqref{eq:Lee-Wang17_result}, that the integrands on both sides of \eqref{eq:Lee_Wang_2.1} are equal.
  
  We prove identity \eqref{eq:intemediate_L_sum} by induction. The $n = 1$ case is obviously true. We express
  \begin{equation}
    H_n(z_1, \dotsc, z_n) = -\sum^n_{s = 1} h_s(z_1, \dotsc, z_s) H_{n - s}(z_{s + 1}, \dotsc, z_n).
  \end{equation}
  Suppose \eqref{eq:intemediate_L_sum} holds for $n - 1, n - 2, \dotsc, 1$. Then, we multiply the left side of \eqref{eq:intemediate_L_sum} by $(-1) \prod^n_{k = 1} w_k$, and it becomes
  \begin{multline} \label{eq:proof_BBG_sum}
    \sum^n_{s = 1} (-1)^{s} q^{\frac{(n - s)(n - s - 1)}{2}} \sum_{\substack{|I| = s \\ I \subset \{1, \dots, n \}}} B_s(w_{I})  B_{n - s}(w_{I^c}) G(I, I^c)  = \\ B_n(w_1, \dotsc, w_n) \sum^n_{s = 1} (-1)^s \qbinom{n}{s} q^{\frac{(n - s)(n - s - 1)}{2}}
  \end{multline}
by applying the $n \to s$ case of \eqref{eq:Lee-Wang17_result} and the $n \to n - s$ case of \eqref{eq:intemediate_L_sum} and using \eqref{eq:sum_of_BB} for the last step. Lastly, we evaluate the sum over $s$ in \eqref{eq:proof_BBG_sum} by \cite[Corollary 10.2.2(c)]{Andrews-Askey-Roy99}, and this gives \eqref{eq:intemediate_L_sum} for $n$.
\end{proof}

Now, we evaluate the sum over $\mu \in S_{N - p - r}$ in \eqref{eq:C(w_1w_NX)_in_F_n} by using Lemmas \ref{lem:Baik-Liu_like} and \ref{lem:h_n_H_n}. We have
\begin{multline}
  \sum_{X \in \Weylchamber_N(L),\ x_N = 0} C(w_1, \dotsc, w_N; X) = \sum^N_{r = 1} \sum^{N - r}_{p = 0} \sum^{N - p - r}_{s = 0} \frac{[p]_q! [r]_q!}{[p + r]_q!} q^{s(s - 1)/2} (-1)^{N - p - r - s} \\
  \times \sum_{\substack{|I_1| =p, |I_2| = s, |I_4| = r \\ I_1 \sqcup I_2 \sqcup I_3 \sqcup I_4 = \{1, \dots, N \}}} \prod_{i=1}^4 B_{|I_i|} (w_{I_i}) \prod_{1 \leq j < k \leq 4}G(I_j , I_k)  \prod_{\alpha \in I_1} (1 - w_{\alpha})^{-L} \prod_{\beta \in I_2} (1 - w_{\beta})^{1 - L} \prod_{\gamma \in I_2 \cup I_3} \frac{1}{w_{\gamma}}.
\end{multline}
so that the inner-most sum on the right side is over disjoint subsets $I_1, I_2, I_3, I_4 \subset \{1, \dots, N \}$ with $|I_1| = p, |I_2| = s, |I_3| = N-p - r- s$, and $|I_4| =r$.

Analogous to \eqref{eq:sum_DC} in the ASEP case, we have

\begin{multline} \label{eq:sum_with_general_ell_TAZRP}
  D_{\Lattice}(w_1, \dotsc, w_N) \sum_{\substack{X \in \Weylchamber_N(L) \\ x_N = 0}} C(w_1, \dotsc, w_N; X) = \sum^N_{r = 1} \sum^{N - r}_{p = 0} \sum^{N - p - r}_{s = 0} \frac{[p]_q! [r]_q!}{[p + r]_q!} q^{s(s - 1)/2} (-1)^{N - p - r - s} \\
  \times \sum_{\substack{|I_2| = r, |I_3| = p, |I_4| = s \\ I_1 \sqcup I_2 \sqcup I_3 \sqcup I_4 =\{ 1, \dots, N\}}} (-1)^{(p + s)(N - p - s)} D_{\Lattice(I_3 \cup I_4)}(w_1, \dotsc, w_N)\\
\times \prod_{i=1}^4 B_{|I_i| } (w_{I_i}) \prod_{1 \leq j < k \leq 4} G(I_j, I_k) \prod_{\alpha \in I_4} (1 - w_{\alpha}) \prod_{\beta \in I_1 \cup I_4} \frac{1}{w_{\beta}},
\end{multline}

with $\Lattice = (\ell_1, \dotsc, \ell_N)$ and $\Lattice(I_3 \cup I_4) = (\ell'_1, \dotsc, \ell'_N)$ defined by $\ell'_k = \ell_k - 1$ if $k \in I_3 \cup I_4$ and $\ell'_k = \ell_k$ otherwise. Note that $\Lattice(I_3 \cup I_4) \in \intZ^N(-p - s)$ if $\Lattice \in \intZ^N(0)$. Then, $\Lattice(I_3 \cup I_4)$ runs over $\intZ^N(-p - s)$  if $\Lattice$ runs over $\intZ^N(0)$ for any fixed $I_3 \cup I_4$. So, we have the identity
\begin{multline} \label{eq:formula_in_S(K)}
  \sum^{N - 1}_{K = 0} S(K) = \sum_{\Lattice \in \intZ^N(0)} \dashint dw_1 \dotsi \dashint dw_N D_{\Lattice}(w_1, \dotsc, w_N) \\
  \times \sum_{X \in \Weylchamber_N(L), \, x_N = 0} C(w_1, \dotsc, w_N; X)  \prod^N_{j = 1} (1 - w_j)^{y_j - a - 1} e^{-w_j t},
\end{multline}
with
\begin{multline} \label{eq:S(K)_in_S_Kpr}
  S(K) = (-1)^{K(N - K)} \sum_{\Lattice \in \intZ^N(-K)} \dashint dw_1 \dotsc \dashint dw_N \sum^K_{p = 0} \sum^{N - K}_{r = 1} D_{\Lattice}(w_1, \dotsc, w_N) \\
  \times S_{K, p, r}(w_1, \dotsc, w_N)  \prod^N_{j = 1} (1 - w_j)^{y_j - a - 1} e^{-w_j t},
\end{multline}
and

\begin{multline} \label{eq:general_S_Kpr}
  S_{K, p, r}(w_1, \dotsc, w_N) = \frac{[p]_q! [r]_q!}{[p + r]_q!} q^{\frac{(K-p)(K-p - 1)}{2}} (-1)^{N -K  - r } \\ 
\times \sum_{\substack{|I_2| = r, |I_3| = p, |I_4| =K-p \\ I_1 \sqcup I_2 \sqcup I_3 \sqcup I_4 = \{1, \dots, N \}}} \prod_{i =1}^4 B_{|I_i|}(w_{I_i}) \prod_{1 \leq j <k \leq 4 } G(I_j, I_k)  \prod_{\alpha \in I_4} (1 - w_{\alpha}) \prod_{\beta \in I_1 \cup I_4} \frac{1}{w_{\beta}}.
\end{multline}

so that $K$ is an integer that corresponds to $p + s$ in \eqref{eq:sum_with_general_ell_TAZRP}. Recall, by Lemma \ref{lem:B_lem} and \eqref{eq:sum_of_BB},  that

\begin{equation}
  \sum_{\substack{|I | = p \\ I \subset U}  } B_p(w_{I}) B_r(w_{I^c}) G(I^c, I) = \frac{[p + r]_q!}{[p]_q! [r]_q!} B_{p + r}(w_{U}).
\end{equation}

if $U$ is a subset of $\{ 1, \dotsc, N \}$ with $\lvert U \rvert = p + r$ and $U = \{ l_1 < \dotsb < l_{p + r} \}$. Then, applying this identity to \eqref{eq:general_S_Kpr}, we have 

\begin{multline}
  S_{K, p, r}(w_1, \dotsc, w_N) = q^{\frac{(K - p)(K - p - 1)}{2}} (-1)^{N - K - r} \\
\times  \sum_{\substack{ |I_2| = p +r, |I_3| = K-p \\ I_1 \sqcup I_2 \sqcup I_3 = \{1, \dots , N \}}} \prod_{i=1}^3 B_{|I_i|}(w_{I_i}) \prod_{1 \leq j< k\leq 3} G(I_j, I_k) \prod_{\alpha \in I_3} (1 - w_{\alpha}) \prod_{\beta \in I_1 \cup I_3} \frac{1}{w_{\beta}}.
\end{multline}

We note that for each $K = 0, \dotsc, N - 1$, $S(K)$ is well-defined because the the right side of \eqref{eq:S(K)_in_S_Kpr} is absolutely convergent by Lemma \ref{lem:estimate_qTAZRP}. Only finitely many terms on the left side of \eqref{eq:formula_in_S(K)} are non-vanishing by Lemmas \ref{lem:finiteness_of_large_max(L)} and \ref{lem:qTAZRP_conv}.

Now, by Lemma \ref{lem:B_lem}  and the $t = 1$ case of \eqref{eq:lem_id_TAZRP_1}, we have

\begin{multline} \label{eq:interm_S_K(w)}
  S_{K, p, r}(w_1, \dotsc, w_N) = q^{\frac{(K - p)(K - p - 1)}{2}} (-1)^{N - K - r}  \prod^N_{k = 1} \frac{1}{w_k}\\
\times  \sum_{\substack{|I| =K +r \\ I \subset \{1, \dots , N\} } } B_{K+r}(w_I) B_{N-K-r}(w_{I^c})G(I^c, I)  \sum^{K + r}_{k = p + r} (-1)^{k+p+r} q^{(p + r)(K + r - k)} \qbinom{k}{p + r} e_k(w_{I})
\end{multline}

Moreover, if we compare the $t^k$ coefficient on both sides of \eqref{eq:lem_id_TAZRP_2} in Lemma \ref{lem:B_lem}, we have 

\begin{equation} \label{eq:most_complicated_BBG}
  \sum_{\substack{|I| = s \\ I \subset \{1, \dots, N \}}} B_s(w_{I}) B_{n - s}(w_{I^c}) e_k(w_{I^c}) G(I, I^c) =   B_n(w_1, \dotsc, w_n) \qbinom{n - k}{s} e_k(w_1, \dotsc, w_n)
\end{equation}

for $k \leq n - s$. Then, by applying \eqref{eq:most_complicated_BBG} in \eqref{eq:interm_S_K(w)}, we have
\begin{equation} \label{eq:formula_S_Kpr}
  S_{K, p, t}(w_1, \dotsc, w_N) = \prod^N_{k = 1} \frac{1}{w_k} B_N(w_1, \dotsc, w_N) \sum^{K + r}_{k = p + r} (-1)^k s_{K, p, r}(k) e_k(w_1, \dotsc, w_N),
\end{equation}
with the coefficients
\begin{equation}
  s_{K, p, r}(k) = (-1)^{N - K + p} q^{\frac{p^2}{2} + (r - k + \frac{1}{2})p + \frac{K^2 - K}{2} + r(K + r - k)} \qbinom{k}{p + r} \qbinom{N - k}{N - K - r}.
\end{equation}
Furthermore, if we assume that $s_{K, p, r}(k) = 0$ if $k < p + r$ or $k > K + r$, we will have
\begin{equation} \label{eq:exact_s(k)}
  \sum^K_{p = 0} \sum^{N - K}_{r = 1} s_{K, p, r}(k) = (-1)^{N - K} q^{(K^2 + K)/2} \qbinom{N - 1}{K},
\end{equation}
for $K = 0, 1, \dotsc, N - 1$. We give the proof of \eqref{eq:exact_s(k)} at the end of this subsection in the Section \ref{sec:misc_proofs}. Then, we have that 
\begin{equation}
  \sum^N_{k = 1} (-1)^k \sum^K_{p = 0} \sum^{N - K}_{r = 1} s_{K, p, r}(k) e_k(w_1, \dotsc, w_N) = (-1)^{N - K} q^{(K^2 + K)/2} \qbinom{N - 1}{K} \left( \prod^N_{i = 1} (1 - w_i) - 1 \right),
\end{equation}
by \eqref{eq:exact_s(k)}. In turn, we may write
\begin{multline} \label{eq:S(K)_final}
  S(K) = (-1)^{(K - 1)(N - K)} q^{(K^2 + K)/2} \qbinom{N - 1}{K} \sum_{\Lattice \in \intZ^N(-K)} \dashint \frac{dw_1}{w_1} \dotsi \dashint \frac{dw_N}{w_N} D_{\Lattice}(w_1, \dotsc, w_N) \\
  \times B_N(w_1, \dotsc, w_N) \prod^N_{j = 1} (1 - w_j)^{y_j - a - 1} e^{-w_j t} \left( \prod^N_{i = 1} (1 - w_i) - 1 \right),
\end{multline}
by \eqref{eq:S(K)_in_S_Kpr} and \eqref{eq:formula_S_Kpr}. Similar to \eqref{eq:S(K)_in_S_Kpr}, only finitely many terms on the right side of \eqref{eq:S(K)_final} are non-vanishing. We simplify the integrand on the right side of \eqref{eq:prob_x_N=a} using \eqref{eq:S(K)_final} and \eqref{eq:formula_in_S(K)}. Next, we obtain the one-point function
\begin{multline}
  \prob_Y(x_N(t) = a) = \frac{(-1)^N}{2\pi i} \oint_0 \frac{dz}{z} C_N(z) \sum^{N - 1}_{k = 0} z^{-k} \sum_{\Lattice \in \intZ^N(-k)} \dashint \frac{dw_1}{w_1} \dotsi \dashint \frac{dw_N}{w_N} D_{\Lattice}(w_1, \dotsc, w_N) \\
  \times B_N(w_1, \dotsc, w_N) \prod^N_{j = 1} (1 - w_j)^{y_j - a - 1} e^{-w_j t} \left( \prod^N_{i = 1} (1 - w_i) - 1 \right),
\end{multline}
by using identity \eqref{eq:tau_identity} with $\tau$ replaced by $q$ and $C_N(z)$ defined in \eqref{eq:C_N_qTAZRP}. Also, due the estimate in Lemma \ref{lem:qTAZRP_conv}, we have the summation formula
\begin{equation}
  \begin{split}
    \prob_Y(x_N(t) > M) = {}& \sum^{\infty}_{a = M + 1} \prob_Y(x_N(t) = a) \\
    = {}& \frac{(-1)^N}{2\pi i} \oint_0 \frac{dz}{z} C_N(z) \sum^{N - 1}_{k = 0} z^{-k} \sum_{\Lattice \in \intZ^N(-k)} \dashint \frac{dw_1}{w_1} \dotsi \dashint \frac{dw_N}{w_N} D_{\Lattice}(w_1, \dotsc, w_N) \\
    & \times B_N(w_1, \dotsc, w_N) \prod^N_{j = 1}  e^{-w_j t} \sum^{\infty}_{a = M + 1} \left( \prod^N_{j = 1} (1 - w_j)^{y_j - a} - \prod^N_{j = 1} (1 - w_j)^{y_j - a - 1} \right).
  \end{split}
\end{equation}
Hence, we prove the $m = 1$ case of Theorem \ref{thm:one_pt_qTAZRP}.

\subsubsection{Miscellaneous proofs}\label{sec:misc_proofs}
\begin{proof}[Proof of \eqref{eq:exact_s(k)}]
Equation \eqref{eq:exact_s(k)} verified easily for $k=1$. Then, by an induction argument, we only need to show that $\sum^K_{p = 0} \sum^{N - K}_{r = 1} (s(k + 1) - s(k)) = 0$ for $k = 1, \dotsc, N - 1$. To this end, we write
  \begin{equation}
    s_{K, p, r}(k) = q^{N(N - 1)/2} f^{(k)}_{p + r} g^{(k)}_{N - K - r},
  \end{equation}
  with
  \begin{equation}
    f^{(k)}_j =  q^{\binom{j}{2}} \qbinom{k}{j} (-q^{1 - k})^j, \quad g^{(k)}_j = q^{\binom{j}{2}} \qbinom{N - k}{j} (-q^{1 - N})^j.
  \end{equation}
These terms have the following generating functions:
  \begin{equation}
    \prod^{k - 1}_{i = 0} (1 - q^{-i} x) = \sum^k_{j = 0} f^{(k)}_j x^j, \quad \prod^{N - 1}_{i = k} (1 - q^{-i} x) = \sum^{N - k}_{j = 0} g^{(k)}_j x^j.
  \end{equation}
  So,
  \begin{equation}
   s_{K, p, r}(k) = \sum_{\substack{\{ a_1 < \dotsb < a_{p + r} \} \subseteq \{ 0, \dotsc, k - 1 \} \\  \{ b_1 < \dotsb < b_{N - K - r} \} \subseteq \{ k, \dotsc, N - 1 \}}} \prod^{p + r}_{i = 1} (-q^{-a_i}) \prod^{N - K - r}_{i = 1}(-q^{-b_i}).
  \end{equation}
  By comparison, we have
  \begin{equation}
    \sum^K_{p = 0} \sum^{N - K}_{r = 1} (s(k + 1) - s(k)) = 
    \begin{cases}
      0 & \text{if $k \geq K$}, \\
      \displaystyle \left( \sum^k_{p = 0} \sum_{\{ a_1, \dotsc, a_p \} \subseteq \{ 0, \dotsc, k - 1 \}} \prod^p_{i = 1} (-q^{a_i}) \right) (-q^k) g^{(k + 1)}_{N - K} & \text{if $k < K$}.
    \end{cases}
  \end{equation}
  Noting that
  \begin{equation}
    \sum^k_{p = 0} \sum_{\{ a_1, \dotsc, a_p \} \subseteq \{ 0, \dotsc, k - 1 \}} \prod^p_{i = 1} (-q^{a_i}) = \prod^{k - 1}_{i = 0} (1 - q^{-i}) = 0
  \end{equation}
for $k \geq 1$. Thus, we have that $\sum^K_{p = 0} \sum^{N - K}_{r = 1} (s(k + 1) - s(k)) = 0$ for $k = 1, \dotsc, N - 1$. This proves \eqref{eq:exact_s(k)}.
\end{proof}

\begin{proof}[Proof of Lemma \ref{lem:B_lem}]
 We only prove \eqref{eq:lem_id_TAZRP_1} since \eqref{eq:lem_id_TAZRP_2} is similar to \eqref{eq:lem_id_TAZRP_1}. Also, we have that \eqref{eq:sum_of_BB} is equivalent to \cite[Formulas (98) and (105)]{Lee-Wang17}, which are proved in the same way as we prove \eqref{eq:lem_id_TAZRP_1} below. So, we also omit the proof of \eqref{eq:sum_of_BB}.
  
  We multiply the left side of \eqref{eq:lem_id_TAZRP_1} by the factor $q^{s(s - 1)/2} \prod_{1 \leq i < j \leq N} (qw_i - w_j)$, and it becomes
  \begin{multline}
    \sum_{\substack{|I|= s \\ I \subset \{1, \dots , n \}}} \prod_{1 \leq k < l \leq s} (qw_{i_k} - qw_{i_l}) \prod_{\alpha \in I} w_{\alpha} \prod_{1 \leq k < l \leq n - s} (w_{j_k} - w_{j_l}) \prod_{\beta \in I^c} (1 - t w_{\beta}) \\
    \times \prod_{\alpha \in I} \prod_{\beta \in I^c} (qw_{\alpha} - w_{\beta}) \prod_{\alpha \in I,\ \beta \in I^c,\ \alpha > \beta} (-1).
  \end{multline}
  We recognize that this can be written as the $z^s$ coefficient of 
  \begin{equation}
    \det \left( (1 - tw_i)w^{j - 1}_i + zw_i(qw_i)^{j - 1} \right)^n_{i, j = 1}.
  \end{equation}
  By the determinantal representation of $e_k(w_1, \dotsb, w_n)$, we have that the determinant above is equal to
  \begin{equation} \label{eq:proof_of_lemma_BBG_other}
    \prod_{1 \leq i < j \leq n} (w_i - w_j) \left( \sum^n_{k = 0} \prod^k_{i = 1} (q^{n - k}z - t) e_k(w_1, \dotsc, w_n) \right).
  \end{equation}
  It is straightforward to check that the $z^s$ term of \eqref{eq:proof_of_lemma_BBG_other} is equal to the right side of \eqref{eq:lem_id_TAZRP_1} multiplied by the factor $q^{s(s - 1)/2} \prod_{1 \leq i < j \leq N} (qw_i - w_j)$. 
\end{proof}

\subsection{The $m = 2, \dotsc, N$ cases of Theorems \ref{thm:one_pt_ASEP} and \ref{thm:one_pt_qTAZRP}} \label{subsubsec:general_m_ASEP}

In this subsection we concentrate on the ASEP, and prove the $m = 2, \dotsc, N$ cases of Theorem \ref{thm:one_pt_ASEP} by the $m = 1$ result and the cyclic invariance of the ASEP on a ring. The argument for the $q$-TAZRP is analogous. Thus, we omit the proof for Theorem \ref{thm:one_pt_qTAZRP}.

We note that 
\begin{equation}
  \prob_Y(x_m(t) \geq M) = \prob_{\tilde{Y}}(x_1(t) \geq M)
\end{equation}
by the cyclic invariance of the model, for the initial conditions $Y = (y_1, \dotsc, y_N) \in \confs_N(L)$ and $\tilde{Y} = (y_m, y_{m + 1}, \dotsc, y_N, y_1 + L, \dotsc, y_{m - 1} + L) \in \confs_N(L) $. Then,
\begin{multline} \label{eq:general_m_ASEP}
  \prob_Y(x_m(t) \geq M) = 
\frac{p^{N(N - 1)/2}}{2\pi i} \oint \frac{dz}{z^m} C_N(z) \sum^{N - m}_{k = 1 - m} z^{-k} \sum_{\Lattice \in \intZ^N(-k)} \dashint_C \frac{d\xi_1}{1 - \xi_1} \dotsi \dashint_C \frac{d\xi_N}{1 - \xi_N} \\
    \times \prod^N_{j = m} \xi^{M - y_j - 1}_{j - m + 1} \prod^{m - 1}_{j = 1} \xi^{M - y_j - 1 - L}_{j + N - m + 1} \prod^N_{k = 1} e^{\epsilon(\xi_k)t} \prod_{1 \leq i < j \leq N} \frac{\xi_j - \xi_i}{p + q\xi_i\xi_j - \xi_i} D_{\Lattice}(\xi_1, \dotsc, \xi_N).
\end{multline}
by the $m = 1$ special case of Theorem \ref{thm:one_pt_ASEP}. Then, we use an identity analogous to \eqref{eq:change_level_id_ASEP}. We have
\begin{multline}
  \prod^{m - 1}_{j = 1} \xi^{-L}_{j + N - m + 1} \prod_{1 \leq i < j \leq N} \frac{\xi_j - \xi_i}{p + q\xi_i\xi_j - \xi_i} D_{\Lattice}(\xi_1, \dotsc, \xi_N) = \\
  (-1)^{(m - 1)(N - m + 1)} \prod_{1 \leq i < j \leq N} \frac{\eta_j - \eta_i}{p + q\eta_i\eta_j - \eta_i} D_{\Lattice''}(\eta_1, \dotsc, \eta_N),
\end{multline}
with the change of variables $\Lattice = (\ell_1, \dotsc, \ell_N) \mapsto \Lattice'' = (\ell''_1, \dotsc, \ell''_N)$ and $\xi \mapsto \eta$ given by
\begin{equation}
  \eta_j =
  \begin{cases}
    \xi_{j + N - m + 1} & \text{for $j = 1, \dotsc, m - 1$}, \\
    \xi_{j - m + 1} & \text{for $j = m, \dotsc, N$},   
  \end{cases}
  \quad
  \ell'' =
  \begin{cases}
    \ell_{j + N - m + 1} - 1 & \text{for $j = 1, \dotsc, m - 1$}, \\
    \ell_{j - m + 1} & \text{for $j = m, \dotsc, N$}.
  \end{cases}
\end{equation}
Hence,
\begin{multline}
  \sum_{\Lattice \in \intZ^N(-k)} \dashint_C \frac{d\xi_1}{1 - \xi_1} \dotsi \dashint_C \frac{d\xi_N}{1 - \xi_N} \prod^N_{j = m} \xi^{M - y_j - 1}_{j - m + 1} \prod^{m - 1}_{j = 1} \xi^{M - y_j - 1 - L}_{j + N - m + 1} \prod^N_{k = 1} e^{\epsilon(\xi_k)t} \\
  \times \prod_{1 \leq i < j \leq N} \frac{\xi_j - \xi_i}{p + q\xi_i\xi_j - \xi_i} D_{\Lattice}(\xi_1, \dotsc, \xi_N)  \\
= (-1)^{(m - 1)(N - m + 1)} \sum_{\Lattice \in \intZ^N(-k - m)} 
  \dashint_C \frac{d\eta_1}{1 - \eta_1} \dotsi \dashint_C \frac{d\eta_N}{1 - \eta_N}\\
\times  \prod^N_{j = 1} \eta^{M - y_j - 1}_j e^{\epsilon(\eta_j)t} \prod_{1 \leq i < j \leq N} \frac{\eta_j - \eta_i}{p + q\eta_i\eta_j - \eta_i} D_{\Lattice}(\eta_1, \dotsc, \eta_N).
\end{multline}
Plugging this into \eqref{eq:general_m_ASEP}, we prove Theorem \ref{thm:one_pt_ASEP} for $m = 2, \dotsc, N$.

\section{One Point Function for TASEP on the Ring and ASEP on the Line}\label{sec:one_point_limits}

We show that the formula \eqref{eq:one_point_distribution} from Theorem \ref{thm:one_pt_ASEP} for the ASEP on the ring agrees with the corresponding formulas for the ASEP on the line from Tracy and Widom in \cite{Tracy-Widom08} and the TASEP on the ring from Baik and Liu in \cite{Baik-Liu18}. In the case of the ASEP on the line, we take the length of the ring $L \in \mathbb{Z}_{\geq 1}$ approaching $\infty$. In the case of the TASEP on the ring, we take the asymmetry parameter $p \in [0,1]$ and set it to $p=1$.

\paragraph{TASEP on the Ring}

The argument for the one point function is similar to that for the transition probability in Section \ref{sec:limits}.

\begin{corollary}
  If $p=1$, the formula for $\mathbb{P}_Y(x_m (t) \geq M)$ from Theorem~\ref{thm:one_pt_ASEP} is equivalent to the one-point formula in \cite[Proposition 6.1]{Baik-Liu18} for the TASEP on the ring:
  \begin{multline}
    \mathbb{P}_{Y}(x_m(t) \geq M) = \\
    \frac{(-1)^{(m-1)(N+1)}}{2 \pi i} \oint_0 \det \left[\frac{1}{L} \sum_{w \in R_z} \frac{w^{j-k+1-m} (1 + w)^{y_j -j-M+m +1} e ^{t w}}{w + \rho}\right]_{j, k =1}^N \frac{d z}{z^{1- (m-1)L}},
  \end{multline}
  where $\rho = N/L$ and $R_z$ is defined in \eqref{eq:bethe_roots}.
\end{corollary}

\begin{proof}
  We start with \eqref{eq:one_point_distribution} and set $p =1$ and $q=0$. We have that 
  \begin{equation}
    \tau =0, \quad C_N(z)=1, \quad \epsilon (\xi) = \xi^{-1} -1.
  \end{equation}
  Then, with the change of variable $\xi_j \to (w_j + 1)^{-1}$, like in \eqref{eq:TASEP_change_var}, we rewrite \eqref{eq:one_point_distribution} as 
  \begin{multline}\label{eq:point_tasep}
    \prob_Y(x_m(t) \geq M) = \sum_{\mathbb{L} \in \intZ^N} \frac{(-1)^{(m-1)(N-1)}}{2 \pi i} \oint_0 \frac{dz}{ z^{1 - L (m-1)}} \dashint d w_1 \cdots \dashint d w_N \det [w_j^{-k}]_{j, k=1}^N \\
    \times\prod_{j=1}^N \left[\left( \frac{z^L}{ w_j^N (w_j +1)^{L-N}} \right)^{\ell_j} \left(\frac{ w_j^{j - m} e^{t w_j}}{(w_j+1)^{M-y_j  +j -m } }\right) \right].
  \end{multline}
  where the contour for $w_j$ can be taken as $\lvert w_j \rvert = R$ with $R$ large enough. Although \eqref{eq:point_tasep} involves a summation over $\mathbb{Z}^N$, actually the summation is over $\mathbb{Z}^N(1-m)$, since the integral with respect to $dz$ will be zero unless $\ell_1 + \cdots + \ell_N =1-m$. Moreover, only for finitely many $\Lattice$ the $N$-fold integral is nonzero. To see this, we can take $K$ to be a large enough positive integer, and find that if $\min(\ell_1, \dotsc, \ell_N) \leq -K$, then the $N$-fold integral associated to $\Lattice = (\ell_1, \dotsc, \ell_N)$ vanishes. Hence we have
\begin{multline}
  \prob_Y(x_m(t) \geq M) = \frac{(-1)^{(m-1)(N-1)}}{2 \pi i} \oint_0 \frac{dz}{ z^{1 - L (m-1)}} \dashint d w_1 \cdots \dashint d w_N \det [w_j^{-k}]_{j, k =1}^N \\
  \times\prod_{j=1}^N \left[ \left( \sum^{\infty}_{\ell_j = -K + 1} \left( \frac{z^L}{ w_j^N (w_j +1)^{L-N}} \right)^{\ell_j} \right) \left(\frac{ w_j^{j - m} e^{t w_j}}{(w_j+1)^{M-y_j  +j -m } }\right) \right].
\end{multline}
Using identity
\begin{equation}
  \sum^{\infty}_{\ell_j = -K} \left( \left( \frac{z^L}{ w_j^N (w_j +1)^{L-N}} \right)^{\ell_j} \right) = \frac{(w^N_j (w_j + 1)^{L - N})^K}{w^N_j (w_j + 1)^{L - N} - z^L},
\end{equation}
we have, similar to \eqref{eq:det_trans_ASEP}, that
\begin{multline}
  \prob_Y(x_m(t) \geq M) = \frac{(-1)^{(m-1)(N-1)}}{2 \pi i} \oint_0 \frac{dz}{ z^{1 - L (m-1)}} \\
  \det \left[ \dashint w^{j - k - m} (1 + w)^{-j - M + y_j + m} \frac{(w^N_j (w_j + 1)^{L - N})^K}{w^N_j (w_j + 1)^{L - N} - z^L} dw \right]^N_{j, k = 1}.
\end{multline}
Then, similar to \eqref{eq:residue_TASEP_trans}, we have
\begin{multline}
  \dashint dw w^{j - k - m} (1 + w)^{-j - M + y_j + m} \frac{(w^N_j (w_j + 1)^{L - N})^K}{w^N_j (w_j + 1)^{L - N} - z^L} = \\
  \sum_{w \in R_z} \frac{w^{j-i+1-m} (1 + w)^{y_j -j-M+m +1} e ^{t w}}{w + \rho},
\end{multline}
and this concludes the proof.
\end{proof}

\paragraph{ASEP on the Line}

\begin{lemma}
  In the limit $L \rightarrow \infty$ with other parameters fixed, the formula for $\mathbb{P}_Y(x_1 (t) \geq M)$ from Theorem~\ref{thm:one_pt_ASEP} is equivalent to the formula in \cite[Theorem 3.1]{Tracy-Widom08} for the ASEP on the line:
  \begin{equation}\label{eq:one_point_infinite}
    \prob_Y(x_1(t) \geq M) = p^{N(N - 1)/2} \dashint_C \frac{d\xi_1}{1 - \xi_1} \dotsi \dashint_C \frac{d\xi_N}{1 - \xi_N} \prod^N_{j = 1} \xi^{M - y_j - 1}_j e^{\epsilon(\xi_j)t} \prod_{1 \leq i < j \leq N} \frac{\xi_j - \xi_i}{p + q\xi_i\xi_j - \xi_i}.
  \end{equation}
\end{lemma}

\begin{proof}
  Starting from \eqref{eq:one_point_distribution}, we write
  \begin{equation} \label{eq:conve_to_line_1pt}
    \prob_Y(x_1(t) \geq M) = \sum_{k=0}^{N-1} \sum_{\mathbb{L} \in \mathbb{Z}^N(k)} P_Y^{\mathbb{L}}(M)
  \end{equation}
  with
  \begin{equation}
    \begin{split}
      P_Y^{\mathbb{L}}(M) = {}& \frac{p^{N(N - 1)/2}}{2\pi i} \oint \frac{C_N(z) dz}{z^{k+1}}\dashint_C \frac{d\xi_1}{1 - \xi_1} \dotsi \dashint_C \frac{d\xi_N}{1 - \xi_N} \prod^N_{j = 1} \xi^{M - y_j - 1}_j e^{\epsilon(\xi_j)t}\\
      &\times  \prod_{1 \leq i < j \leq N} \frac{\xi_j - \xi_i}{p + q\xi_i\xi_j - \xi_i}  \prod^N_{j = 1} \left( \xi^L_j \prod^N_{k = 1} \left( \frac{p + q\xi_k \xi_j - \xi_k}{p + q\xi_k \xi_j - \xi_j} \right) \right)^{\ell_j}.
    \end{split}
  \end{equation}
  We consider the term $P_Y^{\mathbb{L}}(M)$ for two cases:
  \begin{enumerate*}[label=(\roman*)]
  \item \label{enu:case_zero}
    $\mathbb{L} = (0, \dots, 0)$, and 
  \item \label{enu:case_nonzero}
    $\mathbb{L} \neq (0, \dots, 0)$.  
  \end{enumerate*}

  In Case \ref{enu:case_zero} it is easy to see that $P^{\Lattice}_Y(M)$ is the right-hand side of \eqref{eq:one_point_infinite}.

  In Case \ref{enu:case_nonzero}, we show that
  \begin{equation}
    \lim_{L \rightarrow \infty} P_Y^{\mathbb{L}}(M) = 0
  \end{equation} 
  for $\mathbb{L} \neq (0, \dots, 0)$ and the convergence is exponentially fast as $\max(\Lattice) \to \infty$. 
  
  When $\mathbb{L} \neq (0, \dots, 0)$ and $\Lattice \in \intZ^N(k)$ with $0 \leq k \leq N-1$, we have $\max(\Lattice) \geq 1$. Then, by Lemma \ref{lem:ASEP_conv_tech}, we have 
  \begin{equation}
    |P_Y^{\mathbb{L}}(M)| \leq \frac{1}{[\max(\Lattice)L/2]!}
  \end{equation}
  if $L$ is large enough. Thus, not only do we have that $P_Y^{\mathbb{L}}(M) \rightarrow 0$ as $L \rightarrow \infty$, but we also have that the sum of $P_Y^{\mathbb{L}}(M)$ on the right-hand side of \eqref{eq:conve_to_line_1pt}, except for $\Lattice = (0, \dots , 0)$, converges to $0$ absolutely in the limit $L \rightarrow \infty$. Hence the proof is done with the help of Lemma \ref{lem:finiteness_of_large_max(L)}.
\end{proof}

\section{Bethe root formula for ASEP}\label{sec:root_formula}

We give an additional formula for the transitional probability function for the ASEP on a ring, which we call the \emph{Bethe root formula}. In particular, we take the formula from Theorem \ref{thm:main} and derive the Bethe root formula by residue calculations. It turns out that the residues that we consider are located at the solutions of a system of algebraic equations called the \emph{Bethe equations}, which arise in the Bethe ansatz. (See Sec.~6.2 of \cite{Sutherland04} or Ch.~4 of \cite{Gaudin14} for a review on the Bethe ansatz.) In the end, the Bethe root formula will be an integral on a single variable over a finite linear combination of solutions to algebraic functions.

The Bethe equations are the system of equations
\begin{equation}\label{eq:bth}
  \xi_j^L = z^{-1} \prod_{k=1}^N \frac{ p + q \xi_k \xi_j - \xi_j}{p + q \xi_k \xi_j - \xi_k}, \quad j= 1, \dots, N,
\end{equation}
with some complex parameter $z\in \mathbb{C}$. Assuming $\epsilon > 0$ is small enough, we take the solutions to the Bethe equations with each component $\xi_j$ bounded in absolute value by $\epsilon$ and $|z| = \epsilon^{-L}(1 - \sqrt{\epsilon})^L$, and call them \emph{Bethe roots}. Note that the Bethe equations have more solutions, and they are not the Bethe roots considered in our paper. The existence of Bethe roots is guaranteed by part \ref{enu:third_Rouche} of Lemma \ref{lm:bth_zero}, and we have that for a fixed $z$, there are $L^N$ Bethe roots. For some $\delta (\epsilon) \rightarrow 0$ as $\epsilon \rightarrow 0$, we index the roots as follows:
\begin{equation}
  (\xi^{(\beta_1)}_1, \dotsc, \xi^{(\beta_N)}_N), \quad \text{with} \quad \beta_i = 1, \dotsc, L, \quad \text{such that} \quad \lvert \xi^{(k)}_j - e^{2\pi i k/L} z^{-1/L} \rvert \leq \delta(\epsilon) \epsilon^{3/2}.
\end{equation}
We denote for $k, j = 1, \dotsc, N$
\begin{equation}
  \begin{split}
    \gamma_{kj}(\xi_1, \dotsc, \xi_N; z) = {}& \frac{\partial}{\partial \xi_k} \left( 1 - z^{-1} \xi^{-L}_j \prod^N_{l = 1} \frac{p + q \xi_l \xi_j - \xi_j}{p + q \xi_l \xi_j - \xi_l} \right) \\
    = {}& -\frac{p + q \xi^2_j - \xi_j}{(p + q \xi_k \xi_j - \xi_j)(p + q \xi_k \xi_j - \xi_k)} \\
    & + \delta_{kj} \left( L \xi^{-1}_j + \sum^N_{l = 1} \frac{p + q \xi^2_l - \xi_l}{(p + q \xi_l \xi_j - \xi_j)(p + q \xi_l \xi_j - \xi_l)} \right).
  \end{split}
\end{equation}

\begin{rem}\label{r:completeness}
The Bethe equations \eqref{eq:bth} we obtain recover the usual Bethe equations (see equation (3.7) in \cite{Gwa-Spohn92}) when we set $z=1$. Our Bethe equations contain an additional parameter of freedom (i.e.~$z \in \mathbb{C}$) since we introduced a winding number to our model, making the configuration space $\mathcal{X}_{N}(L)$ infinite (see \eqref{eq:ASEP_Weyl_chanmber}). Moreover, our Bethe equations only appear as a consequence of the residue analysis, carried out below, of our transition probability function. In other periodic models without winding number (e.g.~\cite{Gwa-Spohn92}), the Bethe equations appear as a consistency/periodicity condition for the eigenvectors of the corresponding Markov operator. Consequently, the results in this section (or the rest of the paper) don't have any direct implications about the completeness of the Bethe ansatz for the periodic ASEP model (i.e.~diagonalizability of the Markov generator via the Bethe ansatz). The completeness questions is out of the scope of the current work since it requires more delicate results such as classification and linear independence of non-trivial solutions of the Bethe equations.
\end{rem}

The following proposition gives the Bethe root formula for the transition probability function for the ASEP on a ring. 

\begin{proposition}\label{prp:root_formula}
  Let $\mathbb{P}_Y (X;t)$ be the transition probability for the ASEP with $N$ particles on a ring of length $L$ with initial configuration $Y =(y_1, \dots, y_N) \in \confs_N(L)$ at time $t =0$ and configuration $X =(x_1, \dots, x_N) \in \confs_N(L)$ at time $t \in \mathbb{R}_{\geq 0}$. Then, 
  \begin{equation} \label{eq:Gaudin-like}
    \prob_Y(X; t) = \oint \frac{dz}{z} \sum_{\substack{\beta_i = 1, \dotsc, L, \, i = 1, \dotsc, N \\ (\xi_1, \dotsc, \xi_N) = (\xi^{(\beta_1)}_1, \dotsc, \xi^{(\beta_N)}_N)}} \sum_{\sigma \in S_N} A_{\sigma}(\xi_1, \dots, \xi_N)\prod_{j=1}^N \frac{\xi_j^{x_j - y_{\sigma (j) } -1} e^{\epsilon(\xi_j)t}}{\det(\gamma_{kj}(\xi_1, \dotsc, \xi_N; z))^N_{k, j = 1}},
  \end{equation}
with $\epsilon > 0$ a small enough constant and $z \in \mathbb{C}$ on the circle $|z| = \epsilon^{-L}(1 - \sqrt{\epsilon})^L$.
\end{proposition}

\begin{rem}
The determinant in the integrand of \eqref{eq:Gaudin-like} is essentially the Gaudin determinant up to a scalar factor \cite{Gaudin-McCoy-Wu81,Korepin82}.
\end{rem}

We give a proof to Proposition \ref{prp:root_formula} at the end of this section. First, we establish preliminary results. The starting point to obtain the Bethe root formula is the transition probability function from Theorem \ref{thm:main}; it consist of an infinite sum of integrals over the lattice $\mathbb{Z}^N(0)$. In the computations below, we proceed by first performing the sum over $\mathbb{Z}^N(0)$ in Lemma \ref{lem:sum}, and then, evaluating the residues of the resulting integral. Throughout this section, we fix the contour $C = \{ \lvert \xi \rvert = \epsilon \}$.

\begin{lemma} \label{lem:sum}
  Let $\mathbb{P}_Y (X;t)$ be the transition probability for the ASEP with $N$ particles on a ring of length $L$ with initial configuration $Y =(y_1, \dots, y_N) \in \confs_N(L)$ at time $t =0$ and configuration $X =(x_1, \dots, x_N) \in \confs_N(L)$ at time $t \in \mathbb{R}_{\geq 0}$. Then,
  \begin{equation}\label{eq:lim_prob}
    \prob_Y(X; t) = \lim_{K \to +\infty} \prob^{(K)}_Y(X; t)
  \end{equation}
  for 
  \begin{equation}\label{eq:partial_prob}
    \prob^{(K)}_Y(X; t) = \oint \frac{dz}{z} U^{(K)}_Y(X; t; z)
  \end{equation}
  with the contour of $z$, a large positively oriented circle centered at zero of radius $(1 - \sqrt{\epsilon})^{-L} \epsilon^{-L}$ for some small positive $\epsilon \ll 1$, and 
  \begin{equation}\label{eq:UK}
    U^{(K)}_Y(X; t; z) = \dashint_C d\xi_1 \dotsi \dashint_C d\xi_N \sum_{\sigma \in S_N}A_{\sigma} \prod^N_{j = 1} \frac{z^K \xi^{K L + x_j -y_{\sigma(j)} -1}_j e^{\epsilon(\xi_j)t}} {1 - z^{-1} \xi^{-L}_j \prod^N_{k = 1} \frac{p + q \xi_k \xi_j - \xi_j}{p + q \xi_k \xi_j - \xi_k}}
  \end{equation}
  with the contours $C$, positively oriented circles centered around the origin of radius $\epsilon$ for the same small positive $\epsilon \ll 1$.
\end{lemma}

\begin{proof}
  Take the series expansion
  \begin{equation}
    \prod^N_{j = 1} \frac{z^K \xi^{LK}_j}{1 - z^{-1} \xi^{-L}_j \prod^N_{k = 1} \frac{p + q \xi_k \xi_j - \xi_j}{p + q \xi_k \xi_j - \xi_k}} = \sum_{\ell_1 = - \infty}^K \cdots  \sum_{\ell_N =- \infty}^K z^{\ell_1 + \dotsb + \ell_N} D_{(\ell_1, \dots, \ell_N)} (\xi_1, \dotsc, \xi_N)
  \end{equation}
  for $|z| = (1 - \sqrt{\epsilon})^{-L} \epsilon^{-L}$ and $|\xi_i| = \epsilon$. Note that this sum is absolutely convergent since by Lemma \ref{lem:ASEP_conv_tech} 
  \begin{equation}
    \left| z^{\ell_1 + \dotsb + \ell_N} D_{(\ell_1, \dots, \ell_N)} (\xi_1, \dotsc, \xi_N) \right| \leq c(\epsilon)^{\ell_1 + \cdots + \ell_N}
  \end{equation}
  for some positive constant $c(\epsilon)$ that is less than one for $\epsilon$ small enough. We use this identity to write 
  \begin{equation}\label{eq:u_sum}
    U^{(K)}_Y(X; t; z) = \dashint_C d\xi_1 \dotsi \dashint_C d\xi_N \sum_{\Lattice \in \mathbb{Z}^N, \,  \max(\Lattice) \leq K} \sum_{\sigma \in S_N} z^{\Lattice} A_{\sigma}  D_{\Lattice} (\xi_1, \dotsc, \xi_N) \prod^N_{j = 1}  \xi^{-x_j - y_{\sigma(j)-1}}_j e^{\epsilon(\xi_j)t}
  \end{equation}
  with the shorthand $z^{\Lattice}= z^{\ell_1 + \cdots + \ell_N}$ for $\Lattice = (\ell_1, \dots, \ell_N)$. Then, we use \eqref{eq:u_sum} in \eqref{eq:partial_prob} and switch the order of integration in \eqref{eq:partial_prob} so that we take the integral with respect to $z$ before any other integral or summation. Note that the integral with respect to $z$ will vanish unless $\ell_1 + \cdots + \ell_N = 0$ for $\Lattice = (\ell_1, \dots, \ell_N)$. Thus, we have that 
  \begin{equation}
    \prob^{(K)}_Y(X; t) = \sum_{\Lattice \in \mathbb{Z}^N(0), \, \max(\Lattice) \leq K } u_{Y}^{\Lattice}(X; t)
  \end{equation}
  with $u_{Y}^{\Lattice}(X; t)$ given by \eqref{eq:transition_probability}. Therefore, by Theorem \ref{thm:main}, the limit \eqref{eq:lim_prob} follows.
\end{proof}

We obtain the probability formula in Proposition \ref{prp:root_formula} by simplifying the formula in Lemma \ref{lem:sum}. In particular, we compute the nested contour integrals for the function $U^{(K)}_Y(X; t; z)$, given in \eqref{eq:UK}, by residue calculation. So, in the next lemma, we find the poles of the integrand in $U^{(K)}_Y(X; t; z)$ that are inside the contours $C$.

\begin{lemma} \label{lm:bth_zero}
  Let $\epsilon > 0$ be a small enough constant, and $j \in \{ 1, \dots, N\}$. Assume  $|\xi_k| \leq \epsilon $ for all $k \in \{ 1, \dots, N\} \setminus \{ j\}$ and $|z| = (1 - \sqrt{\epsilon})^{-L} \epsilon^{-L} $ for some small enough $\epsilon >0$. Then:
  \begin{enumerate}
  \item
    if $j \neq i$, the equation
    \begin{equation} \label{eq:equations_for_Bethe_roots1}
      1 - z^{-1} \xi^{-L}_i \prod^N_{k = 1} \frac{p + q \xi_k \xi_i - \xi_j}{p + q \xi_k \xi_i - \xi_k} = 0
    \end{equation}
    has no solution with respect to $\xi_j$ so that $|\xi_j| \leq \epsilon$;
  \item
    if $j \neq i$, the equation
    \begin{equation}
      L \xi^{-1}_i + \sum_{k \neq i} \frac{p + q \xi^2_k - \xi_k}{(p + q \xi_k \xi_i - \xi_j)(p + q \xi_k \xi_i - \xi_k)} = 0
    \end{equation}
    has no solution with respect to $\xi_j$ with $|\xi_j| \leq \epsilon$; and
  \item \label{enu:third_Rouche}
    the equation
    \begin{equation} \label{eq:equations_for_Bethe_roots2}
      1 - z^{-1} \xi^{-L}_j \prod^N_{k = 1} \frac{p + q \xi_k \xi_j - \xi_j}{p + q \xi_k \xi_j - \xi_k} = 0,
    \end{equation}
    has $L$ solutions with respect to $\xi_j$ with $|\xi_j| \leq \epsilon$ so that the solutions $\xi^{(k)}_j$ ($k = 1, \dotsc, L$) satisfy
    \begin{equation} \label{eq:approx_Bethe_roots}
      \lvert \xi^{(k)}_j - e^{2\pi i k/L} z^{-1/L} \rvert \leq \delta(\epsilon) \epsilon^{3/2},
    \end{equation}
    where $z^{-1/L}$ takes the principal branch with $\arg(z) \in (-\pi, \pi]$, and $\delta(\epsilon) \to 0$ as $\epsilon \to 0$.
  \end{enumerate}
\end{lemma}
\begin{proof}
  All three statements are based on Rouch\'{e}'s theorem. It states that two holomorphic functions $f$ and $g$ (on a closed region $R$) have the same number of zeros in the interior of $R$ if $|f(z) - g(z)| <|f(z)| + |g(z)|$ for all $z \in \partial R$. Since all the statements follow similar arguments, we only discuss part \ref{enu:third_Rouche} in more detail and the reader may check the other cases.

  First, we show that \eqref{eq:equations_for_Bethe_roots2} has exactly $L$ solutions with respect to $\xi_j$ so that $|\xi_j| \leq \epsilon$. Take $f(\xi_j) = \xi_j^L$ and $g(\xi_j) =  \xi_j^L- z^{-1} \prod^N_{k = 1} \frac{p + q \xi_k \xi_j - \xi_j}{p + q \xi_k \xi_j - \xi_k}$. Then,
  \begin{equation}
    \left| f(\xi_j) - g(\xi_j) \right| = \left| z^{-1} \prod^N_{k = 1} \frac{p + q \xi_k \xi_j - \xi_j}{p + q \xi_k \xi_j - \xi_k} \right| = \epsilon^L (1 +\sqrt{\epsilon})^L (1 + \mathcal{O}(\epsilon) ),
  \end{equation}
  and $\left| f(\xi_j) \right| = \epsilon^L$. By Rouch\'{e} theorem, this means that $f(\xi_j)$ and $g(\xi_j)$ have the same number of zeros inside of the contour $C$ for $\xi_j$. Therefore, \eqref{eq:equations_for_Bethe_roots2} with respect to $\xi_j$ has exactly $L$ solutions so that $|\xi_j| \leq \epsilon$.
  
  Now, we find $L$ distinct solutions to \eqref{eq:equations_for_Bethe_roots2} with respect to $\xi_j$ so that $|\xi_j| \leq \epsilon$. For any $\delta > 0$, let $R_\ell$ be the region enclosed by a circle of radius $\delta \epsilon^{3/2}$ and centered at $e^{2 \pi i \ell /L} z^{-1/L}$ for $\ell =0, \dots, L-1$. Note that $|\xi_j| \leq \epsilon$ for all $\xi_j \in R_{\ell}$. We apply the Rouch\'{e} theorem again. Take $f(\xi_j) = \xi_N^{L} - z^{-1}$ and $g(\xi_j) =  \xi_j^L- z^{-1} \prod^N_{k = 1} \frac{p + q \xi_k \xi_j - \xi_j}{p + q \xi_k \xi_j - \xi_k}$.Then, by a series expansion, we have that
  \begin{equation}
    \left| f(\xi_j) - g(\xi_j) \right| =|z^{-1}| \left|1 -  \prod^N_{k = 1} \frac{p + q \xi_k \xi_j - \xi_j}{p + q \xi_k \xi_j - \xi_k} \right| < c_1(N,L) \epsilon^{L+1} (1- \sqrt{\epsilon})^{L}
  \end{equation}
  for some constant $c_1(N, L)$ independent of $\epsilon$ if $\epsilon$ is small enough. Additionally, 
  \begin{equation}
    |f(\xi_j)| = \prod_{\ell =0}^{L-1} \left|\xi_j - e^{2 \pi i \ell /L} z^{-1/L} \right| \geq c_2(N,L) \delta \epsilon^{L+1/2} (1- \sqrt{\epsilon})^{L-1}
  \end{equation}
  for some constant $c_2(N, L)$ independent of $\delta$ and $\epsilon$ if $\epsilon$ is small enough. Then, by Rouch\'{e} theorem, $f(\xi_j)$ and $g(\xi_j)$ have the same number of roots inside the region $R_\ell$ for $\ell =0, 1, \dots, L-1$, given that $\epsilon$ is small enough. Therefore, the solutions to \eqref{eq:equations_for_Bethe_roots2} with respect to $\xi_j$ so that $|\xi_j| \leq \epsilon$ are approximated by \eqref{eq:approx_Bethe_roots}.
\end{proof}

\begin{proof}[Proof of Proposition \ref{prp:root_formula}]
  We decompose the contour $C$ as $C_1 + C_2 + \dotsb + C_L + C'_1 + C'_2 + \dotsb + C'_L + C_0$, with $C_0 = \{ \lvert \xi \rvert = (1 - 2\sqrt{\epsilon}) \epsilon \}$ a circular contour, $C_k$ a contour with four sides:
  \begin{enumerate}
  \item 
    the arc on $C$ from $\epsilon e^{\arg(z) + 2\pi i k/L - \sqrt{\epsilon}}$ to $\epsilon e^{\arg(z) + 2\pi i k/L + \sqrt{\epsilon}}$,
  \item
    the line segment from $\epsilon e^{\arg(z) + 2\pi i k/L + \sqrt{\epsilon}}$ to $(1 - 2\sqrt{\epsilon}) \epsilon e^{\arg(z) + 2\pi i k/L + \sqrt{\epsilon}}$,
  \item
    the arc on $C_0$ from $(1 - 2\sqrt{\epsilon}) \epsilon e^{\arg(z) + 2\pi i k/L - \sqrt{\epsilon}}$ to $(1 - 2\sqrt{\epsilon}) \epsilon e^{\arg(z) + 2\pi i k/L - \sqrt{\epsilon}}$,
  \item
    and the line segment from $(1 - 2\sqrt{\epsilon}) \epsilon e^{\arg(z) + 2\pi i k/L - \sqrt{\epsilon}}$ to $\epsilon e^{\arg(z) + 2\pi i k/L - \sqrt{\epsilon}}$.
  \end{enumerate}
  and $C'_k$ also a contour with four sides:
  \begin{enumerate}
  \item 
    the arc on $C$ from $\epsilon e^{\arg(z) + 2\pi i k/L + \sqrt{\epsilon}}$ to $\epsilon e^{\arg(z) + 2\pi i (k + 1)/L - \sqrt{\epsilon}}$,
  \item
    the line segment from $\epsilon e^{\arg(z) + 2\pi i (k + 1)/L - \sqrt{\epsilon}}$ to $(1 - 2\sqrt{\epsilon}) \epsilon e^{\arg(z) + 2\pi i (k + 1)/L - \sqrt{\epsilon}}$,
  \item
    the arc on $C_0$ from $(1 - 2\sqrt{\epsilon}) \epsilon e^{\arg(z) + 2\pi i (k + 1)/L - \sqrt{\epsilon}}$ to $(1 - 2\sqrt{\epsilon}) \epsilon e^{\arg(z) + 2\pi i k/L + \sqrt{\epsilon}}$,
  \item
    and the line segment from $(1 - 2\sqrt{\epsilon}) \epsilon e^{\arg(z) + 2\pi i k/L + \sqrt{\epsilon}}$ to $\epsilon e^{\arg(z) + 2\pi i k/L + \sqrt{\epsilon}}$.
  \end{enumerate}
  See Figure \ref{fig:C_Ctilde} for an illustration.
  
  \begin{figure}[htb]
    \centering
    \includegraphics{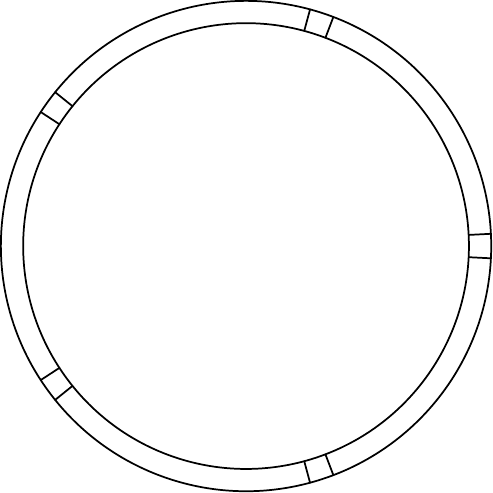}
    \caption{Schematic shape for $C_i$ and $\tilde{C}_i$ for $L = 5$ and $\arg z = 0$.}
    \label{fig:C_Ctilde}
  \end{figure}
  We assume all these contours have a positive orientation. Also, we write $\tilde{C} = C_1 + \dotsb + C_L + C'_1 + \dotsb + C'_L$, which is an annulus.  Moreover, given an array $A = (\alpha_1, \dotsc, \alpha_N)$ whose components are $\alpha_i = 0$ or $1$, we introduce integrals with the different type of contours
  \begin{equation}
    U^{(K)}_Y(X; t; z; A) = \dashint_{C^{(1)}} d\xi_1 \dotsi \dashint_{C^{(N)}} d\xi_N F(\xi_1, \dotsc, \xi_N) \prod^N_{j = 1} \frac{z^K \xi^{LK - \tilde{M}}_j e^{\epsilon(\xi_j)t}} {1 - z^{-1} \xi^{-L}_j \prod^N_{k = 1} \frac{p + q \xi_k \xi_j - \xi_j}{p + q \xi_k \xi_j - \xi_k}} ,
  \end{equation}
with $\tilde{M} = y_1 - x_N + 1$ and
  \begin{equation}
    F(\xi_1, \dotsc, \xi_N) = \sum_{\sigma \in S_N}A_{\sigma} \prod^N_{j = 1} \xi^{x_j - y_{\sigma(j)} - 1 + \tilde{M}}_j,
  \end{equation}
so that 
	\begin{equation}
	C^{(k)} =
      \begin{cases}
         C_0 & \text{if $\alpha_k = 0$}, \\
         \tilde{C} & \text{if $\alpha_k = 1$},
      \end{cases}
	\end{equation}
for $k =1, \dots, N$. Then, we have
  \begin{equation}
    U^{(K)}_Y(X; t; z) = \sum_{A \in \{ 0, 1 \}^N} U^{(K)}_Y(X; t; z; A).
  \end{equation}
For any $A \neq (1, 1, \dotsc, 1)$, take a component, say $\alpha_j$, equal to $0$. If we integrate $\xi_j$ over $C_0$ first, the result is an $N - 1$ fold integral over $C^{(1)} \times \dotsb \times C^{(j - 1)} \times C^{(j + 1)} \times \dotsb \times C^{(N)}$. Since $0$ is the only pole within $C_0$ for $\xi_j$, we integrate $\xi_j$ by the residue theorem, and find that the integrand of the $N - 1$ fold integral is bounded by
  \begin{equation}
    c_1(1 - \sqrt{\epsilon})^{-LNK} \epsilon^{- LNK} e^{\epsilon^{-1} t(N - 1)} \frac{c_2^{L(K + 1) - \tilde{M}}}{(L(K + 1) - \tilde{M})!},
  \end{equation}
with a constant $c_1$ and $c_2$ that don't depend on $K$ but may depend on $\epsilon$ and $F$. It is clear that $U^{(K)}_Y(X; t; z) \to 0$ as $K \to \infty$, and then
  \begin{equation}
    \lim_{K \to \infty} U^{(K)}_Y(X; t; z) - U^{(K)}_Y(X; t; z; \vec{1}) = 0, \quad \vec{1} = (1, 1, \dotsc, 1).
  \end{equation}
  Hence we only need to evaluate $U^{(K)}_Y(X; t; z; \vec{1})$.
  
  Next, we denote $B = (\beta_1, \dotsc, \beta_N)$ as an array whose components are $\beta_i \in \{ 1, \dotsc, L \} \cup \{ -1, \dotsc, -L \}$, and have
  \begin{equation}
    U^{(K)}_Y(X; t; z; \vec{1}) = \sum_{B \in \{ -L, \dotsc, -1, 1, \dotsc, L \}^N} \tilde{U}^{(K)}_Y(X; t; z; B),
  \end{equation}
  with
  \begin{equation} \label{eq:formula_tilde_U^(K)_Y(XtzB)}
    \tilde{U}^{(K)}_Y(X; t; z; B) =  \dashint_{\hat{C}^{(1)}} d\xi_1 \dotsi \dashint_{\hat{C}^{(N)}} d\xi_N F(\xi_1, \dotsc, \xi_N) \prod^N_{j = 1} \frac{z^K \xi^{LK - \tilde{M}}_j e^{\epsilon(\xi_j)t}} {1 - z^{-1} \xi^{-L}_j \prod^N_{k = 1} \frac{p + q \xi_k \xi_j - \xi_j}{p + q \xi_k \xi_j - \xi_k}} 
  \end{equation}
so that 
	\begin{equation}
	\hat{C}^{(k)} =
      \begin{cases}
         C_{\beta_k} & \text{if $\beta_k > 0$}, \\
         C'_{-\beta_k} & \text{if $\beta_k < 0$}.
      \end{cases}
	\end{equation}
  If $B$ contains a negative component, say $\beta_j < 0$, we integrate $\xi_j$ over $C'_{-\beta_j}$ first when evaluating $\tilde{U}^{(K)}_Y(X; t; z; B)$, and find that the integral vanishes since $\xi_j$ has no pole within $C'_{-\beta_j}$. Hence, $\tilde{U}^{(K)}_Y(X; t; z; B)$ is nontrivial only for $B$ with all positive components.

  Now, suppose $B = (\beta_1, \dotsc, \beta_N)$ has all positive components. Then, $\tilde{U}^{(K)}_Y(X; t; z; B)$ is an $N$-fold contour integral over $C_{\beta_1} \times \dotsb \times C_{\beta_N}$. Also, note that the Bethe roots $(\xi^{(\beta_1)}_1, \dotsc, \xi^{(\beta_N)}_N)$ are labeled so that $\xi^{(\beta_k)}_k$ lies inside the contour $C_{\beta_k}$. Hence, near the Bethe root $(\xi^{(\beta_1)}_1, \dotsc, \xi^{(\beta_N)}_N)$, the denominator in the integrand of \eqref{eq:formula_tilde_U^(K)_Y(XtzB)} is 
  \begin{equation}
    \prod^N_{j = 1} \left( \sum^N_{k = 1} \gamma_{kj}(\xi^{(\beta_1)}_1, \dotsc, \xi^{(\beta_N)}_N; z) (\xi_k - \xi^{(\beta_k)}_k) \right) + \bigO \left( \max^N_{k = 1} \lvert \xi_k \rvert^2 \right).
  \end{equation}
  Then, we evaluate $\tilde{U}^{(K)}_Y(X; t; z; B)$ by shrinking the contours $C_{\beta_k}$ into very small circles around $\xi^{(\beta_k)}_k$. Note that if all the $N$ contours shrink at the same scale, the contours do not cross any singularity and so the integral is constant. As the contours approach the singularity $\xi^{(\beta_k)}_k$, we find that the integral \eqref{eq:formula_tilde_U^(K)_Y(XtzB)} can be computed as
  \begin{equation}
    \frac{1}{\det(\gamma_{kj}(\xi_1, \dotsc, \xi_N))^N_{k, j = 1}} \left. F(\xi_1, \dotsc, \xi_N) \prod^N_{k = 1} (\xi_k)^{\tilde{M}} e^{\epsilon(\xi_k) t} \right\rvert_{(\xi_1, \dotsc, \xi_N) = (\xi^{(\beta_1)}_1, \dotsc, \xi^{(\beta_N)}_N)}.
  \end{equation}
  Summing up all these $\tilde{U}^{(K)}_Y(X; t; z; B)$ and taking the limit $K \to \infty$, we derive the result.
\end{proof}

\appendix

\section{Convergence lemmas for ASEP and $q$-TAZRP on a ring} \label{sec:convergence}

We need to make sure that the right side of \eqref{eq:transition_probability} and \eqref{eq:defn_u_Y(X)} are well-defined before we prove Theorems \ref{thm:main} and \ref{thm:trans_prob}. That is, we need to show that the sum of $\Lambda^{\Lattice}_Y(X; t; \sigma)$ over the infinite lattice $\intZ^N(0)$ converges. Also, we need to show that the integrands of the $z$-integral on the right side of \eqref{eq:one_point_distribution} and \eqref{eq:one_point_distribution_2} converge in Theorems \ref{thm:one_pt_ASEP} and \ref{thm:one_pt_qTAZRP}. For these integrands, the problem is more serious since we need to deal with convergence of other series in the proofs of Theorems \ref{thm:one_pt_ASEP} and \ref{thm:one_pt_qTAZRP}. In this appendix, we show that all the required convergences hold.

We establish some notation and a basic estimate to facilitate the upcoming proofs in this section. Recall the notation set up in \eqref{eq:defn_mM}. It is straightforward to see that given any $k$ and $n$, there are only finitely many $\Lattice \in \intZ^N(n)$ such that $\max(\Lattice) < k$. We have a very rough estimate.
\begin{lemma} \label{lem:finiteness_of_large_max(L)}
The number of $\Lattice \in \intZ^N(n)$ with $\max(\Lattice) < k$ is no more than $N^N (k - \lfloor n/N \rfloor)^N$ if $n < Nk$, and there is no $\Lattice \in \intZ^N(n)$ with $\max(\Lattice) < k$ if $n \geq Nk$.
\end{lemma}
\begin{proof}
  The statement for $n \geq Nk$ is obvious. For $n < Nk$, we assume that $n = lN$ with $l$ an integer without loss of generality. Note that, if $\Lattice = (\ell_1, \dotsc, \ell_N) \in \intZ^N(n)$, then $\Lattice' (\ell_1 - l, \ell_2 - l, \dotsc, \ell_N - l) \in \intZ^N(0)$. Then, the inequality $\max(\Lattice) < k$ is equivalent to $\max(\Lattice') < k - l$. Hence, the lemma is reduced to the $n = 0$ case.

  Suppose $n = 0$ and $\Lattice = (\ell_1, \dotsc, \ell_N) \in \intZ^N(0)$ satisfies $\max(\Lattice) < k$. Then, if there is an $\ell_i$ with $\ell_i < (1 - N)k$, the sum of the other $(N - 1)$ components is greater than $(N - 1)k$, and at least one among them is greater than $k$, a contradiction. So, the value of each $\ell_i$ is less than $k$ and no less than $(1 - N)k$. Hence, each component $\ell_i$ has no more than $Nk$ possible values and $\Lattice$ has no more than $(Nk)^N$ possible choices.
\end{proof}

\subsection{$q$-TAZRP case}

In this subsection, we assume $C = \{ \lvert z \rvert = R \}$ is a circular contour with positive orientation and $R$ a large constant. Recall that $b_{[1]}, \dotsb, b_{[L]}$ are positive constants defined in Section \ref{subsubsec:q_TAZRP_defn}. We assume $g(w_1, \dotsc, w_N)$ is an analytic function on $\compC$ and 
\begin{equation} \label{eq:defn_tilde_g}
  \tilde{g}(w_1, \dotsc, w_N) = g(w_1, \dotsc, w_N) \prod^N_{j = 1} \prod^L_{i = 1} (b_{[i]} - w_j)^{-m}
\end{equation}
for some $m > 0$. In this subsection, we take $D_{\Lattice}$ as defined by \eqref{eq:general_D_L_qTAZRP}. 

\begin{lemma} \label{lem:qTAZRP_conv}
  Denote
  \begin{equation}
    \Lambda^{\Lattice} = \dashint_C dw_1 \dotsi \dashint_C dw_N g(w_1, \dotsc, w_N) D_{\Lattice}(w_1, \dotsc, w_N) \prod_{1 \leq i < j \leq N} (qw_i - w_j)^{-1}.
  \end{equation}
  Then $\Lambda^{\Lattice} = 0$ if $\max(\Lattice) > m$.
\end{lemma}
\begin{proof}
  Consider the integral over $w_{m(\Lattice)}$. We have that the contour $C$ encircles no poles with respect to $w_{m(\Lattice)}$. Then, the integral of $w_{m(\Lattice)}$ over $C_{m(\Lattice)}$ vanishes, which implies that $\Lambda^{\Lattice}$ vanishes.
\end{proof}

Next, we consider the special case with $b_{[1]} = \dotsb = b_{[N]} = 1$. Then, the function $\tilde{g}(w_1, \dotsc, w_N)$ defined in \eqref{eq:defn_tilde_g} becomes $g(w_1, \dotsc, w_N)  \prod^N_{j = 1} (1 - w_j)^{-mL} $. Furthermore, we assume that the radius $R$ of the contour $C$ is bigger than $N$. We denote
\begin{equation}
  M_g = \max_{w_i \in C,\ i = 1, \dotsc, N} \lvert g(w_1, \dotsc, w_N) \rvert.
\end{equation}
\begin{lemma} \label{lem:estimate_qTAZRP}
  Denote
  \begin{equation} \label{eq:int_tilde_Lambda_001}
    \tilde{\Lambda}^{\Lattice} = \dashint_C \frac{dw_1}{w_1} \dotsi \dashint_C \frac{dw_N}{w_N} \tilde{g}(w_1, \dotsc, w_N) D_{\Lattice}(w_1, \dotsc, w_N) \prod_{1 \leq i < j \leq N} (qw_i - w_j)^{-1},
  \end{equation}
  and $\tilde{\Lambda}_s = \sum_{\Lattice \in \intZ^N(s)} \lvert \tilde{\Lambda}^{\Lattice} \rvert$ for some $s \in \intZ$. Furthermore, assume the radius $R$ of the contour $C$ satisfies
  \begin{equation}
    \frac{R^{2LN}}{(R - 1)^{L(2N - 1)}} \left( \frac{1 + q}{1 - q} \right)^{2N^2} < 1.
  \end{equation}
  Then, $\tilde{\Lambda}_s < C M_g$ for some constant $C$ that depends on $m, s$ but not $g$.
\end{lemma}
\begin{proof}
  We partition $\intZ^N(s)$ into disjoint subsets $Z_0 \cup Z_1 \cup Z_2 \cup \dotsb$, with 
	\begin{equation}
	\begin{split}
	Z_0 &= \{ \Lattice \in \intZ^N(s) \mid L \max(\Lattice) < m \}\\
	 Z_n &= \{ \Lattice \in \intZ^N(s) \mid m + (n - 1)L \leq L\max(\Lattice) < m + nL \}
	\end{split}
	\end{equation}
 for $n = 1, 2, \dotsc$. By Lemma \ref{lem:finiteness_of_large_max(L)}, each $Z_n$ is a finite set and $\lvert Z_n \rvert < N^N (n + \lceil m/L \rceil - \lfloor s/N \rfloor)^N$ if $n + m/L - s/N > 0$. Without loss of generality, we assume $m = s = 0$ below for notational convenience. Then, in this case, $Z_0 = \emptyset$ and we only need to consider $n \geq 1$.

  For $\Lattice = (\ell_1, \dotsc, \ell_N) \in Z_n$, we have that the integrand in \eqref{eq:int_tilde_Lambda_001} has no pole with respect to $w_{m(\Lattice)}$ within $C$ and has a zero of order at least $(n-1)L$ at $1$. So, by the identity
  \begin{equation}
    \frac{1}{w} = \frac{1}{w(1 - w)^{(n-1)L}} - \sum^{(n-1)L}_{k = 1} \frac{1}{(1 - w)^k},
  \end{equation}
  we have
  \begin{equation} \label{eq:int_tilde_Lambda}
    \tilde{\Lambda}^{\Lattice} = \dashint_C \frac{dw_1}{w_1} \dotsi \dashint_C \frac{dw_N}{w_N} \frac{ g(w_1, \dotsc, w_N)}{(1 - w_{m(\Lattice)})^{(n-1)L}} D_{\Lattice}(w_1, \dotsc, w_N) \prod_{1 \leq i < j \leq N} (qw_i - w_j)^{-1}.
  \end{equation}

  Also, we have $\sum^N_{i = 1} \lvert \ell_i \rvert < 2nN$ since $\max^N_{i = 1} (\ell_i) < n$ and $\sum^N_{i = 1} \ell_i = 0$. Then, it is not hard to see that
  \begin{equation}
    \lvert D_{\Lattice}(w_1, \dotsc, w_N) \rvert < \left( \left( \frac{R}{R - 1} \right)^L \left( \frac{1 + q}{1 - q} \right)^N \right)^{2nN},
  \end{equation}
if $w_i \in C$ for all $i = 1, \dotsc, N$.  Additionally, we have 
  \begin{equation}
    \left\lvert \frac{ g(w_1, \dotsc, w_N)}{(1 - w_{m(\Lattice)})^{(n-1)L}} \prod_{1 \leq i < j \leq N} (qw_i - w_j)^{-1} \right\rvert < \frac{M_g}{(R - 1)^{(n-1)L}} ((1 + q)R)^{N(N - 1)/2}.
  \end{equation}
  Thus, we obtain the following estimate,
  \begin{equation}
    \sum_{\Lattice \in Z_n} \lvert \tilde{\Lambda}^{\Lattice} \rvert <  (N n)^N M_g (R-1) ((1 + q)R)^{N(N - 1)/2} \left( \frac{R^{2N L}}{(R - 1)^{(2N+1) L}} \left( \frac{1 + q}{1 - q} \right)^{2N^2} \right)^n,
  \end{equation}
  by combining the previous two estimates. Hence, we obtain the result by summing the last estimate over $n>1$.
\end{proof}

\subsection{ASEP case}

Later in this subsection, we assume that $C = \{ \lvert z \rvert = p^2 \}$ is a circular contour with positive orientation. Let $f(\xi_1, \dotsc, \xi_N)$ be a meromorphic function such that it is analytic on the multi-cylinder $\{ \lvert \xi_i \rvert \leq p^2 \mid i = 1, \dotsc, N \}$ with $\max_{\lvert \xi_i \rvert = r,\ i = 1, \dotsc, N} \lvert f(\xi_1, \dotsc, \xi_N) \rvert = M_f$. Let $t \in \realR_{>0}$, $m \in \intZ_{>0}$ and $s \in \intZ$ be constants. Then, we have the following estimate.
\begin{lemma}\label{lm:absolutely_convergent}
  Denote
  \begin{equation}
    \Lambda^{\Lattice} = \dashint_C d\xi_1 \dotsi \dashint_C d\xi_N f(\xi_1, \dotsc, \xi_N) D_{\Lattice}(\xi_1, \dotsc, \xi_N) \prod^N_{j = 1} \xi^{-m}_j e^{-\epsilon(\xi_j) t}, \quad \text{and} \quad  \Lambda_{s} = \sum_{\Lattice \in \intZ^N(s)} \lvert \Lambda^{\Lattice} \rvert.
  \end{equation}
  We have $\Lambda_s < C_1 M_f$ for a constant $C_1$ that depends on $m, s$ but not on $f(\xi_1, \dotsc, \xi_N)$.
\end{lemma}
For the proof of Lemma \ref{lm:absolutely_convergent}, we need to estimate $\Lambda^{\Lattice}$ for each $\Lattice$. We decompose the exponential factor $e^{\epsilon(\xi_{m(\Lattice)})t}$ into the sum of two terms: one is the partial sum of the Taylor expansion of the exponential function, and the other is the remainder term. To be precise, we write
\begin{equation}
  e^z = P_n(z) + E_n(z), \quad \text{so that} \quad P_n(z) = \sum^n_{k = 0} \frac{z^k}{k!}, \quad E_n(z) = \sum^{\infty}_{k = n + 1} \frac{z^k}{k!}.
\end{equation}
Then, we write 
\begin{equation}
  \Lambda^{\Lattice} = \Lambda^{\Lattice; P}(n) + \Lambda^{\Lattice; E}(n),
\end{equation}
for any $n \geq 0$ with 
\begin{multline} \label{eq:int_rep_square}
  \Lambda^{\Lattice; \square}(n) = \dashint_C d\xi_1 \dotsi \dashint_C d\xi_N f(\xi_1, \dotsc, \xi_N) D_{\Lattice}(\xi_1, \dotsc, \xi_N) \\
  \times \left( \prod^N_{j = 1} \xi_{\sigma(j)}^{-m} \prod_{j = 1, \dotsc, N\, j \neq m(\Lattice)} e^{\epsilon(\xi_j)t} \right) \square_n(tp \xi^{-1}_{m(\Lattice)}) e^{t(q \xi_{m(\Lattice)} - 1)}
\end{multline}
for $\square = P$ or $E$.
\begin{lemma} \label{lem:ASEP_conv_tech}
  Suppose $\max(\Lattice)L - m = M \geq 0$. Then
  \begin{enumerate}
  \item \label{enu:lem:ASEP_conv_tech_a}
    $\Lambda^{\Lattice; P}(M) = 0$.
  \item \label{enu:lem:ASEP_conv_tech_b}
    If $\Lattice \in \intZ^N(s)$, we have $\lvert \Lambda^{\Lattice; E}(M) \rvert < C_2 M_f e^{c_3 M + c'_3 L} /M!$ for some constant $C_2, c_3, c'_3 > 0$ depending on $s$, $t$, $m$ but independent of $f$ and $L$. 
  \end{enumerate}
\end{lemma}
\begin{proof}
  The proof of statement \ref{enu:lem:ASEP_conv_tech_a} is similar to the proof of Lemma \ref{lem:qTAZRP_conv}. We first note that the contour $C$ encloses no pole with respect to $\xi_{m(\Lattice)}$. In particular, there is no pole at $\xi_{m(\Lattice)} =0$ due to the assumption that $\max(\Lattice) = \ell_{m(\Lattice)} \geq m$. Hence, the integral with respect to $\xi_{m(\Lattice)}$ over $C_{m(\Lattice)}$ vanishes for the multiple integral that defines $\Lambda^{\Lattice; P}_Y(X; t; \sigma; M)$, and so does $\Lambda^{\Lattice; P}_Y(X; t; \sigma; M)$.

  For statement \ref{enu:lem:ASEP_conv_tech_b}, we note that $\lvert \xi_i \rvert = p^2$ if $\xi_i$ is on the contour $C$. Then,
  \begin{equation}
    \left\lvert \prod^N_{j = 1} \xi^{-m}_j \prod_{j \neq m(\mathbb{L})} \xi^{L \ell_j}_j \right\rvert = p^{2(-mN + L(s - m(\mathbb{L})L))}
  \end{equation}
for all $\Lattice \in \intZ^N(s)$. Also, we have
  \begin{equation}
    \left\lvert e^{t(q \xi_{m(\Lattice)} - 1)} \prod_{j = 1, \dotsc, N, \, j \neq m(\Lattice)} e^{\epsilon(\xi_j)t} \right\rvert \leq e^{N c_4 t}.
  \end{equation}
with $c_4 = (1 + p^2)(p + p^{-2}q)$ since $\Re \epsilon(\xi_j) \leq c_1$ and $\Re(q \xi_j - 1) \leq c_4$ if $\lvert \xi_j \rvert = p^2$. Next, we have that
  \begin{equation}
    \prod^N_{j = 1} \prod^N_{k = 1} \left\lvert \frac{p + q\xi_k \xi_j - \xi_k}{p + q\xi_k \xi_j - \xi_j} \right\rvert^{\ell_j} \leq \prod^N_{j = 1} c^{\lvert \ell_j \rvert}_5 \leq c^{N \max(\Lattice)}_5 = c^{N(M + m)}_5.
  \end{equation}
with $c_5 = (1 + p + p^3q)/(q - p^3q)$ for all $i = 1, \dotsc, N$ since
  \begin{equation}
    p(q - p^3q) \leq \lvert p + q\xi_i \xi_j - \xi_i \rvert \leq p(1 + p + p^3q).
  \end{equation}
At last, we have an adequate upper bound of the integrand in \eqref{eq:int_rep_square} for $\Lambda^{\Lattice; E}_Y(X; t; \sigma; M)$ by using the estimate
  \begin{equation}
    \lvert E_M(tp \xi^{-1}_{m(\Lattice)}) \rvert \leq \frac{2}{(M + 1)!} (tp^{-1})^{M + 1},
  \end{equation}
  which corresponds to the estimate of the error term in a Taylor expansion. Then, we have
  \begin{equation}
    \lvert \Lambda^{\Lattice; E}(M) \rvert \leq M_f p^{2(-m(N+1) + Ls)} c^{N(M + m)}_5 \frac{2}{(M + 1)!} (tp^{-1})^{M + 2} e^{N c_4 t}
  \end{equation}
after evaluating the integral. This implies statement \ref{enu:lem:ASEP_conv_tech_b} of the lemma.
\end{proof}

\end{document}